\documentclass[11pt,a4paper]{amsart}
\usepackage[utf8]{inputenc}
\usepackage[T1]{fontenc}
\usepackage{amsmath, amsthm, amssymb, amsfonts, mathtools, graphicx, multicol, array, enumerate, bbm, bm, comment, tikz-cd, hyperref, blindtext, subfiles}
\usepackage[dvipsnames]{xcolor}
\usepackage[margin=1in]{geometry}
\usepackage{hyperref}
\usetikzlibrary[decorations.pathmorphing]

\newtheorem{theorem}{Theorem}[section]
\newtheorem{definition}[theorem]{Definition}
\newtheorem{exmp}[theorem]{Example}
\newtheorem{lemma}[theorem]{Lemma}
\newtheorem{proposition}[theorem]{Proposition}
\newtheorem{corollary}[theorem]{Corollary}
\newtheorem{rem}[theorem]{Remark}

\newenvironment{example}{\begin{exmp}\rm}{\end{exmp}}
\newenvironment{remark}{\begin{rem}\rm}{\end{rem}\rm}

\DeclareMathOperator{\Spec}{\mathrm{Spec}}

\DeclareMathOperator{\ch}{\mathrm{ch}}

\DeclareMathOperator{\Aut}{\mathrm{Aut}}
\DeclareMathOperator{\Coh}{\mathrm{Coh}}

\DeclareMathOperator{\Hom}{\mathrm{Hom}}
\DeclareMathOperator{\Ext}{\mathrm{Ext}}
\DeclareMathOperator{\ext}{\mathrm{ext}}
\DeclareMathOperator{\Syst}{\mathrm{Syst}}
\DeclareMathOperator{\SSyst}{\mathcal{S}\mathrm{yst}}
\DeclareMathOperator{\Q}{\mathbb{Q}}
\DeclareMathOperator{\R}{\mathbb{R}}

\DeclareMathOperator{\C}{\mathbb{C}}
\DeclareMathOperator{\Z}{\mathbb{Z}}

\DeclareMathOperator{\Pic}{\mathrm{Pic}}
\DeclareMathOperator{\rk}{\mathrm{rk}}

\DeclareMathOperator{\Gr}{\mathrm{Gr}}

\DeclareMathOperator{\ev}{\mathrm{ev}}
\DeclareMathOperator{\coker}{\mathrm{coker}}
\DeclareMathOperator{\supp}{\mathrm{supp}}

\DeclareMathOperator{\Spl}{\mathrm{Spl}}

\renewcommand{\P}{\mathbb{P}}
\renewcommand{\O}{\mathcal{O}}

\renewcommand{\L}{\left}
\renewcommand{\R}{\right}
\renewcommand{\tilde}{\widetilde}

\newcommand{\nocontentsline}[3]{}
\let\origcontentsline\addcontentsline
\newcommand\stoptoc{\let\addcontentsline\nocontentsline}
\newcommand\resumetoc{\let\addcontentsline\origcontentsline}

\author{Samir Geiger and Andreas Kretschmer}

\title{On the Brill--Noether Theory of Sheaves on Regular Surfaces}

\begin{document}
\begin{abstract}
    We study the Brill--Noether theory of higher rank sheaves on surfaces~$X$ with $H^1(X,\O_X) = 0$ and its interactions with the classical Brill--Noether theory of line bundles on curves lying on~$X$. Denoting by $\Spl(v)$ the moduli space of simple sheaves on $X$ with Chern character $v = (r,c_1,\ch_2)$, we introduce the Brill--Noether loci $BN^k(v)^\circ \subseteq \Spl(v)$ of \emph{generically $k$-generated} sheaves and show, under assumptions on $c_1$ and $K_X$, certain well-behavedness results. In particular, if $c_1$ is \emph{indecomposable} and $X$ is a K3 surface, we generalize a Brill--Noether theorem for sheaves on K3s due to Yoshioka. This is achieved by employing He's deformation theory of the moduli space of coherent systems and a globalization of the Lazarsfeld--Mukai construction, providing a correspondence between the Brill--Noether theories of generically $r$-generated sheaves on $X$ and that of line bundles on curves $C \in |c_1|$. The same techniques yield upper bounds on the dimension of $\mathcal{W}^{r,\mathrm{bpf}}_d(|c_1|) \setminus \mathcal{W}^{r+1}_d(|c_1|)$ consisting of basepoint-free line bundles in terms of the number of endomorphisms of associated Lazarsfeld--Mukai bundles, for arbitrary $c_1$. This generalizes a result of Aprodu and Farkas in the context of Green's conjecture beyond K3 surfaces. In the case $r = 1$ of pencils, we obtain smoothness and expected-dimension results under an effectivity condition on $-K_X|_C$ for $C \in |c_1|$. As an application, we prove Green's conjecture for the general curve on a rational surface with a nef and effective anticanonical divisor.
\end{abstract}
\maketitle

\section{Introduction}
\noindent The classical Brill--Noether theorem asserts that, for a general curve $C$ of genus $g$, the Brill--Noether loci
$$
    W^r_d(C) \coloneqq \{L\in \Pic^d(C)\colon h^0(C,L)\geq r+1\}
$$
are irreducible of the expected dimension $\rho(g,r,d)= g-(r+1)(g-d+r)$ whenever this number is positive, non-empty and $0$-dimensional when it is zero, and empty otherwise, see, e.g., \cite{HarrisBNSurvey} and the references therein. It is natural to ask for generalizations of this statement to Brill--Noether loci of higher rank sheaves on surfaces.
Let $\mathrm{Spl}(v)$ and $M(v)$, respectively, be the moduli spaces of simple and slope-stable sheaves (with respect to a chosen polarization) on a smooth projective surface $X$ over $\C$ with Chern character $v=(r,c_1,\ch_2)$.
Define the Brill--Noether loci
$$
    BN^k(v) \coloneqq \{E\in \Spl(v)\colon h^0(X,E)\geq k\}
$$
and the locally closed subloci
$$
    \widetilde{BN}^k(v) \coloneqq \{E\in \mathrm{Spl}(v) \colon h^0(X,E)=k\}
$$
as subsets of $\mathrm{Spl}(v)$ or $M(v)$.
For K3 surfaces $X$, Yoshioka \cite{YoshiokaBNTHeory} and Leyenson \cite{Leyenson} proved that these Brill--Noether loci, in $M(v)$, have the expected dimension whenever $c_1$ satisfies a strong numerical minimality assumption with respect to the polarization (it is, e.g., sufficient that $\Pic(X) = \Z \cdot c_1$). For Hirzebruch surfaces, Costa--Miró-Roig \cite{CostaHirzebruch} give conditions for the existence of an irreducible component of the expected dimension for $\ch_2 \ll 0$. More recently, for $\P^2$, Gould--Liu--Lee \cite{GouldLiuLee} give criteria for non-emptiness and (ir)reducibility as well as results on the dimensions of irreducible components for many Brill--Noether loci in $M(v)$ under various conditions on $v$. For a recent general survey on the Brill--Noether theory for sheaves on surfaces, we refer to \cite{CoskunSurvey}.

Our approach is to work instead on the moduli space of simple sheaves $\Spl(v)$ and unify some of the techniques in the literature, going back to Lazarsfeld \cite{Lazarsfeld}, Donagi--Morrison \cite{DonagiMorrison}, Le Potier \cite{LePotier}, He \cite{He}, Yoshioka \cite{YoshiokaMukaiReflections,YoshiokaBNTHeory}, Aprodu--Farkas \cite{FarkasAproduGreenConj} and others, to obtain Brill--Noether-type results for a larger class of surfaces and under weaker assumptions on the Chern character. In particular, we make no assumptions on the Picard rank.
We identify the notion of \emph{generically $k$-generated sheaves} (Definition \ref{def:genericallygenerated}) to be well-suited for generalizing several classical techniques to a broader class of surfaces and Chern characters, and without the necessity of choosing a stability condition.
Denoting by $BN^k(v)^{\circ}$ the locus of all generically $k$-generated sheaves in $\Spl(v)$, we prove:

\begin{theorem}[Theorem \ref{BNconnected}]
\label{intro:thm1}
    Let $X$ be a smooth projective surface with $H^1(X,\O_X) = 0$ and let $v = (r, c_1, \ch_2)$, $c_1 \in \Pic(X)$, be a Chern character with $|c_1| \neq \emptyset$.
    \begin{enumerate}[(i)]
        \item Assume that either $K_X = 0$ or $-K_X.c_1 \geq \max(p_a(C),1)$ for $C \in |c_1|$. Then the open sublocus of $\widetilde{BN}^r(v)^{\circ}$ consisting of sheaves $E$ with $\mathrm{coker}(H^0(E) \otimes \O_X \rightarrow E)$ supported on an integral curve $C \in |c_1|$ is smooth of the expected dimension, or empty.
        \item If $0 \leq k < r$, assume $-K_X$ effective. Then the open sublocus of $\widetilde{BN}^k(v)^{\circ}$ consisting of sheaves $E$ with $\mathrm{coker}(H^0(E) \otimes \O_X \rightarrow E)$ simple and torsion-free in a neighborhood of the base locus of $|-K_X|$ is smooth of the expected dimension, or empty.
        \item If $k \geq \mathrm{max}(\chi(v),r)$, assume $c_1$ is indecomposable (see Definition~\ref{def:indecomposable}) and $X$ is a K3 surface. Then the Brill--Noether loci $BN^k(v)^{\circ}$ 
        are normal and all connected components are irreducible of the expected dimension. Moreover, its number of connected components is bounded above by that of $\mathrm{Spl}(k-r-\ch_2,c_1,\ch_2)$.
    \end{enumerate}
\end{theorem}

\noindent Here, the expected dimension of $BN^k(v)$ is given by the formula
$$
\mathrm{expdim}(BN^k(v)) = \dim \Spl(v)-k(k-\chi(v)).
$$
Non-emptiness of the loci discussed in Theorem~\ref{intro:thm1} can fail, see however Theorem~\ref{BNconnected}(iv) for a positive result.

A special case of Theorem~\ref{intro:thm1} is when $X$ is a K3 surface and $c_1$ is indecomposable (equivalently, the complete linear system $|c_1|$ consists only of irreducible and reduced curves). In Picard rank~$> 1$, this assumption on $c_1$ is weaker than the numerical minimality assumption with respect to a polarization used by Yoshioka \cite[Cor. 3.11]{YoshiokaBNTHeory} and Leyenson \cite[Thm. 3.9]{Leyenson} whose results we recover in particular (see Remark~\ref{rem:specializingmainthm}).

It is a feature of the Lazarsfeld--Mukai construction that it connects the Brill--Noether theory of line bundles on curves in a fixed linear system $|c_1|$ on $X$ to that of a certain class of higher rank vector bundles on $X$ (namely, the class of Lazarsfeld--Mukai bundles, or, more generally, \emph{generically $r$-generated} torsion-free sheaves). This leads naturally to studying the moduli space of coherent systems (see Section~\ref{coherentsystems}) as introduced by Le Potier \cite{LePotier}. Its deformation theory has been worked out thoroughly by He \cite{He}. The latter may be applied directly to obtain smoothness, irreducibility and dimension results for the (relative) Brill--Noether loci of line bundles on curves in a fixed linear system.

One of our results in this direction is a generalization of a dimension bound due to Aprodu and Farkas \cite[Prop. 2.4]{FarkasAproduGreenConj} especially relevant to Green's conjecture: Denote by $\mathcal{W}^r_d(|\mathcal{L}|)$ the relative Brill--Noether locus over the linear system $|\mathcal{L}|$ and by $W^r_d(C)$ its fiber over $C\in |\mathcal{L}|$. Denote by $E_{C,A}$ the Lazarsfeld--Mukai bundle associated to a curve $C\subset X$ and a basepoint-free torsion-free rank $1$ sheaf $A$ on $C$. We show:

\begin{proposition}[Proposition \ref{prop:generalizedAproduFarkas}]
 \label{intro:thm2}
     Let $X$ be a smooth projective surface with $H^1(X,\O_X) = 0$. Let $C \in |\mathcal{L}|$ be an integral curve with $g \coloneqq p_a(C) \geq  2$.     
     \begin{enumerate}[(i)]
         \item If $K_X \neq 0$, we assume in addition that $-K_X|_C > 0$. Let $A$ be a torsion-free rank~$1$ sheaf on $C$, basepoint-free of degree $d$ with $h^0(A) = r+1$. Then
        $$
        \dim_{(C,A)} \mathcal{W}^r_d(|\mathcal{L}|)\leq \rho(g,r,d)+ \dim|\mathcal{L}|+\mathrm{dim}\Hom(E_{C,A},E_{C,A})-1.
        $$
        \item We assume that at least one of the following holds:
        \begin{enumerate}[(a)]
            \item $-K_X.\mathcal{L} \geq g$;
            \item $-K_X > 0$ and $K_X.\mathcal{L} \neq 0$; 
            \item $K_X = 0$.
        \end{enumerate}
        Then, for the general $C \in |\mathcal{L}|_{\mathrm{int}}$ and any torsion-free rank~$1$ sheaf $A$ of degree $d$ on $C$, basepoint-free with $h^0(A) = r+1$, we have
        $$
            \dim _AW^r_d(C) \leq \rho(g,r,d)+\dim \Hom(E_{C,A},E_{C,A})-1.
        $$
     \end{enumerate}
 \end{proposition}
 In the $r=1$ case, similar results have been obtained for certain rational surfaces in \cite{Lelli-ChiesaGreenConj} and for $r=2$ on K3 surfaces in \cite{Lelli-ChiesaRank3LM}.

 This has two main applications. In the case of basepoint-free complete pencils $\widetilde{\mathcal{W}}^{1, \mathrm{bpf}}_d(|\mathcal{L}|_{\mathrm{int}})$, we identify an explicit subspace containing its entire singular locus. Namely, set 
  $$
  \mathcal{S}\coloneqq \{(C,A)\colon A \text{ is split or surface-induced}\}\subseteq \widetilde{\mathcal{W}}^{1,\mathrm{bpf}}_d(|\mathcal{L}|_{\mathrm{int}}) \quad (\text{Definition } \ref{def:surface-induced}).
  $$

  \noindent Then we have:
  \begin{theorem}[Theorem~\ref{theorem:GlobalCompletePencilIsSmooth}]
  \label{intro:thm3}
    Let $X$ be a smooth projective surface with $H^1(X,\O_X) = 0$. Let $\mathcal{L} \in \Pic(X)$ be non-trivial and effective and let $C \in |\mathcal{L}|$ be an integral curve. Set $g \coloneqq p_a(C)$.
    
    \begin{enumerate}[(i)]
        \item If $K_X \neq 0$, assume $C$ satisfies $-K_X|_C>0$. Let $A$ be a complete basepoint-free pencil of degree $d$ on $C$ that is neither split nor surface-induced. Then,
    $$
    (C,A)\in \mathcal{W}^1_d(|\mathcal{L}|)
    $$
    is a smooth point of local dimension $\rho(g,1,d)+\dim |\mathcal{L}|$.
    \item  We assume that either
    \begin{enumerate}[(a)]
        \item $-K_X > 0$, $K_X.\mathcal{L} \neq 0$ and $g \geq 2$; or
        \item $K_X = 0$.
    \end{enumerate}
    Then $\widetilde{\mathcal{W}}^{1,\mathrm{bpf}}_d(|\mathcal{L}|_{\mathrm{int}})\setminus\mathcal{S}$ is smooth of dimension $\rho(g,1,d)+\dim |\mathcal{L}|$, or empty.
    \end{enumerate} 
\end{theorem}

Following the strategies of Aprodu--Farkas \cite{FarkasAproduGreenConj} and Lelli-Chiesa \cite{Lelli-ChiesaGreenConj}, the second application of Proposition~\ref{intro:thm2} consists of new cases of Green's conjecture which predicts
\begin{equation*}
    K_{p,2}(C,\omega_C) = 0 \quad \Leftrightarrow \quad p < \mathrm{Cliff}(C)
\end{equation*}
for all smooth curves $C$, where $\mathrm{Cliff}(C)$ denotes the Clifford index.

\begin{theorem}[{Theorem \ref{thm:genericgreen}}]
\label{intro:greensconjecture}
    Let $X$ be a regular surface with $H^0(X,\omega_X) = 0$. Let $\mathcal{L}$ be an effective line bundle on $X$ such that $|\mathcal{L}|$ contains a smooth integral curve and the general $C\in |\mathcal{L}|$ satisfies $-K_X|_C>0$. Assume that at least one of the following holds:
    \begin{enumerate}[(i)]
    \item For the general $C \in |\mathcal{L}|_{\mathrm{sm}}$, we have $h^0(\omega_X^\vee|_C) \geq 2$.
    \item For all decompositions $[C]=2H+H'$ in $\Pic(X)$ with $H > 0$, $H' \geq 0$, we have $K_X.H'\leq 0$ or $H^1(X,\O_X(H+H')) =0$ (or both).
    \end{enumerate}
     Then Green's conjecture holds for the general curve in $|\mathcal{L}|$.
\end{theorem}

\begin{corollary}[{Corollary \ref{cor:delPezzoBlowup}}]
\label{intro:greencorollary}
    Let $Y$ be a del Pezzo surface and $X \to Y$ the blowup of $Y$ in $K_Y^2$ general points. Let $\mathcal{L}$ be any line bundle on $X$ containing a smooth curve $C \in |\mathcal{L}|$ which is not a fixed component of $|-K_X|$ and satisfies $C.K_X \neq 0$. Then Green's conjecture holds for the general smooth curve in $|\mathcal{L}|$.
\end{corollary}

\medskip
\noindent
The techniques used to prove Theorem \ref{intro:thm1} can be divided into two cases:
In the $k \geq r$ case, we globalize the Lazarsfeld--Mukai construction pioneered in \cite{Lazarsfeld}, setting up a framework to systematically translate between relative Brill--Noether results on curves in a given linear system $|c_1|$ on a surface, to Brill--Noether results on the surface itself. On the technical level, this is facilitated by working with the moduli space of simple sheaves $\mathrm{Spl}(v)$ instead of the moduli space of stable sheaves $M(v)$, especially for surfaces of Picard rank $> 1$.
Under certain conditions on $K_X$ and $c_1$, we (re-)prove Brill--Noether results for curves in a given linear system and employ the previously developed framework to prove parts of Theorem~\ref{intro:thm1}.
In the $k<r$ case, we leverage the deformation theory of the \emph{moduli space of coherent systems} (Definition \ref{def:coherentSystem}) as developed in \cite{He} to obtain smoothness results on $\widetilde{BN}^k(v)^{\circ}$ directly.

\noindent The applications to basepoint-free pencils on curves come from a stronger simplicity lemma (Lemma~\ref{lemma:simplicity2}) in rank two which is a direct generalization of the Donagi--Morrison Lemma \cite{DonagiMorrison}, and through the observation that an open subspace of $\Spl(0,c_1,\ch_2)$ can be identified with a relative compactified Jacobian over $|c_1|_{\mathrm{int}}$. For the applications to Green's conjecture, Aprodu \cite{AproduSyzygies} reduced the conjecture to a linear growth condition on $\dim W^1_d(C)$. We analyze constant-cohomology-strata in a parameter space of Lazarsfeld--Mukai bundles and provide dimension counts for each stratum to establish linear growth.

\medskip
\noindent
The paper is organized as follows: In Section~\ref{preliminaries} we recall the Lazarsfeld--Mukai construction and set up the necessary homological algebra of generically globally generated sheaves, including key lemmas needed to prove smoothness results in later sections.
In Section~\ref{coherentsystems} we introduce a variant of the moduli spaces $\Syst^k(v)$ of coherent systems and outline its tangent-obstruction theory~\cite{He}, taking a slightly more general perspective that does not require ($\alpha$-)stability. We derive Proposition~\ref{intro:thm2} from the tangent-obstruction theory and prove certain smoothness results for $\Syst^k(v)$ in Theorems~\ref{theorem:SystSmoothnessFork=r} and~\ref{theorem:SystSmoothnessFork<r}.
The main mechanism in Section~\ref{BNTheorem} is that a generically generated coherent system $(E,V)$ defines a quotient $\mathrm{coker}(V\otimes \O_X\to E)$ and conversely, the Lazarsfeld--Mukai construction yields an association $(C,A,V)\mapsto (E_{C,A,V},H^0(C,A)^{\vee})$, where $C\subset X$ is an integral curve, $(A,V)$ a basepoint-free linear system on $C$ and $E_{C,A,V}$ the assoicated Lazarsfeld--Mukai bundle. We globalize these constructions over the relevant moduli spaces introduced in Section \ref{coherentsystems} (Lemma \ref{lemma:QuotientMapSurjective}). This amounts to the main correspondence (\ref{eqn:correspondence}) used to prove Theorem \ref{intro:thm1}.
In Section~\ref{ApplicationToCurves} we discuss applications to pencils on curves. We apply the generalized Donagi--Morrison Lemma~\ref{lemma:simplicity2} in the rank~$2$ case in combination with the strategies from Sections~\ref{coherentsystems} and~\ref{BNTheorem} to prove Theorem~\ref{intro:thm3}. No indecomposability assumption on the linear system is required, and the effectivity assumption on $-K_X$ can be relaxed to one for $-K_X|_C$ for $C \in |\mathcal{L}|$. We end by proving our application to Green's conjecture, Theorem~\ref{intro:greensconjecture}.

\medskip
\noindent
\textbf{Acknowledgements:} We are grateful to Angela Ortega and Gavril Farkas for helpful discussions and comments over the course of this project. During the time of this research, S.G. was funded by the Deutsche Forschungsgemeinschaft (DFG, German Research Foundation) — RTG2965 — Project number 512730679 at the Humboldt-University of Berlin. A.K. was supported by the ERC Advanced Grant SYZYGY (No. 834172, PI: Gavril Farkas).

\medskip
\noindent
\textbf{Statement on the use of AI:} Conversations with an LLM have improved the following parts of the paper: A slicker proof of Lemma~\ref{lemma:CurvesOnDelPezzo}, the realization that the moduli space $\Spl(v)$ need not always be connected depending on the generality of the allowed surfaces and Chern characters, an observation about the locus of split pencils in the proof of Corollary~\ref{cor:degreeconditionspencil}, and some improvements to Example~\ref{example}. Moreover, Theorem~\ref{thm:genericgreen} has been improved through a better understanding of the locus of split pencils and of the bound on $\gamma_1$.
The responsibility for the validity of all results and proofs in this paper remains solely with the authors.

\section{Preliminaries}
\label{preliminaries}
\noindent In this section, we set up the necessary homological algebra of coherent sheaves on surfaces. In particular, we recall and generalize fundamental simplicity lemmas of Lazarsfeld and Donagi--Morrison that are crucial in proving our main Brill--Noether theorem in Section~\ref{BNTheorem} and the applications to curves in Section~\ref{ApplicationToCurves}. The central objects of study are the Lazarsfeld--Mukai bundles and their torsion-free relaxations. Let us recall the basic constructions in this context and introduce some necessary notation.

\subsection{Notation and conventions.}
\begin{itemize}
    \item For a Cartier divisor $D$ on a variety $Y$, we write $D\geq 0$ if $D$ is effective, i.e., $\O_Y(D)$ has a nonzero section. We write $D>0$ if $D$ is effective and not rationally equivalent to zero.
    \item  We denote by $\mathrm{Spl}(v)$ the moduli space of pure simple sheaves with Chern character $v$. A coherent sheaf is said to be simple if $\Hom(E,E) \simeq \C$.
    \item We denote by $M(v)$ the moduli space of slope-stable sheaves with Chern character $v$ (with respect to a fixed polarization). The inclusion $M(v) \subseteq \mathrm{Spl}(v)$ is an open embedding.
    \item  We denote by $BN^k(v)$ the $k$-th Brill--Noether locus in $\mathrm{Spl}(v)$, i.e., 
    $$
    BN^k(v)\coloneqq \{E\in \mathrm{Spl}(v)\colon h^0(X,E)\geq k\}.
    $$
    We denote by $\widetilde{BN}^k(v)\subset BN^k(v)$ the locally closed Brill--Noether locus 
    $$
    \widetilde{BN}^k(v)\coloneqq \{E\in \mathrm{Spl}(v)\colon h^0(X,E) = k\}.
    $$
    These will always be understood as subspaces of the moduli space of simple sheaves. The corresponding stable Brill--Noether loci in $M(v)$ are denoted $BN^k(v) \cap M(v)$.
    \item We view $c_1(E) \in \mathrm{CH}^1(X) \simeq \Pic(X)$ and we denote by $|c_1(E)|$ or $|c_1|$ the associated complete linear system on $X$.
\end{itemize}

\subsection{The Lazarsfeld--Mukai construction.} To a surface $X$ and a curve $i\colon C\hookrightarrow X$ together with a basepoint-free linear system $(A,V)$ on $C$ one can associate the \emph{Lazarsfeld--Mukai bundle} $E_{C,A,V}$ defined by the exact sequence
$$
0\to E_{C,A,V}^{\vee}\to V\otimes \O_X\to i_*A\to 0.
$$
The notation is justified as the kernel of the evaluation map is indeed a vector bundle of rank $\rk(E_{C,A,V}) = \dim V$. Dualizing yields
$$
0\to V^{\vee}\otimes \O_X\to E_{C,A,V}\to i_*L\to 0,
$$
where 
$$
L = A^{\vee}\otimes\omega_C\otimes\omega_X^{\vee}|_C\simeq \mathcal{E}xt^1_{\O_X}(i_\ast A,\O_X)|_C.
$$

\subsection{Some homological algebra}
The following notion will be crucial throughout the paper.

\begin{definition}
We say a coherent sheaf $E$ is \emph{generically globally generated} (or \emph{ggg} for short) if the total evaluation map $H^0(X,E)\otimes\O_X\to E$ is surjective at the generic point. In particular, if $E$ is ggg, then $h^0(X,E)\geq \rk(E)$ and there exists an injection $\O_X^{\rk(E)}\hookrightarrow E$.
\end{definition}

\begin{lemma}
    \label{nosubskyscraper}
    Let $X$ be a smooth projective surface. Let $F$ be a locally free sheaf on $X$ and $E$ a torsion-free sheaf on $X$ with $\rk(F)= \rk(E) = r$. If $F \hookrightarrow E$ is an injective map of sheaves, the cokernel $Q =\mathrm{coker}(F\to E)$ is either zero or supported in pure codimension one. In particular, $Q$ contains no subsheaves supported in dimension~$0$.
\end{lemma}

\begin{proof}
    Take $m\gg0$ and consider the short exact sequence
    $$
    0\to F(-m)\to E(-m)\to Q(-m)\to 0.
    $$
    Assume there exists a nonzero subsheaf $T_p \subset Q$ supported on a closed point $p$. Then $T_p(-m) = T_p \subseteq Q(-m)$ and hence $h^0(X,Q(-m))\neq 0$ for all $m$ because $h^0(X, T_p) = \dim_{\C}(T_p(p)) > 0$. Consider the cohomology exact sequence:
    $$
    \cdots\to H^0(X,E(-m))\to H^0(X,Q(-m))\to H^1(X,F(-m))\to \cdots.
    $$
    Since $F$ is locally free, Serre duality applies and we find $H^1(X,F(-m))\simeq H^1(X,F^{\vee}\otimes \omega_X(m))^\vee = 0$ for $m\gg0$. Moreover, as $E$ is torsion-free, we find $E\subseteq E^{\vee\vee}$ and $E^{\vee\vee}$ is locally free on a surface, thus 
    $$
    H^0(X,E(-m))\subseteq H^0(X,E^{\vee\vee}(-m)) \simeq H^2(X,E^{\vee}\otimes \omega_X(m))^\vee = 0
    $$
    for $m \gg 0$. This contradicts $h^0(X,Q(-m))\neq 0$ for all $m$.
\end{proof}

\begin{proposition}
    \label{extensionistorsionfree}
    Let $X$ be a smooth projective surface and $i \colon C \hookrightarrow X$ the embedding of an integral curve. Let $L$ be a torsion-free sheaf of rank $1$ on $C$. Then any sheaf $F$ of rank $r \geq 1$ sitting in a short exact sequence
    $$
    0 \to \O_X^r \to F \to i_\ast L \to 0
    $$
    is torsion-free unless the sequence splits.
\end{proposition}

\begin{proof}
    We show that the extension splits or the natural map $F\to F^{\vee\vee}$ is injective. Consider the following commutative diagram with exact rows and columns:

\[\begin{tikzcd}
	&& 0 & 0 \\
	&& {F_{\text{tor}}} & K \\
	0 & {\mathcal{O}_X^r} & F & {i_*L} & 0 \\
	0 & {\mathcal{O}_X^r} & {F^{\vee \vee}} & {i_*A} & 0 \\
	&& {\bigoplus_pF_p} & {\bigoplus_{p'} F'_{p'}} & 0 \\
	&& 0 & 0
	\arrow[from=1-3, to=2-3]
	\arrow[from=1-4, to=2-4]
	\arrow["\simeq", from=2-3, to=2-4]
	\arrow[from=2-3, to=3-3]
	\arrow[from=2-4, to=3-4]
	\arrow[from=3-1, to=3-2]
	\arrow[from=3-2, to=3-3]
	\arrow[equal, from=3-2, to=4-2]
	\arrow[from=3-3, to=3-4]
	\arrow[from=3-3, to=4-3]
	\arrow[from=3-4, to=3-5]
	\arrow[from=3-4, to=4-4]
	\arrow[from=4-1, to=4-2]
	\arrow[from=4-2, to=4-3]
	\arrow[from=4-3, to=4-4]
	\arrow[from=4-3, to=5-3]
	\arrow[from=4-4, to=4-5]
	\arrow[from=4-4, to=5-4]
	\arrow[from=5-3, to=5-4]
	\arrow[from=5-3, to=6-3]
	\arrow[from=5-4, to=5-5]
	\arrow[from=5-4, to=6-4]
\end{tikzcd}\]
The left vertical sequence is the four-term exact sequence coming from the natural map $F\to F^{\vee\vee}$. The cokernel of $F\to F^{\vee \vee}$ is zero or supported in dimension zero and is thus of the form $\bigoplus_pF_p$ where $F_p$ is a finite length sheaf supported at $p$. Then $F^{\vee\vee}$ is a vector bundle since reflexive sheaves on a surface are vector bundles. The composition $\O_X^r\to F\to F^{\vee\vee}$ is injective. We denote the cokernel by $i_*A$; this notation is justified as we will see shortly. From the universal property of the cokernel we obtain a map $i_*L \to i_*A$. Denote its cokernel by $Q$ for the moment. Again by the universal property of the cokernel, we get a surjective map $\bigoplus_p F_p \to Q$. Thus, indeed, $Q$ has the form $Q = \bigoplus_{p'} F'_{p'}$. As $\bigoplus_p F_p$ and thus $\bigoplus_{p'} F'_{p'}$ are supported on $C$ (scheme-theoretically), it follows that $i_*A$ is supported on $C$ as well, justifying the notation.

\noindent Now, $A$ is of rank at most one on $C$ by the right vertical sequence. Therefore, two cases can occur:
\begin{enumerate}[(i)]
    \item $A$ has rank zero on $C$. In this case $A=0$ by Proposition \ref{nosubskyscraper}. Hence, the sheaf $K$ is isomorphic to $i_*L$ and, by the Snake Lemma, we get $i_\ast L \simeq K \simeq F_{\mathrm{tor}}$. The map $i_*L \xrightarrow{\sim} K \xrightarrow{\sim}F_{\mathrm{tor}} \to F$ then provides a splitting of the original sequence $F \simeq \O_X^r \oplus i_*L$.
    \item $A$ has rank one on $C$. In this case the map $L \to A$ has full rank since the cokernel has rank zero. As $L$ is a torsion-free sheaf of rank one on $C$, the map $L \to A$ is necessarily injective. It follows that $K=0$ and, again by the Snake Lemma, $F_{\mathrm{tor}} = 0$, in particular~$F$ is torsion-free.\qedhere
\end{enumerate}
    
\end{proof}

\begin{lemma}\label{lemma:simplenessAndDoubleDual}
    Let $E$ be a torsion-free coherent sheaf on a smooth surface $X$. If $ E^{\vee \vee}$ is a simple sheaf, i.e., $\Hom(E^{\vee \vee}, E^{\vee \vee}) \simeq \C$, then so is $E$.
\end{lemma}

\begin{proof}
    For every coherent sheaf there is a canonical map $i\colon E \rightarrow E^{\vee \vee}$ which is injective if and only if $E$ is torsion-free. These canonical maps provide a natural transformation from the identity functor to the biduality functor. Hence, for any map $f\colon E \rightarrow E$, we have a commutative square
    \begin{equation*}
        \begin{tikzcd}
            E \ar[d, hook, "i"'] \ar[r, "f"] & E \ar[d, hook, "i"] \\
            E^{\vee \vee} \ar[r, "f^{\vee \vee}"'] & E^{\vee \vee}.
        \end{tikzcd}
    \end{equation*}
    By assumption, $f^{\vee \vee} = \lambda \cdot \mathrm{id}$ for some $\lambda \in \C$. But then, replacing the top horizontal map $f\colon E \rightarrow E$ in the diagram by $\cdot \lambda\colon E \rightarrow E$ still makes the diagram commutative. This implies $i \circ (f - \lambda \cdot \mathrm{id}) = 0$, hence $f = \lambda \cdot \mathrm{id}$ as $i$ is injective.
\end{proof}

\begin{definition}
    Let $E$ be a coherent sheaf on $X$. We say that a map $f\colon \O_X^k \rightarrow E$ \emph{splits off a trivial summand} if there is a commutative diagram
    \begin{equation*}
        \begin{tikzcd}[ampersand replacement=\&,column sep = 6em, row sep=3em]
            \O_X^k \ar[d, "\simeq"'] \ar[r, "f"] \& E \ar[d, "\simeq"] \\
            \O_X \oplus \O_X^{k-1} \ar[r, "{\begin{bsmallmatrix} 1 & 0 \\ 0 & \ast \end{bsmallmatrix}}"'] \& \O_X \oplus E'
        \end{tikzcd}
    \end{equation*}
    with vertical isomorphisms.
\end{definition}

\begin{remark}\label{rmk:splittingOffTrivialSummand}
    An injection $\O_X^r \hookrightarrow E$ with $r = \rk(E)$ splits off a trivial summand if and only if $E$ has a nonzero map to $\O_X$, i.e., $\Hom(E, \O_X) \neq 0$. Indeed, the composition $\O_X^r \hookrightarrow E \rightarrow \O_X$, where the second map is any such nonzero map, is still nonzero. If not, the nonzero map $E \rightarrow \O_X$ would factor through the cokernel $\mathrm{coker}(\O_X^r \hookrightarrow E) \rightarrow \O_X$ which is necessarily zero because it is a map from a torsion sheaf to a torsion-free sheaf. But now, any nonzero map $\O_X^r \rightarrow \O_X$ is automatically surjective and in fact at least one of the $r$ direct summands $\O_X \subseteq \O_X^r$ must map isomorphically onto $\O_X$ under the composition $\O_X \subseteq \O_X^r \hookrightarrow E \rightarrow \O_X$. This, in turn, implies that $\O_X^r \hookrightarrow E$ splits off a trivial summand.
\end{remark}

The same reasoning as in Remark~\ref{rmk:splittingOffTrivialSummand} also shows that on a smooth projective surface $X$ with $H^0(X,\omega_X) = 0$ any ggg sheaf $E$ on $X$ satisfies $H^2(X,E) = 0$. Indeed, if $H^2(X,E) \neq 0$, there exists a nonzero morphism $E \to \omega_X$ which, after precomposition with $\O_X^r \hookrightarrow E$, yields a nonzero section of $\omega_X$. For $\omega_X \simeq \O_X$ the same is true unless $E$ has $\O_X$ as a direct summand.

\subsection{Simplicity and obstruction lemmas}

\begin{definition}\label{def:indecomposable}
    Let $X$ be a smooth projective variety. We call a nonzero effective divisor class $D \in \Pic(X)$ \emph{indecomposable} if there is no decomposition $D = H_1 + H_2$ as a sum of two nonzero effective divisor classes $H_1$ and $H_2$. Equivalently, all members $C \in |O_X(D)|$ are irreducible and reduced.
\end{definition}

The following is a variation of \cite[Lemma~1.3]{Lazarsfeld}.

\begin{lemma}[Simplicity Lemma 1]\label{simplicitylemma1}
    Let $E$ be a ggg torsion-free sheaf on a smooth projective surface $X$ with $\rk(E)\geq 1$. Assume that $c_1(E)$ is indecomposable. Then $E$ is simple if and only if $E$ does not have a trivial direct summand or, equivalently, $H^0(E^\vee) = \Hom(E, \O_X) = 0$.
\end{lemma}

\begin{proof}
    Let $r \coloneqq \rk(E)$. Clearly, if an injection $\O_X^r \hookrightarrow E$ splits off a trivial summand, then in particular $E \simeq \O_X \oplus E'$ with $E' \neq 0$, so $E$ is not simple.

    Conversely we may assume that $E$ does not have a trivial direct summand. If $E$ is not simple, there is a nonzero endomorphism $\phi: E \rightarrow E$ which drops rank everywhere. To see this, let $f: E \rightarrow E$ be any endomorphism that is not a multiple of the identity. Let $\lambda \in \C$ be an eigenvalue of the linear map $f(p): E(p) \rightarrow E(p)$ at some point $p \in X$ in the open locus where $E$ is locally free. Then the determinant section $\det(f - \lambda \cdot \mathrm{id}) \in H^0(\det(E)^\vee \otimes \det(E)) \simeq H^0(\O_X)$ vanishes at~$p$ and hence everywhere, but $f - \lambda \cdot \mathrm{id} \neq 0$. Hence, $\phi \coloneqq f - \lambda \cdot \mathrm{id}$ has the desired property. Just as in Lazarsfeld's paper \cite{Lazarsfeld}, we define $N \coloneqq \mathrm{Im}(\phi)$, $M \coloneqq \mathrm{coker}(\phi)$ and $M_{\mathrm{tf}} \coloneqq M/M_{\mathrm{tor}}$, the latter denoting the torsion-free part of $M$. The ranks of both $N$ and $M_{\mathrm{tf}}$ on $X$ are positive and strictly smaller than $r$. From the short exact sequences
    \begin{gather*}
        0 \rightarrow N \rightarrow E \rightarrow M \rightarrow 0, \\
        0 \rightarrow M_{\mathrm{tor}} \rightarrow M \rightarrow M_{\mathrm{tf}} \rightarrow 0,
    \end{gather*}
    we obtain $c_1(E) = c_1(N) + c_1(M_{\mathrm{tf}}) + c_1(M_{\mathrm{tor}})$. Clearly, $c_1(M_{\mathrm{tor}}) \geq 0$. Now, observe that the compositions $\O_X^r \hookrightarrow E \twoheadrightarrow N$ and $\O_X^r \hookrightarrow E \twoheadrightarrow M_{\mathrm{tf}}$ give surjections $\O_{X,\eta}^r \twoheadrightarrow N_\eta$ and $\O_{X,\eta}^r \twoheadrightarrow (M_{\mathrm{tf}})_\eta$ at the generic point $\eta \in X$ because $\O_{X,\eta}^r \overset{\simeq}{\longrightarrow} E_\eta$ is an isomorphism. From this, we obtain maps $\O_X^{\rk N} \rightarrow N$ and $\O_X^{\rk M_{\mathrm{tf}}} \rightarrow M_{\mathrm{tf}}$ that are isomorphisms at $\eta$ and therefore injective. Consider the induced short exact sequences
     \begin{gather*}
        0 \rightarrow \O_X^{\rk N} \rightarrow N \rightarrow P \rightarrow 0, \\
        0 \rightarrow \O_X^{\rk M_{\mathrm{tf}}} \rightarrow M_{\mathrm{tf}} \rightarrow P' \rightarrow 0.
    \end{gather*}
    Then $P$ and $P'$ are both torsion sheaves and in particular $c_1(N)\geq0$ and $c_1(M_{\mathrm{tf}})\geq 0$. Since $c_1(E)$ is indecomposable, at least one of $c_1(N)$ and $c_1(M_{\mathrm{tf}})$ must vanish. If $c_1(N) = c_1(P) = 0$, then $P$ is supported in dimension $0$. But since $N$ is torsion-free, Lemma~\ref{nosubskyscraper} implies $P = 0$, i.e., $N \simeq \O_X^{\rk(N)}$. As $E$ surjects onto $N$, this implies that $E$ has a trivial direct summand. The conclusion in the case $c_1(P') = c_1(M_{\mathrm{tf}}) = 0$ is the same.
\end{proof}

\begin{lemma}\label{lemma:StrongGenericGeneration}
    Let $E$ be a ggg torsion-free sheaf on a smooth projective surface $X$ and $E \not\simeq \O_X \oplus E'$. Assume that $c_1(E)$ is indecomposable. Then $E$ is strongly ggg, i.e., every map $\O_X^{\rk(E)} \rightarrow E$ which induces an injection on global sections is an injective map of sheaves.
\end{lemma}

\begin{proof}
    Write $r \coloneqq \rk(E)$ and let $\alpha: \O_X^r \rightarrow E$ be any map that gives an injection on global sections. Let $F \subseteq E$ be the saturation of $\mathrm{Im}(\alpha)$. Then $\rk(F) \leq r$ and $\alpha$ is injective if and only if $\rk(F) = r$. Moreover, $F$ is ggg with $h^0(F) \geq r \geq \rk(F)$, so $c_1(F) \geq 0$.
    Now, by definition, $E/F$ is torsion-free and ggg because $E/F$ is a quotient of $E$. In particular, $c_1(E/F) \geq 0$. Since $c_1(E) = c_1(F) + c_1(E/F)$ and $c_1(E)$ is indecomposable, one of the two summands must vanish.
    If $c_1(E/F) = 0$, then $E/F \simeq \O_X^{r-\rk(F)}$ because $E/F$ is ggg. If $r > \rk(F)$, this means that $E$ has $\O_X$ as a direct summand. If on the other hand $c_1(F) = 0$, then $F \simeq \O_X^{\rk(F)}$, so $r \leq h^0(F) = \rk(F)$, implying that $\alpha$ is injective.
\end{proof}

\begin{remark}
A sheaf $E$ being strongly ggg is a quite restrictive property. For example, if $E$ is strongly ggg then any subsheaf $F\subset E$ with $\rk(F)<\rk(E)$ satisfies $h^0(F)\leq \rk(F)$. Indeed, assume $h^0(F) >  \rk(F)$, then the total evaluation map $H^0(F)\otimes \O_X \to F$ is not injective. If $h^0(F)<\rk(E)$, adding $\rk(E)-h^0(F)$ linearly independent global sections to $H^0(F)$, we obtain an injection $\O_X^{\rk(E)}\hookrightarrow E$, implying injectivity of $H^0(F)\otimes \O_X \to F$. If $h^0(F)\geq \rk(E)$, choose any $\rk(E)$-dimensional subspace $V\subset H^0(F)\subset H^0(E)$. Then $V \otimes \O_X \hookrightarrow E$ is injective and restricts to an injective map $V \otimes \O_X \to F$, contradicting $\rk(F) < \rk(E)$. In this sense, being strongly ggg vaguely resembles a slope stability condition; it restricts the number of global sections of all subsheaves and therefore, in some sense, their positivity. The previous two lemmas show that any torsion-free simple ggg sheaf on a surface with $c_1(E)$ indecomposable is strongly ggg.
\end{remark}

\begin{lemma}
\label{lemma:intermediatesimplicity}
    Let $E$ be a ggg torsion-free sheaf on a smooth projective surface $X$. Assume that the complete linear system $|c_1(E)|$ is non-empty and consists only of reduced curves or, equivalently, that for every nonzero effective divisor class $H$, the class $c_1(E) - 2H$ is not effective. Assume in addition that $E$ satisfies at least one of the following two conditions:
    \begin{enumerate}[(i)]
        \item $\rk(E) = 2$, or
        \item $E$ is strongly ggg, i.e., any map $\O_X^{\rk(E)} \rightarrow E$ that induces an injection on global sections is an injective map of sheaves.
    \end{enumerate}
    Then $E$ is either simple or a direct sum $E \simeq E' \oplus E''$ with $E' \neq 0 \neq E''$.
\end{lemma}

In particular, condition $(ii)$ is automatically satisfied for a ggg sheaf $E$ with $h^0(E) = \rk(E)$.

\begin{proof}
    If $E$ is not simple, then as in the proof of Lemma~\ref{simplicitylemma1} there is a nonzero endomorphism $f: E \rightarrow E$ which drops rank everywhere. Let $N \coloneqq \mathrm{Im}(f)$, $M \coloneqq \mathrm{coker}(f)$ and $K \coloneqq \ker(f)$. We may also assume that $N$ is simple because if not, we may apply the same construction to $N$ again until we reach a simple sheaf. We have the two short exact sequences
    \begin{align*}
        0 \rightarrow N \rightarrow E \rightarrow M \rightarrow 0, \\
        0 \rightarrow K \rightarrow E \rightarrow N \rightarrow 0.
    \end{align*}
    Since both $N$ and $M$ are quotients of $E$, both are ggg and in particular, $h^0(N) \geq \rk(N)$, $h^0(M) \geq \rk(M)$. We claim that also $h^0(K) \geq \rk(K)$. Indeed, define $k \geq 0$ by $h^0(E) = \rk(E) + k$. If $h^0(N) > \rk(N) + k$, then the exact sequence
    \begin{equation*}
        0 \rightarrow H^0(N) \rightarrow H^0(E) \rightarrow H^0(M)
    \end{equation*}
    shows that the image of the map $H^0(E) \rightarrow H^0(M)$ has dimension $h^0(E) - h^0(N) < \rk(E) + k - (\rk(N) + k) = \rk(M)$. This is impossible because for any injection $\O_X^r \hookrightarrow E$, the composition $\O_X^r \hookrightarrow E \rightarrow M$ gives a surjection at the generic point $\eta \in X$. In particular, at least $\rk(M)$ of the $r$ independent global sections of $E$ given by the injection $\O_X^r \hookrightarrow E$ map to independent global sections of $M$. This proves $h^0(N) \leq \rk(N) + k$.
    But then, by the second short exact sequence, we have 
    $$
    h^0(K) \geq h^0(E) - h^0(N) \geq \rk(E) + k - (\rk(N) + k) = \rk(M) = \rk(K).
    $$
    Under condition $(ii)$ we get in addition that $K$ is (even strongly) ggg because it is a subsheaf of $E$.

    Assume $E$ is not a direct sum. The composition $E \overset{f}{\rightarrow} E \rightarrow M$ vanishes by definition. Therefore, we obtain the following commutative diagram with identical exact rows, for some $\lambda \in \C$:
    \begin{equation*}
        \begin{tikzcd}
            0 \ar[r] & N \ar[r] \ar[d, "\lambda \cdot \mathrm{id}"] & E \ar[r] \ar[d, "f"] & M \ar[r] \ar[d, "0"] & 0 \\
             0 \ar[r] & N \ar[r] & E \ar[r] & M \ar[r] & 0.
        \end{tikzcd}
    \end{equation*}
    Since $f(N)$ is a subsheaf of $N$ and $N$ is simple, we cannot have $f(N) \neq 0$, because then $f(N)\simeq N$ and we have a splitting $E\xrightarrow{f}f(N) \xrightarrow{\sim} N$ of the sequence, therefore $E \simeq N \oplus M$. We conclude $f(N) = 0$, i.e., $\lambda = 0$. This means that $f$ factors as 
    $$
    E \rightarrow M \overset{\hat{f}}{\twoheadrightarrow} N \rightarrow E.
    $$
    Since $f(E) = N$, $\hat{f}$ is necessarily surjective. Now, under condition $(i)$, $\rk(E) = 2$ and hence $\rk(N) = \rk(M) = 1$. In particular $\ker(\hat{f})$ is torsion and therefore $c_1(\ker(\hat{f})) \geq 0$. Under condition $(ii)$, we get the same conclusion because $\ker(\hat{f}) = K/N$ is a quotient of $K$. In both cases we obtain $c_1(M) = c_1(N) + c_1(\ker(\hat{f}))$ and hence
    \begin{equation*}
        c_1(E) = c_1(N) + c_1(M) = 2c_1(N) + c_1(\ker(\hat{f})) \geq 2 c_1(N).
    \end{equation*}
    By assumption on $c_1(E)$, we must therefore have $c_1(N) = 0$. But since $N$ is ggg and $h^0(N) \geq \rk(N)$, this is only possible if $N \simeq \O_X^{\rk(N)}$. Using the surjection $E \rightarrow N$, we deduce $\mathrm{hom}(E,\O_X) \geq \rk(N)$, implying $E \simeq \O_X^{\rk(N)} \oplus E'$.
\end{proof}
The following is a variation of \cite[Lemma 4.4]{DonagiMorrison} and will become central in Section \ref{ApplicationToCurves}.
\begin{lemma}[Simplicity Lemma 2]
\label{lemma:simplicity2}
    Let $E$ be a ggg vector bundle with $\rk(E) = 2$ on a smooth projective surface $X$ with $H^0(E^\vee) = 0$. Then $E$ is either simple or one of the following two statements holds:
    \begin{enumerate}[(i)]
        \item $E$ is a direct sum $E \simeq \O_X(D_1) \oplus \O_X(D_2)$ for two effective and nonzero divisor classes $D_1$ and $D_2$ with $D_1 + D_2 = c_1(E)$;
        \item There is a decomposition $c_1(E) = 2H + H'$ in $\Pic(X)$ with $H > 0$ and $H' \geq 0$ such that $E$ sits in a non-split short exact sequence
    \begin{equation*}
        0 \rightarrow \O_X(H+H') \rightarrow E \rightarrow \mathcal{I}_Z(H) \rightarrow 0,
    \end{equation*}
    where $\mathcal{I}_Z$ is the ideal sheaf of a $0$-dimensional local complete intersection subscheme $Z \subseteq X$.
    \end{enumerate}
\end{lemma}

\begin{proof}
Assume $E$ is not simple. If $E$ is a direct sum of line bundles, both line bundles have to be effective and non-trivial as $E$ is ggg and $H^0(E^\vee) = 0$. This is $(i)$. For $(ii)$, assume $E$ is not a direct sum. Take a nontrivial endomorphism $f\colon E \to E$. As in the proof of Lemma~\ref{lemma:intermediatesimplicity}, $f$ factors as 
$$
E\to M\xrightarrow{\hat{f}} N\to E,
$$
 where $N = \mathrm{Im}(f)$, so $\hat{f}$ is a surjection and $M$ is defined by 
$$
0\to N\to E\to M\to 0.
$$
Since both $N$ and $M$ are rank~$1$ coherent sheaves and $N$ is torsion-free, we obtain
$$
0\to M_{\mathrm{tor}}\to M\xrightarrow{\hat{f}}N\to 0,
$$
hence $N \simeq M_{\mathrm{tf}}$.
Denote $c_1(N) = H$ and $c_1(M_{\mathrm{tor}}) = H'$, so $c_1(M) = H + H'$. We have $H,H'\geq 0$ as $N$ is a quotient of a ggg sheaf and $M_{\mathrm{tor}}$ is a torsion sheaf. Moreover, $c_1(N) = H > 0$ as otherwise $N^{\vee\vee} \simeq \O_X$ which would yield a nonzero map $E\to N\to N^{\vee\vee}\simeq \O_X$ contradicting $H^0(E^\vee) = 0$.
Now, consider the saturation $N^{\mathrm{sat}}$ of $N$ in $E$. By definition, we obtain the sequence
$$
0\to N^{\mathrm{sat}}\to E\to M_{\mathrm{tf}}\to 0.
$$
As $N^{\mathrm{sat}}$ is a saturated subsheaf of a vector bundle on a surface, it is also a vector bundle. Since its rank is~$1$, $N^{\mathrm{sat}}$ is a line bundle with first Chern class $c_1(E) - c_1(M_{\mathrm{tf}}) = c_1(E) - c_1(N) = c_1(M) = H + H'$. Hence, $N^{\mathrm{sat}} \simeq \O_X(H+H')$. As $M_{\mathrm{tf}}$ is a rank~$1$ torsion-free sheaf with first Chern class $H$, the last short exact sequence becomes
$$
0\to \O_X(H+H')\to E\to \mathcal{I}_Z(H)\to 0,
$$
where $\mathcal{I}_Z$ is the ideal sheaf of a $0$-dimensional subscheme $Z\subseteq X$ which is necessarily a local complete intersection by this very presentation.
\end{proof}

\begin{lemma}[Obstruction Lemma]\label{obstructionLemma}
    Let $X$ be a smooth projective surface, $i: C \hookrightarrow X$ the inclusion of an integral curve and $A$ a basepoint-free torsion-free sheaf of rank $1$ on $C$. Let $E \coloneqq E_{C,A}$ be the Lazarsfeld--Mukai bundle defined by the short exact sequence
    \begin{equation}\label{eq:dualLMSequence}
        0 \rightarrow E^\vee \longrightarrow H^0(A) \otimes \O_X \overset{\ev_A}{\longrightarrow} i_\ast A \rightarrow 0.
    \end{equation}
    Its dual sequence is
    \begin{equation}\label{eq:LMSequence}
        0 \rightarrow \O_X^r \longrightarrow E \overset{f}{\longrightarrow} i_\ast L \rightarrow 0,
    \end{equation}
    where $r \coloneqq h^0(A) = \rk(E) \geq 1$ and $L \coloneqq \omega_C \otimes A^\vee \otimes \omega_X^\vee|_C$. In case $h^1(\O_X) \neq 0$, assume that $h^1(E^\vee) = 0$.
    Then $\Hom(E,i_\ast L) \simeq \Hom(E,E)$. In particular, if $E$ is simple, $\Hom(E, i_\ast L) \simeq \C$, generated by~$f$.
\end{lemma}

\begin{remark}\label{rmk:obstructionLemmaRmk}
    Here, $\omega_C$ is the dualizing sheaf of $C$ which is a line bundle because $C$ is Gorenstein. Note that, for a smooth projective surface $X$ of irregularity $h^1(\O_X) \neq 0$, the condition $h^1(E^\vee) = 0$ is equivalent to the injectivity of the map $H^1(\ev_A): H^0(A) \otimes H^1(\O_X) \rightarrow H^1(A)$. On the other hand, if $X$ has irregularity $h^1(\O_X) = 0$, the vanishing $h^1(E^\vee) = 0$ is \emph{automatic} and follows from the long exact cohomology sequence induced by \eqref{eq:dualLMSequence}:
    \begin{equation*}
        0 \rightarrow H^0(E^\vee) \overset{0}{\rightarrow} H^0(A) \overset{\simeq}{\rightarrow} H^0(A) \overset{0}{\rightarrow} H^1(E^\vee) \rightarrow H^0(A) \otimes H^1(\O_X).
    \end{equation*}
\end{remark}

\begin{proof}[Proof of Lemma~\ref{obstructionLemma}]
    By Remark~\ref{rmk:obstructionLemmaRmk}, we have $h^1(E^\vee) = 0$ in all cases. Since $E^\vee$ is locally free, tensoring \eqref{eq:LMSequence} with $E^\vee$ retains exactness. Its induced long exact sequence in cohomology starts with
    \begin{equation*}
        H^0(E^\vee)^{\oplus r} \rightarrow H^0(E^\vee \otimes E) \rightarrow H^0(E^\vee|_C \otimes L) \rightarrow H^1(E^\vee)^{\oplus r}.
    \end{equation*}
    The first and last term both vanish, so we conclude using $H^0(E^\vee \otimes E) = \Hom(E,E)$ and $H^0(E^\vee|_C \otimes L) \simeq \Hom(E|_C,L) \simeq \Hom(E, i_\ast L)$.
\end{proof}

As we shall see, Lemma~\ref{obstructionLemma} shows that, for a simple Lazarsfeld--Mukai bundle with $h^1(E^\vee) = 0$, the injective map in \eqref{eq:LMSequence} defines a smooth point in the moduli space of coherent systems $\Syst^r(\ch(E))$, $r = \rk(E)$, as long as either $K_X = 0$ or $-K_X|_C > 0$.

\begin{lemma}
    \label{lemma:CurvesOnDelPezzo}
    Let $X$ be a smooth projective surface with $H^1(X,\O_X) = 0$ and $-K_X > 0$. Let $c_1 \in \Pic(X)$ be indecomposable. Then any curve $C \in |c_1|$ satisfies $p_a(C) \leq 1$, with equality if and only if $C \in |-K_X|$.
\end{lemma}

\begin{proof}
    The conditions on $X$ imply $\chi(\O_X) = 1$. Riemann--Roch applied to $\O_X(K_X+C)$ together with adjunction yields $\chi(\O_X(K_X+C)) = p_a(C)$. Using $H^2(X,\O_X(K_X+C)) = H^0(X,\O_X(-C))^{\vee} = 0$, we obtain
    $$
    h^0(X,\O_X(K_X+C)) =p_a(C)+h^1(X,\O_X(K_X+C)) \geq p_a(C).
    $$
    In particular, if $p_a(C) \geq 1$, there exists an effective divisor $B \sim K_X+C$ and we obtain $C \sim B-K_X$, contradicting the indecomposability of $c_1$ unless $B \sim 0$. In that case, $C \sim -K_X$ is an anticanonical curve.
\end{proof}

\section{The Moduli Space of Coherent Systems} \label{coherentsystems}

In this section we introduce the moduli stack of coherent systems on a surface and its tangent-obstruction theory as developed in \cite{He}. 
With the goal of proving dimension, smoothness and irreducibility results for Brill--Noether loci in the moduli space of simple sheaves under certain assumptions on the Chern character, we show smoothness of loci of coherent systems whose underlying sheaf is simple and whose evaluation map is injective.

Let $X$ be a smooth and irreducible projective surface over $\C$ of irregularity zero. For our purposes, it will be convenient to view the Chern character of a coherent sheaf $\mathcal{F}$ on $X$ as a triple
\begin{equation*}
    \ch(\mathcal{F}) = (r,c_1,\ch_2)
\end{equation*}
with $r = \rk(\mathcal{F}) \in \Z_{\geq 0}$, $c_1(\mathcal{F}) \in \Pic(X)$ and $\ch_2 = \int_X \left(\frac{1}{2} c_1(\mathcal{F})^2 - c_2(\mathcal{F})\right) \in \frac{1}{2} \Z \subseteq \Q$. Therefore, $c_1 = c_1(\mathcal{F})$ can be identified with the isomorphism class of $\det(\mathcal{F})$; this is finer information than its numerical class. By the assumption that $h^1(\O_X) = 0$, $\Pic(X)$ is smooth of dimension~$0$ and therefore $c_1$ is locally constant in flat families of coherent sheaves. In particular, fixing $c_1 \in \Pic(X)$ defines an open and closed property for families of coherent sheaves.\\
In view of the Lazarsfeld--Mukai construction it is of particular interest to study pairs $(E,V)$ of a vector bundle $E$ and a subspace $V\subset H^0(X,E)$ such that the evaluation map $V\otimes\O_X\to E$ is injective.

\begin{definition}\label{def:coherentSystem}
    A \emph{coherent system} on $X$ is a pair $(E,V)$, where $E$ is a coherent sheaf on $X$ and $V \subseteq H^0(X,E)$ a vector subspace. We say that $(E,V)$ has \emph{type $(v,k)$} if $\ch(E) = v$ and $\dim V = k$. A coherent system $(E,V)$ is called \emph{generically generated} if the evaluation map $V \otimes \O_X \rightarrow E$ is of full rank at the generic point.
\end{definition}

\begin{definition}
\label{def:genericallygenerated}
    A coherent sheaf $E$ on $X$ is said to be \emph{generically $k$-generated} if there exists a generically generated coherent system $(E,V)$ of type $(\ch(E),k)$.
\end{definition}

\begin{remark}
    If $k<\rk(E)$, being generically $k$-generated is equivalent to the existence of an injection $\O_X^k\hookrightarrow E$. If $k\geq \rk(E)$, then $E$ being generically $k$-generated is equivalent to $E$ being ggg and $h^0(E) \geq k$.
\end{remark}

\noindent
Let us now introduce several different but closely related stacks of coherent systems. Following \cite{LePotier}, a \emph{family of coherent systems on $X$ of type $(v,k)$ parametrized by a $\C$-scheme $S$} can be defined as a triple $(E,W,p)$ given by a coherent sheaf $E$ on $X \times S$, flat over $S$, with every geometric fiber $E_s$ on $X$ having Chern character $v$, a locally free sheaf $W$ on $S$ of rank $k$ and a surjective morphism of $\O_S$-modules
\begin{equation*}
    p: \mathcal{E}xt^2_{\mathrm{pr}_S}(E,\omega_{X \times S/S}) \twoheadrightarrow W,
\end{equation*}
where $\mathrm{pr}_S: X \times S \rightarrow S$ is the projection. By relative duality and cohomology and base change, the sheaf $\mathcal{E}xt^2_{\mathrm{pr}_S}(E,\omega_{X \times S/S})$ commutes with arbitrary base change \cite[Section~4.3]{LePotier}. For a geometric point $s \in S$ we have $\Ext^2(E_s,\omega_X) \simeq H^0(E_s)^\vee$, so that letting $V \coloneqq W_s^\vee$, the surjection $p_s$ gives a dual injection $V \subseteq H^0(E_s)$, making $(E_s,V)$ into a coherent system in the sense of Definition~\ref{def:coherentSystem}. He's definition of families of coherent systems is equivalent to Le Potier's \cite[p.~553]{He}.

\smallskip \noindent
With the above convention on $c_1$, one may define a prestack (or category fibered in groupoids) $\tilde{\SSyst}_X^k(v)$ whose objects over a $\C$-scheme $S$ are given by families of coherent systems on $X$ of type $(v,k)$ parametrized by $S$, and we will in addition require the underlying coherent sheaf $E$ on $X \times S$ to be pure.
For a morphism of $\C$-schemes $f: S' \rightarrow S$ and objects $(E',W',p')$ over $S'$ and $(E,W,p)$ over $S$, a morphism $(E',W',p') \rightarrow (E,W,p)$ is a pair of isomorphisms $E' \overset{\sim}{\rightarrow} f^\ast(E)$ and $W' \overset{\sim}{\rightarrow} f^\ast(W)$ such that the diagram
\begin{equation*}
    \begin{tikzcd}
        \mathcal{E}xt^2(E',\omega_{X \times S'/S'}) \ar[rr, two heads, "p'"] \ar[d, "\sim"'] & & W' \ar[d, "\sim"] \\
        \mathcal{E}xt^2_{\mathrm{pr}_{S'}}(f^\ast(E),\omega_{X \times S'/S'}) \ar[r, "\simeq"'] &f^\ast(\mathcal{E}xt^2_{\mathrm{pr}_{S}}(E,\omega_{X \times S/S})) \ar[r, two heads, "f^\ast(p)"'] & f^\ast(W)
    \end{tikzcd}
\end{equation*}
commutes.

The prestack properties of $\tilde{\SSyst}_X^k(v)$ are straightforward to check. The usual descent properties for coherent sheaves and morphisms between them show that $\tilde{\SSyst}_X^k(v)$ is in fact a stack for the étale site. Now, the stack $\Coh^{\mathrm{pure}}_X(v)$ of pure coherent sheaves of Chern character~$v$ on $X$ is an open substack of the full stack $\Coh_X(v)$ of coherent sheaves and is therefore algebraic and locally of finite presentation over $\C$. There is the forgetful morphism
\begin{equation*}
    \tilde{\SSyst}_X^k(v) \rightarrow \Coh^{\mathrm{pure}}_X(v),
\end{equation*}
sending $(E,W,p) \mapsto E$. This forgetful morphism $\tilde{\SSyst}_X^k(v) \rightarrow \Coh^{\mathrm{pure}}_X(v)$ is representable, projective and locally of finite presentation. Indeed, by the base change property of the top-degree relative Ext sheaf, for any $\C$-scheme $T$ and morphism $T \rightarrow \Coh^{\mathrm{pure}}_X(v)$ corresponding to a pure coherent sheaf $\mathcal{E}$ on $X \times T$, the fiber product
\begin{equation*}
    \tilde{\SSyst}_X^k(v) \times_{\Coh^{\mathrm{pure}}_X(v)} T
\end{equation*}
is the relative Grassmannian of locally free rank~$k$ quotients of $\mathcal{E}xt^2_{\mathrm{pr}_T}(\mathcal{E},\omega_{X \times T/T})$; this Grassmannian is a projective scheme over $T$. This proves that also $\tilde{\SSyst}_X^k(v)$ is algebraic and locally of finite presentation over $\C$.

Finally, because every object in $\tilde{\SSyst}_X^k(v)$ has $\mathbb{G}_m$ in the center of its automorphism group (in a way compatible with pullbacks), we may use \cite[Theorem~5.1.5]{Rigidification} to work instead with its rigidification $\SSyst_X^k(v)$ which is still algebraic and locally of finite presentation and has the same isomorphism classes of objects over $\Spec(\C)$. Moreover, there is a smooth surjective morphism of finite presentation $\tilde{\SSyst}_X^k(v) \rightarrow \SSyst_X^k(v)$, in fact a $\mathbb{G}_m$-gerbe, and the automorphism group of a coherent system $\Lambda \in \SSyst_X^k(v)(\C)$ is precisely $\Aut(\Lambda)/\C^\ast$.

\begin{definition}
We define the following substacks of $\SSyst_X^k(v)$:
\begin{itemize}
    \item Let $\Syst_X^k(v) \subseteq \SSyst_X^k(v)$ be the open substack consisting of all coherent systems whose underlying \emph{sheaf} is simple; $\Syst_X^k(v)$ is an algebraic space.
    \item Let $\Syst_X^k(v)' \subseteq \SSyst_X^k(v)$ be the open substack consisting of all coherent systems $\Lambda$ with $\Hom(\Lambda,\Lambda) \simeq \C$. Also $\Syst_X^k(v)'$ is an algebraic space and contains $\Syst_X^k(v)$ as an open subspace.
    \item Let $\SSyst_X^k(v)^\circ \subseteq \SSyst_X^k(v)$ be the open substack consisting of all generically generated coherent systems, and let $\Syst_X^k(v)^\circ$ and $\Syst_X^k(v)'^{\circ}$ be the intersections with the previous two substacks.
\end{itemize}
\end{definition}

The forgetful map $\tilde{\SSyst}_X^k(v) \rightarrow \Coh^{\mathrm{pure}}_X(v)$ descends after taking the rigidification of both source and target and in particular induces a representable forgetful map $\Syst_X^k(v) \rightarrow \Spl(v)$, where $\Spl(v)$ is the algebraic space of simple (pure) sheaves.

\begin{proposition}[{\cite[Cor.~1.6]{He}}]
\label{extsequence}
    Let $\Lambda =(E,V)$ be a coherent system over $\C$. Then there exists an exact sequence
    \begin{align*}
    0&\to\Hom(\Lambda,\Lambda)\to \Hom(E,E)\to  \Hom(V,H^0(E)/V)\\
    &\to\Ext^1(\Lambda,\Lambda)\to \Ext^1(E,E)\to\Hom(V,H^1(E))\\
    &\to \Ext^2(\Lambda,\Lambda)\to \Ext^2(E,E)\to \Hom(V,H^2(E))\to 0.
    \end{align*}
\end{proposition}

The tangent-obstruction theory developed by He \cite[Section~3.2]{He} in the context of \mbox{$\alpha$-stability} applies verbatim to our more general setting. We can, however, choose a slightly smaller obstruction space, similarly to what happens for moduli spaces of sheaves:

\begin{proposition}[\cite{He,huyMS}]
\label{smoothnescrit}
    Let $\Lambda \in \Syst_X^k(v)'(\C)$. The tangent space of $\Syst_X^k(v)'$ at $\Lambda$ is $T_{[\Lambda]} \Syst_X^k(v)' \simeq \Ext^1(\Lambda,\Lambda)$. An obstruction space is given by
    \begin{equation*}
        \Ext^2(\Lambda,\Lambda)_0 \coloneqq \ker \left( \Ext^2(\Lambda,\Lambda) \rightarrow \Ext^2(E,E) \overset{\mathrm{tr}}{\rightarrow} H^2(X,\O_X) \right).
    \end{equation*}
    The same holds for every open subspace of $\Syst_X^k(v)'$ containing $\Lambda$.
\end{proposition}

\begin{proof}
   To be precise, He shows that $\Ext^2(\Lambda,\Lambda)$ is an obstruction space for the liftability of $\Lambda$ along small extensions. Consider the forgetful functor $\pi \colon \mathrm{CS}_{X/S} \to \mathrm{Coh}_{X/S}$ from the abelian category of relative coherent systems over the spectrum $S$ of an Artin ring over $\C$ to coherent sheaves; this functor is exact. Notably, coherent sheaves can be viewed as coherent systems of the form $(E,0)$, so He's tangent-obstruction theory also recovers that of coherent sheaves.
   Let now $A' \twoheadrightarrow A$ be a small extension and $G$ an arbitrary coherent system over $A$. Consider the Tor-complex $\Lambda\otimes^{\mathbf{L}}_{A'}A$ and the induced hypercohomology spectral sequence
   $$
E_2^{p,q}= \Ext^p_A(\mathcal{T}\!\mathrm{or}^{A'}_q(\Lambda,A),G)\Rightarrow \mathbb{H}^*(\mathbf{R}\Hom_A(\Lambda\otimes_{A'}^{\mathbf{L}}A,G)).
   $$
   From derived Tensor-Hom-adjunction we find
   $$
   \mathbf{R}\Hom_A(\Lambda\otimes_{A'}^{\mathbf{L}}A,G)\simeq \mathbf{R}\Hom_{A'}(\Lambda,\mathbf{R}\Hom_A(A,G))\simeq \mathbf{R}\Hom_{A'}(\Lambda,G).
   $$
   In particular, the abutment in degree $m=p+q$ of the spectral sequence is 
   $$
   H^m(\mathbf{R}\Hom_{A'}(\Lambda,G)) =\Ext^m_{A'}(\Lambda,G).
   $$
   The exactness of $\pi$ gives a quasi-isomorphism $\pi(\Lambda\otimes^{\mathbf{L}}_{A'}A) \simeq E\otimes^{\mathbf{L}}_{A'}A$ and thus, by functoriality of $\pi$, a map of spectral sequences
   $$
   \L\{\Ext^p_A\L(\mathcal{T}\!\mathrm{or}_q^{A'}(\Lambda,A),G\R)\R\}_{p,q}\to \L\{\Ext^p_A\L(\mathcal{T}\!\mathrm{or}_q^{A'}(E,A),\pi(G)\R)\R\}_{p,q},
   $$
   where $E = \pi(\Lambda)$ is the coherent sheaf underlying $\Lambda$. Therefore, the forming of obstructions as defined in \cite[Def.~3.9]{He} commutes with the forgetful functor, i.e., under $\pi$, the obstruction to lifting $\Lambda$ along $A' \twoheadrightarrow A$ maps to the obstruction to lifting $E$. Now, by \cite[Thm.~4.5.3]{huyMS}, the determinant map
   \begin{align*}
   \mathrm{det} \colon \Coh_X^{\mathrm{pure}}(v) &\to \Pic(X)\\
   E &\mapsto \det(E)
   \end{align*}
   induces a natural map on obstruction spaces
   $$
   \Ext^2(E,E) \overset{\mathrm{tr}}{\to} H^2(X,\O_X) \simeq \Ext^2(\det(E),\det(E)),
   $$
   which also, in the appropriate sense, maps obstructions to obstructions. As $\Pic(X)$ is smooth, deformations of line bundles are unobstructed, so we find that 
   $$
   \Ext^2(\Lambda,\Lambda)_0=\mathrm{ker}\L(\Ext^2(\Lambda,\Lambda) \to \Ext^2(E,E) \to H^2(X,\O_X)\R)
   $$
   is also an obstruction space.
\end{proof}

\begin{proposition}[{\cite[Lemma 1.7]{He}}]
\label{proposition:CalculatingExt}
    Let $\Lambda = (E,V)$ be a coherent system with injective evaluation map and let $Q \coloneqq \mathrm{coker}(V \otimes \O_X \to E)$. Then there is a canonical map
    $$
    \Ext^i(Q,E) \to \Ext^i(\Lambda,\Lambda)
    $$
    which is an isomorphism for all $i \geq 2$ and surjective for $i=1$. If the canonical morphism $\Hom(\Lambda,\Lambda) \to \Hom(V,V)$ is surjective, then the above morphism is also an isomorphism for $i=1$.
\end{proposition}

\begin{corollary}
    Let $v=(r,c_1,\ch_2)$ be a Chern character and let $k\leq r$. The expected dimension of $\mathrm{Syst}_X^k(v)$ is given by 
    $$
    \mathrm{expdim}(\mathrm{Syst}_X^k(v)) \coloneqq \ext^1(\Lambda,\Lambda)-\ext^2(\Lambda,\Lambda) = \mathrm{expdim} (\mathrm{Spl}(v)) -k(k-\chi(v)).
    $$ 
\end{corollary}
    
\begin{proof}
    From Proposition~\ref{extsequence} for a point $(E,V) = \Lambda \in \mathrm{Syst}_X^k(v)$ we get the formula
    $$
    \ext^1(\Lambda,\Lambda) -\ext^2(\Lambda,\Lambda)= \hom(\Lambda,\Lambda)-\chi(E,E)-k(k-\chi(E)).
    $$
    Now, by simplicity, $\hom(\Lambda,\Lambda) = \hom (E,E) =  1$, and from the formula for the expected dimension $\mathrm{expdim}(\mathrm{Spl}(v)) = \ext^1(E,E)-\ext^2(E,E)$ we obtain the result.
\end{proof}

\subsection{A dimension bound of Aprodu--Farkas}
In the course of proving Green's conjecture for smooth curves on K3 surfaces, Aprodu and Farkas prove the following foundational result.

\begin{proposition}[{\cite[Prop. 2.4]{FarkasAproduGreenConj}}]\label{prop:AproduFarkas}
    Let $L$ be a globally generated line bundle on a K3 surface~$X$ and let $\mathcal{W}^r_d(|L|)$ be the relative Brill--Noether locus of line bundles over the linear system $|L|$. Suppose $\mathcal{W} \subseteq \mathcal{W}^r_d(|L|)$ is a dominating component, and $(C,A) \in \mathcal{W}$ is a general point such that $A$ is globally generated and $h^0(A)= r+1$. Then,
    $$
    \dim _AW^r_d(C)\leq \rho(g,r,d)+\dim \Hom(E_{C,A},E_{C,A})-1.
    $$
\end{proposition}

We employ the tangent-obstruction theory of $\Syst^k_X(v)'^\circ$ to generalize Proposition~\ref{prop:AproduFarkas} to surfaces satisfying $H^1(X,\O_X) = 0$. Moreover, we can prove the corresponding relative statement for all of $\mathcal{W}^r_d(|L|_{\mathrm{int}})$ (not only for dominating components), and the above discussion of He's tangent obstruction theory obviates an appeal to the deformation-theoretic argument in \cite{Pareschi}.
Let us begin with two simple observations.

\begin{lemma}
 \label{lem:numericsimplyupgrade}
     Let $X$ be a smooth projective surface and $|c_1|$ a non-empty linear system on~$X$ such that $-K_X.c_1 \geq \max(p_a(C),1)$ for $C \in |c_1|$. Then all integral curves $C \in |c_1|$ satisfy $-K_X|_C > 0$.
 \end{lemma}
 
 \begin{proof}
     This follows from Riemann--Roch. We have
     \[
     h^0(C,\omega_X^{\vee}|_C) \geq -K_X.c_1-p_a(C)+1 \geq 1,
     \]
     so $-K_X|_C \geq 0$. Moreover, $-K_X|_C \not\simeq \O_C$ because $\deg(-K_X|_C) = -K_X . c_1 \geq 1$.
 \end{proof}

\begin{lemma}
    \label{lemma:simplecoherentsystem}
    Let $\Lambda =(E,V)$ be a generically generated coherent system with $\dim V\leq \rk E$. Assume that $Q\coloneqq \mathrm{coker}(V\otimes\O_X\to E)$ is simple and $H^0(X,E^{\vee}) = 0$. Then $\Hom(\Lambda,\Lambda) \simeq \C$.
\end{lemma}

\begin{proof}
    Let $(f,g)\colon \Lambda\to\Lambda$ be an endomorphism of the coherent system $\Lambda$. Recall that this is given by the commutative diagram
\[\begin{tikzcd}
	{V\otimes \mathcal{O}_X} & E \\
	{V\otimes\mathcal{O}_X} & E
	\arrow[from=1-1, to=1-2]
	\arrow["g"', from=1-1, to=2-1]
	\arrow["f", from=1-2, to=2-2]
	\arrow[from=2-1, to=2-2].
\end{tikzcd}\]
As $(E,V)$ is generically generated, we can complete this to a diagram of short exact sequences
\[\begin{tikzcd}
	0 & {V\otimes \mathcal{O}_X} & E & Q & 0 \\
	0 & {V\otimes\mathcal{O}_X} & E & Q & 0
	\arrow[from=1-1, to=1-2]
	\arrow[from=1-2, to=1-3]
	\arrow["g"', from=1-2, to=2-2]
	\arrow[from=1-3, to=1-4]
	\arrow["f", from=1-3, to=2-3]
	\arrow[from=1-4, to=1-5]
	\arrow["{\lambda\cdot\mathrm{id}_Q}", from=1-4, to=2-4]
	\arrow[from=2-1, to=2-2]
	\arrow["i"', from=2-2, to=2-3]
	\arrow["p"', from=2-3, to=2-4]
	\arrow[from=2-4, to=2-5].
\end{tikzcd}\]
Since $Q$ is assumed to be simple, the induced endomorphism $Q\to Q$ is indeed of the form $\lambda\cdot\mathrm{id_Q}$. Consider now the endomorphism $(f',g')\colon \Lambda\to \Lambda$ with $f' = f-\lambda\cdot\mathrm{id}_E$ and $g' = g-\lambda\cdot\mathrm{id}_V$. Then $(f',g')$ induces the zero map on $Q$, in particular $p\circ f' = 0$. By exactness, $f'$ factors as $i\circ h$ for a map $h\colon E\to V\otimes \O_X$. As $H^0(X,E^{\vee}) = 0$ we find $h = 0$ and thus $f'=0$. It follows $f= \lambda\cdot\mathrm{id_E}$ and $g=\lambda\cdot\mathrm{id}_V$.
\end{proof}

 From now on, we consider a surface $X$ satisfying $H^1(X,\O_X) = 0$ together with a non-empty complete linear system $|c_1|$ on $X$. We denote by 
 $$
 \mathcal{W}^r_d(|c_1|) \coloneqq \left\{ i_\ast A \in \Spl_X(0, c_1, d - \frac{1}{2}c_1^2) \ \colon \ i \colon C \hookrightarrow X, \ C \in |c_1|, \ h^0(A) \geq r+1 \right\} 
 $$

 the relative Brill--Noether locus inside $\Spl_X(0,c_1,d-\frac{1}{2}c_1^2)$ over $|c_1|$, where $\mathcal{C} \subseteq X \times |c_1|$ is the universal curve. Over the locus of integral curves $|c_1|_{\mathrm{int}}\subset |c_1|$, the latter is isomorphic to the relative Brill--Noether locus of torsion-free rank one sheaves inside the compactified Jacobian $\overline{\Pic}^d(\mathcal{C}/|c_1|_{\mathrm{int}})$ (see, e.g., \cite[Prop. 2.2]{RelativeJacobian}). Indeed, any sheaf $F \in \mathrm{Spl}_X(0,c_1,d-\frac{1}{2}c_1^2)$ with $c_1 > 0$ is pure of dimension one. If $\mathrm{supp}(F)$ is integral and in class $c_1$, then $F$ is necessarily of rank one on its support, and purity implies that $F|_{\mathrm{supp}(F)}$ is torsion-free. Consequently, $F \simeq i_*A$ for some integral curve $i \colon C \hookrightarrow X$, $C \in |c_1|$, and $A \in \overline{\Pic}^d(C)$. Moreover, if $C$ is smooth, then $A$ is a line bundle. Let
 $$
 \widetilde{\mathcal{W}}^r_d(|c_1|)\subset \mathcal{W}^r_d(|c_1|)
 $$
 be the open sublocus consisting of those $(C,A) \in \mathcal{W}^r_d(|c_1|)$ with $h^0(C,A) = r+1$. Finally, 
 $$
 \widetilde{\mathcal{W}}^{r,\mathrm{bpf}}_d(|c_1|) \subset \widetilde{\mathcal{W}}^r_d(|c_1|)
 $$
 denotes the open sublocus of \emph{basepoint-free} sheaves.
 
 \begin{proposition}
 \label{prop:WRDopeninSyst}
     Let $|c_1|_{\mathrm{int}}\subseteq |c_1|$ be the open locus of integral curves. Then $\widetilde{\mathcal{W}}^{r,\mathrm{bpf}}_d(|c_1|_{\mathrm{int}})$ is isomorphic to an open subspace of $\mathrm{Syst}^{r+1}_X(r+1,c_1,\frac{1}{2}c_1^2-d)'^{\circ}$ via the map
     $$
     (C,A)\mapsto (E_{C,A},H^0(A)^{\vee}).
     $$
 \end{proposition}
 
 \begin{proof}
     Let $A$ be basepoint-free of degree $d$ on an integral curve $C\in |c_1|$ with $h^0(C,A) = r+1$. The Lazarsfeld--Mukai construction yields a vector bundle $E_{C,A}$ with $\ch(E_{C,A}) = (r+1,c_1,\frac{1}{2}c_1^2-d)$. The sequence
     $$
     0\to H^0(C,A)^{\vee}\to E_{C,A}\to i_*(A^{\vee}\otimes\omega_C\otimes\omega_X^{\vee}|_C)\to 0
     $$
     shows that the coherent system $(E_{C,A},H^0(A)^{\vee})$ is indeed generically generated. Moreover, the same sequence, after dualizing, shows $H^0(X,E_{C,A}^{\vee}) = 0$. Thus, Lemma \ref{lemma:simplecoherentsystem} applies, and setting $\Lambda_{C,A} \coloneqq (E_{C,A},H^0(A)^{\vee})$, we find $\Hom(\Lambda_{C,A},\Lambda_{C,A}) \simeq \C$. Therefore, we indeed have $\Lambda_{C,A}\in\Syst^{r+1}_X(r+1,c_1,\frac{1}{2}c_1^2-d)'^{\circ}$. Let $\mathcal{U}\subset \Syst^{r+1}_X(r+1,c_1,\frac{1}{2}c_1^2-d)'^{\circ}$ be the open sublocus of coherent systems $\Lambda  = (E,V)$ such that the quotient $Q_{\Lambda}\coloneqq \mathrm{coker}(V\otimes \O_X\to E)|_{\mathrm{supp}(Q_{\Lambda})}$ is a torsion-free rank~$1$ sheaf on an integral curve $C_{\Lambda}\in |c_1|_{\mathrm{int}}$. Then an inverse to the above map is given by
     \begin{align*}
         \mathcal{U}&\to \widetilde{\mathcal{W}}^{r,\mathrm{bpf}}_d(|c_1|_{\mathrm{int}})\\
         \Lambda &\mapsto Q_{\Lambda}^{\vee}\otimes\omega_{C_{\Lambda}}\otimes \omega_X^{\vee}|_{C_{\Lambda}}. \qedhere
     \end{align*}
 \end{proof}
 
 Thus, we can apply the tangent-obstruction theory developed at the beginning of this section to $\widetilde{\mathcal{W}}^{r,\mathrm{bpf}}_d(|c_1|_{\mathrm{int}})$ in order to find upper bounds on the local dimension. We obtain the following generalization of \cite[Prop. 2.4]{FarkasAproduGreenConj}.
 
 \begin{proposition}
 \label{prop:generalizedAproduFarkas}
     Let $X$ be a surface with $H^1(X,\O_X) = 0$ and $|c_1|$ a non-empty linear system on $X$. Let $C \in |c_1|$ be an integral curve with $g \coloneqq p_a(C) \geq 2$.     
     \begin{enumerate}[(i)]
         \item If $K_X \neq 0$, we assume in addition that $-K_X|_C > 0$. Let $A$ be a torsion-free sheaf of rank~$1$ and degree $d$ on $C$, basepoint-free with $h^0(A) = r+1$. Then
        $$
        \dim_{(C,A)} \mathcal{W}^r_d(|c_1|)\leq \rho(g,r,d)+ \dim|c_1|+\mathrm{dim}\Hom(E_{C,A},E_{C,A})-1.
        $$
        \item We assume that at least one of the following holds:
        \begin{enumerate}[(a)]
            \item $-K_X.c_1 \geq g$;
            \item $-K_X > 0$ and $-K_X.c_1 \neq 0$; 
            \item $K_X = 0$.
        \end{enumerate}
        Then, for the general $C \in |c_1|_{\mathrm{int}}$ and any torsion-free rank~$1$ sheaf $A$ of degree $d$ on $C$, basepoint-free with $h^0(A) = r+1$, we have
        $$
            \dim _AW^r_d(C) \leq \rho(g,r,d)+\dim \Hom(E_{C,A},E_{C,A})-1.
        $$
     \end{enumerate}
 \end{proposition}

 \begin{proof}
     The data given in $(i)$ is exactly the data of a point $(C,A)\in \widetilde{\mathcal{W}}^{r,\mathrm{bpf}}_d(|c_1|_{\mathrm{int}})$ with the added assumption $-K_X|_C>0$. Proposition \ref{prop:WRDopeninSyst} allows us to apply the tangent-obstruction theory developed for $\Syst^{r+1}_X(v)'^{\circ}$ to obtain dimension bounds for $\widetilde{\mathcal{W}}^{r,\mathrm{bpf}}_d(|c_1|_{\mathrm{int}})$. A calculation yields
     $$
     \mathrm{expdim}\Syst^{r+1}_X(r+1,c_1,\frac{1}{2}c_1^2-d)'^{\circ} = \rho(g,r,d)+\dim |c_1|.
     $$
     Now, for the coherent system $\Lambda_{C,A} \coloneqq (E_{C,A},H^0(A)^{\vee})$ and its quotient $Q_{C,A}\coloneqq \coker(H^0(A)^{\vee} \to E)$ we have that 
     $$
\Ext^2(\Lambda_{C,A},\Lambda_{C,A})\simeq \Hom(E_{C,A},Q_{C,A}\otimes\omega_X)
     $$
     is an obstruction space to deforming $\Lambda_{C,A}\in \Syst^{r+1}_X(r+1,c_1,\frac{1}{2}c_1^2-d)'^{\circ}$. The condition $-K_X|_C>0$ guarantees the existence of an injection $s\colon \omega_X|_C\hookrightarrow \O_C$, inducing an injective linear map
     \begin{equation}
     \label{eqn:A}
     \Hom(E_{C,A},Q_{C,A}\otimes \omega_X)\hookrightarrow \Hom(E_{C,A},Q_{C,A}).
     \end{equation}
     If $\omega_X \simeq \O_X$, the trace map $\mathrm{tr}\colon \Ext^2(\Lambda_{C,A},\Lambda_{C,A})\to H^2(X,\O_X)$ is nonzero. Hence, we have $\dim \Ext^2(\Lambda_{C,A},\Lambda_{C,A})_0 \leq \dim \Hom(E_{C,A},Q_{C,A})-1$. In case $K_X \neq 0$, injection (\ref{eqn:A}) can never be surjective. Indeed, any map in the image of (\ref{eqn:A}) cannot be surjective as it factors through the multiplication $s\colon \omega_X|_C\to \O_C$ but by definition there exists a surjective map in $\Hom(E_{C,A},Q_{C,A})$. Therefore, also in this case, $\dim \Ext^2(\Lambda_{C,A},\Lambda_{C,A}) \leq \dim \Hom(E_{C,A},Q_{C,A})-1$.
     We obtain the upper bound
     $$
     \dim_{(C,A)}\mathcal{W}^r_d(|c_1|)\leq \rho(g,r,d)+\dim|c_1| + \dim \Hom(E_{C,A},Q_{C,A})-1.
     $$
     As $A$ is a complete linear system, $H^0(X,E_{C,A}^{\vee}) = H^1(X,E_{C,A}^{\vee}) = 0$. Applying $\Hom(E_{C,A},-)$ to the short exact sequence
     $$
     0\to H^0(A)^{\vee}\otimes \O_X\to E_{C,A}\to Q_{C,A}\to 0
     $$
     yields $\Hom(E_{C,A},E_{C,A})\simeq \Hom(E_{C,A},Q_{C,A})$. In sum, we find
     $$
        \dim_{(C,A)}\mathcal{W}^r_d(|c_1|)\leq \rho(g,r,d)+\dim |c_1|+\dim \Hom(E_{C,A},E_{C,A})-1,
     $$
     as desired. For $(ii)$, if $-K_X>0$,  then $p_a(C)\geq 2$ and $C$ integral together with $-K_X.c_1 \neq 0$ already imply $-K_X|_C>0$. In the $\omega_X =\O_X$ case, as above, the trace map is nonzero. If $-K_X.c_1\geq p_a(C)$, the condition $-K_X|_C>0$ is also automatically satisfied for all integral curves in $|c_1|$ by Lemma \ref{lem:numericsimplyupgrade}. Thus, in all of those cases the conditions in $(i)$ are satisfied for all $(C,A)\in \widetilde{\mathcal{W}}^{r,\mathrm{bpf}}_d(|c_1|_{\mathrm{int}})$. Upper semicontinuity of fiber dimensions implies $(ii)$.
 \end{proof}
 
 For the above proof one can drop the assumption $p_a(C)\geq 2$ and replace $|c_1|_{\mathrm{int}}$ by the locus of integral curves that do not contain a fixed component of $-K_X$. For example, if $-K_X$ is effective and has no one-dimensional base components, no assumption on the curves is necessary. Another sufficient assumption is $g(C)\geq 1$ for all $C\in |c_1|$ and $h^0(X,\omega_X^{\vee})\geq 2$.

\subsection{Determinantal structure of the space of coherent systems}
Let $X$ be a smooth projective surface. One may realize $\Syst^k_X(v)$ as a determinantal scheme in a Grassmannian, possibly after an étale cover. We do this explicitly for the preimage of the open locus $\mathrm{Spl}(v)' \subseteq \mathrm{Spl}(v)$ of sheaves with vanishing $H^2$. This is especially useful when $\mathrm{Spl}(v)$ is unobstructed because then it implies that all irreducible components have at least the expected dimension.

As a special case of the tangent-obstruction theory explained at the beginning of this section, the tangent space of $\mathrm{Spl}(v)$ at $E$ is isomorphic to $\mathrm{Ext}^1(E,E)$, and an obstruction space is given by $\mathrm{Ext}^2(E,E)_0$ \cite[Theorem 4.5.3]{huyMS}. In particular, if $K_X = 0$, the trace-free obstructions vanish as the trace map is nonzero and $\Ext^2(E,E) \simeq \Hom(E,E)^\vee \simeq \C$. If $-K_X > 0$, any nonzero section $s \in H^0(\omega_X^\vee)$ yields an injective multiplication map 
$$
\Hom(E,E\otimes\omega_X)\hookrightarrow\Hom(E,E)\simeq \C \cdot \mathrm{id}_E.
$$
Every endomorphism $E \to E$ in the image must vanish somewhere, giving $\mathrm{Ext}^2(E,E) = 0$. Thus $\mathrm{Spl}(v)$ is smooth whenever $-K_X \geq 0$. Hirzebruch--Riemann--Roch yields the following dimension formula for $r>0$:
\begin{align}
\begin{split}
    \dim_E \mathrm{Spl}(v) &= \ext^1(E,E) = \begin{rcases}\begin{dcases}
        2 &\text{if } \hspace{1.21em} K_X = 0,\\
        1 &\text{if } -K_X > 0
    \end{dcases}\end{rcases}  -\chi(E,E) \\
    &= c_1^2-\chi(\O_X)r^2-2r \ch_2 + \begin{rcases}\begin{dcases}
        2 &\text{if } \hspace{1.21em} K_X = 0,\\
        1 &\text{if } -K_X > 0
    \end{dcases}\end{rcases}.
\end{split}
\end{align}
In particular, $\Spl(v)$ is smooth and equi-dimensional for all smooth projective surfaces $X$ with $-K_X \geq 0$.

In order to exhibit (the restriction of) $\mathrm{Syst}^k(v)$ as a degeneracy locus in a relative Grassmannian over $\mathrm{Spl}(v)'$, we assume, for simplicity, that there is a universal family $\mathcal{E}$ on $\mathrm{Spl}(v)' \times X$ (or else replace $\mathrm{Spl}(v)'$ by an étale cover). We have projection maps
\[\begin{tikzcd}
	{\mathrm{Spl}(v)' \times X} & X \\
	{\mathrm{Spl}(v)'}
	\arrow["q", from=1-1, to=1-2]
	\arrow["p"', from=1-1, to=2-1]
    .
\end{tikzcd}\]
Let $D\in |mH|$ for some $m\gg0$ be a sufficiently ample divisor on $X$ such that $R^i p_\ast \mathcal{E}(D) = 0$ for $i > 0$. Then $p_\ast \mathcal{E}(D)$ is locally free and commutes with arbitrary base change.

\begin{lemma}
    Assuming $m \gg 0$ in the choice of $D$, $H \coloneqq p_\ast (\mathcal{E}(D)|_{\Spl(v)' \times D})$ is a vector bundle on $\mathrm{Spl}(v)'$ and $R^1 p_\ast (\mathcal{E}(D)|_{\Spl(v)' \times D}) = 0$.
\end{lemma}

\begin{proof}
    By ampleness, for large $m$ we have $H^1(X,E(D)) = H^2(X,E(D)) = 0$. The divisor sequence yields
    $$
    \cdots \to H^1(X,E(D))\to H^1(X,E\otimes \O_D(D))\to H^2(X,E)\to H^2(X,E(D))\to \cdots
    $$
    and we find $H^1(X,E\otimes \O_D(D))\simeq H^2(X,E)=0$ and $H^2(X,E\otimes\O_D(D))\simeq H^3(X,E)=0$. It follows that $H^1(X,E\otimes \O_D(D)) =H^2(X,E) = 0$. In particular, $h^0(X,E\otimes \O_D(D))=\chi(E\otimes \O_D(D))$ does not depend on $E$ over $\mathrm{Spl}(v)'$, so $H$ restricts to a vector bundle on this open.
\end{proof}

Hence, $H = p_\ast (\mathcal{E}(D)|_{\Spl(v)' \times D})$ is also locally free and commutes with arbitrary base change. The divisor exact sequence
\begin{equation*}
    0 \to \O_X \to \O_X(D) \to \O_D(D) \to 0
\end{equation*}
gives rise to the exact sequence
\begin{equation*}
    0 \rightarrow p_\ast \mathcal{E} \rightarrow p_\ast \mathcal{E}(D) \rightarrow p_\ast(\mathcal{E}(D)|_{\Spl(v)' \times D}) \rightarrow R^1 p_\ast \mathcal{E} \rightarrow 0
\end{equation*}
on $\Spl(v)'$. Denote $F \coloneqq p_\ast \mathcal{E}(D)$; its fiber over $E$ is $F_{[E]}= H^0(X,E(D))$.
We can now form the Grassmannian $\rho \colon \mathrm{Gr}(k,F)\to \mathrm{Spl}(v)$ whose fiber over $E$ is $\rho^{-1}(E) = \mathrm{Gr}(k,H^0(E(D)))$. There is a universal subbundle $\mathcal{U}\subset \rho^* F$, and we obtain an evaluation map
$$
d' \coloneqq (\rho^* d)|_{\mathcal{U}} \colon \mathcal{U} \hookrightarrow \rho^* F \to \rho^* H
$$
on $\mathrm{Gr}(k,F)$. We define 
$$
D_0(d') \coloneqq \{V \in \mathrm{Gr}(k,F)\colon d'|_{[V]} = 0\}.
$$

\begin{lemma}
    The preimage of $\Spl(v)'$ in $\mathrm{Syst}^k(v)$, after an étale cover if necessary, is represented by $D_0(d')$, the Fitting support of $\ker(d')$.
\end{lemma}

\begin{proof}
    A map $T \to D_0(d')$ is the same as a map $\alpha \colon T \rightarrow \Gr(k,F)$ such that the composition $\alpha^\ast d' \colon \alpha^\ast \mathcal{U} \hookrightarrow \alpha^\ast \rho^\ast F \rightarrow \alpha^\ast \rho^\ast H$ is the zero map. Moreover, $\rho \circ \alpha$ gives a map to $\Spl(v)$ which corresponds to a flat family $\mathcal{E}_T$ of simple sheaves with Chern character $v$. Since the formation of $F$ and $H$ commute with arbitrary pullback, we have the exact sequence
    \begin{equation*}
        0 \rightarrow \mathrm{pr}_{T, \ast} \mathcal{E}_T \rightarrow \alpha^\ast \rho^\ast F \rightarrow \alpha^\ast \rho^\ast H,
    \end{equation*}
    where $\mathrm{pr}_T \colon T \times X \rightarrow T$. We therefore obtain a map $\alpha^\ast \mathcal{U} \hookrightarrow \mathrm{pr}_{T, \ast} \mathcal{E}_T$ from the universal property of the kernel, and this map is fiberwise injective because so is $\alpha^\ast \mathcal{U} \hookrightarrow \alpha^\ast \rho^\ast F$. But this fiberwise injection $\alpha^\ast \mathcal{U} \hookrightarrow \mathrm{pr}_{T, \ast} \mathcal{E}_T$ is nothing else than a $T$-family of coherent systems of type $(v,k)$ on $X$ by \cite[Lemma~28]{MarkmanBN}.
\end{proof}

By push-pull adjunction, a fiberwise injective coherent system $\iota \colon \mathcal{V}\hookrightarrow \mathrm{pr}_{T, \ast} \mathcal{E}_T$ also has an associated evaluation map $\mathrm{ev} \colon \mathrm{pr}_T^\ast \mathcal{V} \to \mathcal{E}_T$ on $T \times X$. The locus of generically generated coherent systems $\mathrm{Syst}^{k}(v)^{\circ} \subseteq \Syst^k(v)$ is the universal locus where this map has full rank at the generic point of $X$; it is open for every $k$. The same is true for the Brill--Noether locus of generically $k$-generated sheaves with exactly $k$ sections $\widetilde{BN}^k(v) \cap BN^k(v)^{\circ} \eqqcolon \widetilde{BN}^k(v)^{\circ} \subseteq \widetilde{BN}^k(v)$, see Definition~\ref{def:genericallygenerated}.

\begin{remark}
\label{rem:picardrankonestable}
    There are natural conditions implying that a sheaf is generically $k$-generated. For example, if $H$ is an ample class on $X$ and $c_1(E).H$ is minimal among all intersections of effective classes with $H$,\footnote{E.g., $\Pic(X) = \Z\cdot H$ and $c_1(E) = H$.} then any $H$-semistable sheaf $E$ satisfying $h^0(E)\geq \rk(E)$ is ggg. To see this, one can consider the saturation $F$ of $\mathrm{im}(H^0(X,E) \otimes \O_X \to E) \subseteq E$ and notice that the assumptions guarantee $c_1(F) = 0$ and thus $F\simeq \O_X^{\rk(F)}$. It follows that $\rk(F) = \rk(E)$ as all global sections of $E$ factor through $F$. Thus, if $v_1.H$ is minimal, then for all $k \geq r$ we have
    $$
    BN^k(v)^{\circ}\cap M(v) = BN^k(v)\cap M(v).
    $$
\end{remark}

\noindent Let us also note the following simple fact.

\begin{lemma}\label{lemma:GrassmannBundle}
    The forgetful map $\pi\colon\mathrm{Syst}^k(v)\to \mathrm{Spl}(v)$ admits a locally closed stratification into étale-local Grassmann-bundles. More precisely, consider the base change diagram
\[\begin{tikzcd}
	{\mathrm{Syst}^k(v)|_{\widetilde{BN}^{k'}(v)}} & {\mathrm{Syst}^k(v)} \\
	{\widetilde{BN}^{k'}(v)} & {\mathrm{Spl}(v).}
	\arrow[from=1-1, to=1-2]
	\arrow["{\pi_{k'}}"', from=1-1, to=2-1]
	\arrow["\pi", from=1-2, to=2-2]
	\arrow[from=2-1, to=2-2]
\end{tikzcd}\]
Then $\pi_{k'}$ is an étale-local fiber bundle with fiber $\pi_{k'}^{-1}(E) = \mathrm{Gr}(k,H^0(X,E)) = \mathrm{Gr}(k,k')$.
\end{lemma}

\noindent For $k \geq r$, $\mathrm{Syst}^k(v)^\circ$ lies entirely over the locus
$$
    BN^{r}(v)^{\circ}\coloneqq \{E\in \mathrm{Spl}(v)\colon E \text{ is ggg}\}
$$
of simple ggg sheaves. Whenever $X$ is a surface with $h^0(\omega_X) = 0$ or $\omega_X \simeq \O_X$, any simple ggg sheaf $E$ satisfies $H^2(X,E) = 0$, so $BN^r(v)^{\circ} \subseteq \mathrm{Spl}(v)'$. Stable sheaves also frequently satisfy $H^2(X,E) = 0$ even without assumptions on their global sections. For example, if $H$ is a polarization on $X$ and $r \cdot K_X.H < c_1(E).H$, then $h^2(E) = \hom(E, \omega_X) = 0$ for any $H$-stable sheaf $E$ by Serre duality.

\begin{corollary}\label{corollary:determinantal}
Assume $k \geq \max(0, \chi(v))$.
\begin{itemize}
    \item If $\Spl(v)'$ is equi-dimensional, then all irreducible components of the open $\pi^{-1}(\Spl(v)') \cap \Syst^k(v)$, if non-empty, have at least the expected dimension $\dim \Spl(v)' - k(k - \chi(v))$.
    \item  Similarly, if $H$ is a polarization on $X$ with $r \cdot K_X.H < c_1.H$, then all irreducible components of both $\pi^{-1}(M(v)) \cap \Syst^k(v)$ and $BN^k(v)^\circ \cap M(v)$, if non-empty, have at least the expected dimension $\dim M(v) - k(k - \chi(v))$.
    \item Finally, if $-K_X \geq 0$, then all irreducible components of $BN^k(v)$, if non-empty, have at least the expected dimension $\dim \Spl(v) - k(k - \chi(v))$. If in addition $k \geq r$, the same is true for $BN^k(v)^\circ$.
\end{itemize}
\end{corollary}

\begin{proof}
    We have identified $\pi^{-1}(\Spl(v)') \cap \Syst^k(v)$ with $D_0(d')$. By construction of $d'$, we have
    \begin{equation*}
        \chi(E) = h^0(E) - h^1(E) = h^0(E(D)) - h^0(E(D)|_D).
    \end{equation*}
    Hence, by equi-dimensionality of $\Spl(v)'$ and the principal ideal theorem, every irreducible component of the degeneracy locus $D_0(\rho^\ast d)$ has dimension at least
    \begin{gather*}
        \dim \Gr(k,F) - \rk(\rho^\ast \mathcal{U}) \cdot \rk(H) \\
        = \dim \Spl(v)' + k(h^0(E(D)) - k) - k \cdot h^0(E(D)|_D) = \dim \Spl(v)' - k(k - \chi(E)).
    \end{gather*}
    The same proof works for $\pi^{-1}(M(v)) \cap \Syst^k(v)$.
    The remaining statement follows because, for $-K_X \geq 0$ we have shown that $\Spl(v)$ is smooth and equi-dimensional, so in particular Cohen--Macaulay. Moreover, $BN^k(v)$ is precisely the locus of points $[E] \in \Spl(v)$ where $\ker(d([E]) \colon F([E]) \rightarrow H([E])) = H^0(E)$ has dimension at least $k$ or, equivalently, where $d([E])$ has rank at most $s \coloneqq \rk(F) - k$. By the general theory of determinantal varieties, the codimension of every irreducible component is then at most
    \begin{equation*}
        (\rk F - s)(\rk H - s) = k(k - (\rk F - \rk H)) = k(k - \chi(v)).
    \end{equation*}
    The same is true for $BN^k(v)^\circ$ which, for $k \geq r$, is open in $BN^k(v)$.
\end{proof}

Finally, we collect a few results on the connectedness and irreducibility of $\mathrm{Spl}(v)$ and $M(v)$. We will see in Section~\ref{BNTheorem} that the (ir)reducibility of $\mathrm{Spl}(v)$ is exactly what governs the (ir)reducibility of certain Brill--Noether loci. The same philosophy applies to the stable Brill--Noether loci $BN^k(v) \cap M(v) \subseteq M(v)$.
 
\begin{proposition}
\label{ConnectednessCollection}
    Let $X$ be a surface with $H^1(X,\O_X) = 0$ and $-K_X\geq 0$. Let $v=(r,c_1,\ch_2)\in H^*(X,\Q)$ be a Chern character.
    \begin{enumerate}[(i)]
        \item If $\ch_2\ll 0$, then $M(v)$ is irreducible \cite[Thm. 9.4.3]{huyMS}.
        \item If $-K_X$ is ample, then $\mathrm{Spl}(v)$ and thus $M(v)$ are irreducible \cite[Prop. 2.5]{YoshiokaEllipticSurfaces}.
        \item If $X$ is a K3 surface, $v$ primitive and $H$ a polarization generic with respect to $v$, then $M(v)$ is irreducible \cite[Thm. 6.2.5]{huyMS}.
        \item  If $X$ is a K3 surface with $\Pic(X) \simeq \Z\cdot H$ and $v$ primitive satisfying $\dim M(v) \leq  2$, then $\mathrm{Spl}(v) = M(v)$ and they are irreducible \cite[Prop. 3.14]{MukaiK3}.
    \end{enumerate}
\end{proposition}

\subsection{Towards a Brill--Noether theorem}
We now use Propositions~\ref{smoothnescrit} and~\ref{proposition:CalculatingExt} to obtain smoothness results on $\Syst^k(v)$. For this, we generalize a method of Yoshioka \cite[Lemma 5.117]{YoshiokaStringPartitionFunctions}.
To this end, let $X$ be a smooth projective surface and $E$ a coherent sheaf on $X$. The identification
$$
\mathrm{End}(\Ext^1(E,\O_X)) \simeq \Ext^1(E,\O_X\otimes \Ext^1(E,\O_X)^\vee)
$$
gives rise to the \emph{tautological extension}
$$
0 \to \O_X \otimes \Ext^1(E,\O_X)^\vee \to G \to E \to 0,
$$
corresponding to $\mathrm{id} \in \mathrm{End}(\Ext^1(E,\O_X))$.

\begin{theorem}\label{theorem:SystSmoothnessFork=r}
     Let $X$ be a smooth projective surface with $H^1(X,\O_X) = 0$. Let $v=(r,c_1,\ch_2)$ be a Chern character with $c_1$ indecomposable. Assume that $K_X = 0$ or $-K_X|_C > 0$ for all $C \in |c_1|$. Let $\Lambda= (E,V)\in \mathrm{Syst}^r(v)$ be a generically generated coherent system on $X$ with $E$ simple. Then we have $\Ext^2(\Lambda,\Lambda)_0 = 0$. In particular, $\mathrm{Syst}^r(v)^{\circ}$ is smooth.
\end{theorem}

As a consequence of Lemma~\ref{lem:numericsimplyupgrade}, the condition $-K_X|_C > 0$ for all $C \in |c_1|_{\mathrm{int}}$ (which is all of $|c_1|$ when $c_1$ is indecomposable) is for example implied by $-K_X.c_1 \geq p_a(C) \geq 1$.

\begin{proof}
Let $Q \coloneqq \mathrm{coker}(V \otimes \O_X \to E)$. We first show $\Hom(G,G) \simeq \C$. Setting $i = r + \ext^1(E,\O_X) = \rk(G)$, we have an exact sequence 
    $$
    0 \to \O_X^i \to G \to Q \to 0.
    $$
Indeed, the composition $f \colon G \to E \to Q$ is surjective and the kernel is an extension of trivial bundles, thus itself trivial using $H^1(X,\O_X)=0$. As $Q$ is a torsion-free rank~$1$ sheaf supported on the integral curve $C \in |c_1|$ with $c_1$ indecomposable, we may apply Simplicity Lemma~\ref{simplicitylemma1} to conclude that $G$ is simple unless $\Hom(G,\O_X) \neq 0$. To rule out the latter, we apply $\Hom(-,\O_X)$ to the sequence defining $G$ and use $H^1(X,\O_X)= 0$, obtaining
    $$
    0 \to \Hom(E,\O_X) \to \Hom(G,\O_X) \to \Ext^1(E,\O_X) \to \Ext^1(E,\O_X) \to \Ext^1(G,\O_X) \to 0.
    $$
By definition of $G$, the boundary map $\Ext^1(E,\O_X)\to \Ext^1(E,\O_X)$ is the identity. Since $E$ is simple and ggg, we have $\Hom(E,\O_X) = 0$, implying $\Hom(G,\O_X) = 0 = \Ext^1(G,\O_X)$, as desired. Hence, $G$ is simple.

Next, by Proposition~\ref{proposition:CalculatingExt}, $\Ext^2(\Lambda,\Lambda) \simeq \Ext^2(Q,E)$. Applying $\Hom(G,-)$ to the above short exact sequence defined by $f$, we obtain $\Hom(G,G) \simeq \Hom(G,Q)$. Serre duality together with the defining sequence of $G$ then yield
    $$
    \Ext^2(Q,E)^\vee \simeq \Hom(E, Q \otimes \omega_X) \subseteq \Hom(G, Q \otimes \omega_X).
    $$
In case $K_X=0$, the trace map $\Ext^2(\Lambda,\Lambda) \to H^2(X,\O_X) \simeq \mathbb{C}$ is nonzero and so it suffices to show that $\Ext^2(\Lambda,\Lambda)$ is one-dimensional because then the trace-free obstruction space vanishes. This follows from $\Hom(G,Q) \simeq \Hom(G,G) \simeq \mathbb{C}$. If $K_X \neq 0$ but $-K_X|_C>0$ for $C \coloneqq \supp(Q) \in |c_1|$ which, by indecomposability of $c_1$, is an integral curve, we may choose a nonzero section $s \in H^0(C,\omega_X^\vee|_C)$, inducing an injective multiplication map $s \colon Q \otimes \omega_X \hookrightarrow Q$. If $g \in \Hom(G,Q)$ is in the image of the injection $\Hom(G,Q \otimes \omega_X) \hookrightarrow \Hom(G,Q)$ induced by $s$, then $g$ necessarily vanishes where $s$ vanishes. In particular, $g$ is not surjective. But apart from the zero map, $\Hom(G,Q) = \C \cdot f$ consists solely of surjections, so $\Hom(G,Q \otimes \omega_X) = 0$ and thus $\Hom(E, Q \otimes \omega_X) = 0$.
\end{proof}

The following result shows that, for $k \leq r$, the smoothness of $\Syst^k(v)$ at a generically generated point $(E,V)$ is governed by the endomorphisms of its cokernel under weak assumptions on $c_1$.

\begin{theorem}
\label{theorem:SystSmoothnessFork<r}
    Let $X$ be a smooth projective surface with $H^1(X,\O_X) = 0$. Let $v=(r,c_1,\ch_2)$ be any Chern character with $c_1$ effective. Consider the forgetful map $\pi\colon\Syst^k(v)\to BN^k(v)$ for $k\leq r$. Then the following statements hold:
    \begin{enumerate}[(i)]
        \item If $k = r$ and $K_X = 0$, then the open of $\pi^{-1}\L(\widetilde{BN}^r(v)\R)$ consisting of generically generated coherent systems with simple cokernel is smooth of the expected dimension.
        \item If $k=r$ and $K_X \neq 0$, an open neighborhood of the sublocus of $\pi^{-1}\L(\widetilde{BN}^r(v)\R)$ consisting of generically generated coherent systems whose cokernel is a simple sheaf supported on an integral curve $C$ with $-K_X|_C>0$ is smooth of the expected dimension.
        \item If $k<r$, assume $-K_X \geq 0$. Then the open of $\pi^{-1}\L(\widetilde{BN}^k(v)\R)$ consisting of generically generated coherent systems whose cokernel is simple and torsion-free in a neighborhood of the base locus of $|-K_X|$ is smooth of the expected dimension.
     \end{enumerate}
\end{theorem}

\begin{proof}
    Let $\Lambda = (E,V)\in \pi^{-1}\L(\widetilde{BN}^k(v)\R)$ be a generically generated coherent system in one of the loci considered in $(i)$, $(ii)$ and $(iii)$, respectively. Set $Q \coloneqq \mathrm{coker}(V \otimes \O_X \to E)$.
    We first claim that it suffices to show $\hom(E,Q) = 1$. Indeed, in each case this implies $\ext^2(\Lambda,\Lambda) = \hom(E,Q \otimes \omega_X) = 0$. This can be seen as follows:
    
    If $K_X = 0$, we may argue as in the proof of Theorem~\ref{theorem:SystSmoothnessFork=r}.
    If $K_X \neq 0$ and $k=r$, assume $C \coloneqq \mathrm{supp}(Q)$ is an integral curve with $-K_X|_C>0$. Then there is an injection $s \colon \omega_X|_C \hookrightarrow \O_C$, and since $Q$ is pure by Lemma~\ref{nosubskyscraper}, we obtain the injection $\Hom(E,s) \colon \Hom(E,Q\otimes\omega_X) \hookrightarrow \Hom(E,Q)$. The image of any map $f \colon E \to Q \otimes \omega_X$ under $\Hom(E,s)$ is not surjective because it vanishes where $s$ does. But if $\Hom(E,Q) \simeq \C$, it is generated by the surjection $E \twoheadrightarrow Q$, implying $f = 0$, so $\hom(E, Q \otimes \omega_X) = 0$.
    Finally, if $-K_X > 0$ and $k<r$, consider any nonzero section $s \in H^0(\omega_X^\vee)$, inducing a map $s \colon Q \otimes \omega_X \rightarrow Q$ with kernel $Q_s \subseteq Q \otimes \omega_X$. This map $s$ further induces the linear map $\Hom(E,s) \colon \Hom(E, Q \otimes \omega_X) \rightarrow \Hom(E,Q)$ whose kernel is $\Hom(E,Q_s)$. Again, since $\Hom(E,Q)$ is generated by the surjection $E \twoheadrightarrow Q$, the image of any $f \in \Hom(E,Q \otimes \omega_X)$ in $\Hom(E,Q)$ vanishes along $V(s)$, hence is not surjective. This forces the image of $f$ to be the zero map, meaning that $f$ factors through $Q_s$. As this is true for every $s \in H^0(\omega_X^\vee)$, $f$ factors through $\bigcap_{s \in H^0(\omega_X^\vee)} Q_s$ which, by assumption, is zero. Therefore, $f = 0$, proving the claim.

     \noindent The rest of the argument follows from the short exact sequence
     $$
        0\to V\otimes \O_X\to E\to Q\to 0.
    $$
    As $h^0(X,E) = k = \dim V$ and $H^1(X,\O_X) = 0$, we find $h^0(X,Q) = 0$. Applying $\Hom(-,Q)$ to the sequence yields $\Hom(E,Q) \simeq \Hom(Q,Q) \simeq \C$, as desired.
\end{proof}

In special cases, one can give numerical conditions under which any generically generated coherent system has a simple cokernel. However, these conditions are in general quite strong and require Picard rank~$1$, see, e.g., \cite[Lemma 2.1]{YoshiokaMukaiReflections}. For $k=r$, however, this is implied by $c_1$ being indecomposable. Indeed, then the cokernel $Q$ is a torsion-free rank-one sheaf on an integral curve and is thus automatically simple. In Theorem~\ref{theorem:SystSmoothnessFork<r}(ii), the assumption $-K_X|_C>0$ for all integral curves in $|c_1|$ can be replaced by the global assumption $-K_X.c_1\geq \max(p_a(C),1)$, see Lemma \ref{lem:numericsimplyupgrade}.

\section{Brill--Noether Theorem for Surfaces of Irregularity Zero}
\label{BNTheorem}
In this section we make explicit how the Lazarsfeld--Mukai construction provides a translation between the relative Brill--Noether loci of line bundles on curves in a fixed linear system, on the one hand, and the Brill--Noether loci in the moduli space of simple ggg sheaves on the surface on the other. The correspondence itself works for any smooth projective surfaces with $H^1(X,\O_X) = 0$. Under further effectivity assumptions on $-K_X$, we may use the deformation theory of $\mathrm{Syst}^k(v)$ to re-prove results in the Brill--Noether theory of curves and apply the framework to obtain new results for Brill--Noether loci on the surface.

Recall that $\mathrm{Syst}^k(v)^{\circ}$ denotes the moduli space of coherent systems $(E,V)$ of type $(v,k)$ with $E$ simple and $V \otimes \O_X \to E$ of full rank at the generic point, and that $BN^k(v)^{\circ}$ denotes the Brill--Noether locus of generically $k$-generated sheaves. In particular, if $k \geq r$, $BN^k(v)^{\circ}$ is the locus of ggg sheaves in $BN^k(v)$.
We consider the isomorphism
\begin{align*}
    \omega\colon \mathrm{Spl}(v)&\to \mathrm{Spl}(\omega(v)),\quad 
    E\mapsto E\otimes \omega_X.
\end{align*}
We also denote by $\omega$ its action on Chern characters, 
$$\omega(v)\coloneqq v\cdot \ch(\omega_X) = v+(0,rK_X,c_1.K_X+\frac{r}{2}K_X^2).
$$
Clearly, if $X$ is K3, $\omega = \mathrm{id}$. In general, $\omega$ does not preserve Brill--Noether loci.

\begin{remark}
    It is important to note that, as $\omega\colon \mathrm{Spl}(v)\to \mathrm{Spl}(\omega(v))$ is an isomorphism, we can prove connectedness, irreducibility, smoothness and dimension statements for 
    $$
    BN^k(\omega(v))\subset \mathrm{Spl}(\omega(v))
    $$
    by proving them for 
    $$    \omega^{-1}\L(BN^k(\omega(v))\R)\subset \mathrm{Spl}(v),
    $$
    even if the latter is not a Brill--Noether locus.
\end{remark}

The following lemma can be understood as a globalization of the Lazarsfeld--Mukai construction $(C,A,V)\mapsto (E_{C,A,V},V^{\vee})$.

\begin{lemma}
\label{lemma:QuotientMapSurjective}
    Let $v=(r,c_1,\ch_2)$ be a Chern character with $c_1$ indecomposable. Set $v'\coloneqq \omega(0,c_1,\ch_2)$. The map
    \begin{align*}
        \phi_v\colon \Syst^r(v)^{\circ}&\to \omega^{-1}\L(BN^{\chi(v')+r}(v')\R) \subseteq \Spl(0,c_1,\ch_2)\\
        (E,V)&\mapsto\mathrm{coker}(V\otimes\O_X\to E)
    \end{align*}
    is well-defined and surjective. Moreover, for any $Q \in \omega^{-1}\L(BN^{\chi(v')+r}(v')\R)$ the fiber $\phi_v^{-1}(Q)$ is irreducible of dimension 
    $$
    \dim \phi_v^{-1}(Q) = rh^1(Q\otimes\omega_X)-r^2\geq 0.
    $$
\end{lemma}

\begin{proof}
    Let $(E,V)\in \Syst^r(v)^{\circ}$ and $Q_{E,V}\coloneqq \phi_v(E,V)$ its image. Then $Q_{E,V}$ is supported on an integral curve and Proposition \ref{nosubskyscraper} shows it is the pushforward of a torsion-free sheaf of rank one on the curve and is thus simple. The sequence
    $$
        0 \to V \otimes \O_X \to E \to Q_{E,V} \to 0
    $$
    shows $\ext^1(Q_{E,V},\O_X^r) \geq r^2$. In particular, by Serre duality, $h^1(X,Q_{E,V}\otimes \omega_X)\geq r$ and thus 
    $$
    h^0(X,Q_{E,V}\otimes \omega_X)\geq \chi(Q_{E,V}\otimes \omega_X)+r = \chi(v')+r.
    $$
    It follows that $Q_{E,V}\in \omega^{-1}\L(BN^{\chi(v')+r}(v')\R)$, so the map is well-defined.
    
    \noindent To show surjectivity, let $Q \in \omega^{-1}\L(BN^{\chi(v')+r}(v')\R)$. Then, by definition, $h^1(X,Q\otimes \omega_X)\geq r$, so again $\ext^1(Q,\O_X^r) \geq r^2$ by Serre duality. This implies the existence of an extension $0 \to \O_X^r \xrightarrow{f} E \to Q \to 0$ that does not split off a trivial direct summand. By Lemma~\ref{simplicitylemma1}, such an $E$ is automatically simple. Set $V \coloneqq \mathrm{Im}(H^0(f))$. Then $(E,V) \in \mathrm{Syst}^r(v)^{\circ}$ defines a preimage of $Q$ under $\phi_v$.

    \noindent
    For the final statement, let $Q \in \omega^{-1} \L( BN^{\chi(v')+r}(v') \R)$. We have $r h^1(X,Q \otimes \omega_X) - r^2 \geq 0$. The fiber over $Q$ parametrizes extensions $\Xi \in \Ext^1(Q,\O_X^r)$ such that the extension sheaves are simple. By Lemma~\ref{simplicitylemma1}, the extensions that are not simple are exactly the ones splitting off a trivial direct summand, corresponding to the images of $\Ext^1(Q,\O_X^{r-1}) \hookrightarrow \Ext^1(Q,\O_X^r)$ given by direct sum with $\O_X$. They can also be viewed as the linear maps in $\Ext^1(Q,\O_X^r) \simeq \Ext^1(Q,\O_X) \otimes \C^r \simeq \Hom_{\C}(\C^r, \Ext^1(Q,\O_X))$ of rank $< r$. It follows that
    $$
    \phi_v^{-1}(Q) \simeq \Ext^1(Q,\O_X^r)^\circ / \mathrm{GL}(r) \simeq \Gr(r, \Ext^1(Q, \O_X)),
    $$
    where $\Ext^1(Q,\O_X^r)^\circ$ denotes the locus of full rank linear maps. Irreducibility follows, and so does the dimension formula using $\ext^1(Q,\O_X^r) = r h^1(Q\otimes\omega_X)$.
\end{proof}

\begin{remark}
In fact, the quotient map $(E,V)\mapsto \mathrm{coker}(V\otimes\O_X\to E)$ can be defined on the level of stacks for any generically generated coherent system $\Lambda = (E,V)$ without assumptions on $c_1$. However, it is not clear in general when the quotient will be simple or how to characterize the fibers.
\end{remark}

Similar arguments show:
\begin{corollary}
\label{cor:quotientisfiberbundle}
    The map $\phi_v$ restricts to an étale-local fiber bundle over 
    $$    \omega^{-1}\L(\widetilde{BN}^k(v')\R)\subset \omega^{-1}\L(BN^{\chi(v')+r}(v')\R)
    $$
    for $k\geq \chi(v')+r$. The forgetful map $\pi\colon \mathrm{Syst}^k(v)^{\circ}\to BN^k(v)^{\circ}$ restricts to an étale-local fiber bundle over the locally closed Brill--Noether loci $\widetilde{BN}^{k'}(v)^{\circ}\subset BN^{k}(v)^{\circ}$ for $k'\geq k$.
\end{corollary}

The main tool of this section is the following correspondence for $c_1$ indecomposable:

\begin{equation}
\tag{\(\boldsymbol{\ast}\)}
\label{eqn:correspondence}
\begin{tikzcd}
	{\mathrm{Syst}^r(v)^{\circ}} & {\omega^{-1}\L(BN^{\chi(v')+r}(v')\R)} \\
	{BN^r(v)^{\circ}}
	\arrow["{\phi_v}", from=1-1, to=1-2]
	\arrow["\pi"', from=1-1, to=2-1]
\end{tikzcd}
\end{equation}

\begin{lemma}
\label{lemma:correspondencetranslate}
    Assume $H^1(X,\O_X) = 0$. Let $\pi\colon \Syst^r(v)^{\circ}\to \mathrm{Spl}(v)$ be the forgetful map. Then,
    $$
    \pi^{-1}\L(\omega^{-1}\L(\widetilde{BN}^{k'+rh^0(\omega_X)}(\omega(v))\R)^{\circ}\R) = \phi_v^{-1}\L(\omega^{-1}\L(\widetilde{BN}^{k'}(v')\R)\R)\subset \mathrm{Syst}^r(v)^{\circ}
    $$
    for all $k'\geq 0$.
\end{lemma}

\begin{proof}
   The forgetful map $\pi$ factors as $\pi\colon \Syst^r(v)^{\circ}\twoheadrightarrow BN^r(v)^{\circ}\hookrightarrow \mathrm{Spl}(v)$. An element 
   $$
   (E,V)\in \pi^{-1}\L(\omega^{-1}\L(\widetilde{BN}^{k'+rh^0(\omega_X)}(\omega(v))\R)^{\circ}\R)
   $$
   is a generically generated coherent system such that $h^0(X,E\otimes \omega_X) = k'+rh^0(\omega_X)$. An element
   $$
   (E,V)\in \phi_v^{-1}\L(\omega^{-1}\L(\widetilde{BN}^{k'}(v')\R)\R)
   $$
   is a generically generated coherent system such that the cokernel $Q_{E,V}=\coker(V\otimes \O_X\to E)$ satisfies $h^0(X,Q_{E,V}\otimes\omega_X)=k'$. From the exact sequence
    $$
    0\to V\otimes\O_X\to E\to Q_{E,V}\to 0
    $$
    and $H^1(X,\O_X) = 0$, we deduce $h^0(X,E\otimes\omega_X) = h^0(X,Q_{E,V}\otimes \omega_X)+rh^0(X,\omega_X)$, proving the claim.
\end{proof}

\begin{remark}
\label{remark:transportresults}
    Lemma~\ref{lemma:correspondencetranslate} illustrates that knowing independent results for $BN^k(0,c_1,\ch_2)$ yields corresponding results for $BN^{k+rh^0(\omega_X)}(\omega(v))$. For instance, one could translate Brill--Noether statements for line bundles on special linear systems of plane curves of degree $d$ to Brill--Noether statements for sheaves on degree $d$ hypersurfaces in $\P^3$ and vice versa.
\end{remark}

If $c_1$ is indecomposable, the Brill--Noether loci $BN^k(v)^{\circ}$ are trivial for $k \geq r$ if $-K_X>0$. More precisely, given a generically globally generated sheaf $E$ on a regular surface $X$ with $-K_X>0$ and $c_1(E)$ indecomposable, the cokernel $Q$ of any injection $\O_X^r\hookrightarrow E$ is a torsion-free rank~$1$ sheaf on an integral curve $C$ of arithmetic genus $\leq 1$ by Lemma~\ref{lemma:CurvesOnDelPezzo}. In particular, if $Q \not\simeq \O_C$, then $h^0(X,Q) = \chi(Q)$ and thus $h^0(X,E) =r+\chi(Q) = \chi(E)$. It follows that $BN^k(v)^{\circ} = \emptyset$ for $k > \chi(v)$ and $BN^k(v)^{\circ} = \widetilde{BN}^{\chi(v)}(v)^{\circ}  =\{E\in \Spl(v)\colon E\text{ is ggg}\}$ for all $\chi(v) \geq k \geq r$.

\begin{proposition}
\label{connectednesscompactifiedjaobian}
    Let $X$ be a smooth projective surface with $H^1(X,\O_X) = 0$ and $-K_X\geq 0$. Let $v=(0,c_1,\ch_2)$ be a Chern character with $c_1$ indecomposable. The number of connected components of the Brill--Noether loci $BN^k(v)$ in the relative compactified Jacobian over $|c_1|$ is bounded above by the number of connected components of $\mathrm{Spl}(w)$ for $w=(k-\chi(v),c_1,\ch_2-c_1.K_X)$. In particular, if $\mathrm{Spl}(w)$ is connected, then so is $BN^k(v)$.
\end{proposition}

\begin{proof}
    The case $k=\chi(v)$ is clear because $BN^{\chi(v)}(v) = \mathrm{Spl}(v)$ and 
    $$
    \omega^{-1}\colon \Spl(v)\to \Spl(0,c_1,\ch_2-c_1.K_X)
    $$
    is an isomorphism. Let $k=\chi(v)+k'$ with $k'>0$. Let $w = (k',c_1,\ch_2-c_1.K_X)$ and consider the correspondence
    \[\begin{tikzcd}
	{\mathrm{Syst}^{k'}(w)^{\circ}} & {\omega^{-1}\L(BN^k(v)\R)} \\
	{BN^{k'}(w)^{\circ}}
	\arrow["{\phi_{w}}", from=1-1, to=1-2]
	\arrow["\pi"', from=1-1, to=2-1].
\end{tikzcd}\]
It suffices to show $\chi(w)\geq k'$. Indeed, then $BN^{k'}(w) = \mathrm{Spl}(w)$ and $BN^{k'}(w)^{\circ}\subset \mathrm{Spl}(w)$ is open and thus has at most as many connected components as $\mathrm{Spl}(w)$ (note that $\mathrm{Spl}(w)$ is smooth and thus all connected components are irreducible). The statement then follows, as $\pi$ is proper with connected fibers and both $\pi$ and $\phi_w$ are surjective.

Now, using the Hirzebruch--Riemann--Roch formula and $\mathrm{td}(X) = (1,-\frac{1}{2}K_X,\chi(\O_X))$, we find
$$
\chi(w) = \chi(v)+\chi(k',0,-c_1.K_X) = \chi(v)+k'\chi(\O_X)-c_1.K_X.
$$
Using $k'+\chi(v) = k$, the condition $\chi(w)\geq k'$ is equivalent to
$$
(2-\chi(\O_X))\chi(v)-c_1.K_X\geq k(1-\chi(\O_X)).
$$
If $X$ is K3, this condition is vacuous. For $-K_X > 0$, Lemma~\ref{lemma:CurvesOnDelPezzo} shows that all curves in an indecomposable linear system have arithmetic genus at most~$1$ and hence the only non-empty Brill--Noether locus is the entire $\mathrm{Spl}(0,c_1,\ch_2)$.
\end{proof}
 
We now establish the central result of this section, the equi-dimensionality and irreducibility of connected components of the Brill--Noether loci. From here on, it is crucial that $-K_X$ be effective so that we can invoke the smoothness results from Section~\ref{coherentsystems}.

\begin{proposition}
\label{proposition:main}
 Let $X$ be a K3 surface and $v=(r,c_1,\ch_2)$ with $c_1$ indecomposable. Then, for all $k \geq \mathrm{max}(\chi(v),r)$, the Brill--Noether loci $\widetilde{BN}^k(v)^{\circ}$ are smooth and of the expected dimension
 \begin{equation}
 \label{eqn:dimensionformula}
\dim(\widetilde{BN}^k(v)^{\circ}) = \dim \mathrm{Spl}(v) -k(k-\chi(v)).     
 \end{equation}
 In particular, all connected components are irreducible. Moreover, the number of connected components is bounded above by the number of connected components of 
 $
 \mathrm{Spl}(k-r-\ch_2,c_1,\ch_2).
 $
\end{proposition}

\begin{proof}
We first consider the case $r = 0$, so $v=(0,c_1,\ch_2)$ with $c_1$ indecomposable. If $k = \chi(v)$, the statement is clear as $BN^{\chi(v)}(v) =\mathrm{Spl}(v)$. As in the proof of Proposition~\ref{connectednesscompactifiedjaobian}, let $k=\chi(v)+k'$ with $k'>0$ and define $w=(k',c_1,\ch_2)$. With this notation, the surjective map
    $$
    \phi_w\colon \Syst^{k'}(w)^{\circ}\to BN^k(v),
    $$
    has fiber dimension $\dim\phi_w^{-1}(L)  =k'h^1(X,L)-k'^2 $. For $L \in \widetilde{BN}^k(v)$ we have $h^1(X,L) = k'$ and thus $\dim \phi_w^{-1}(L) = 0$ , i.e., $\phi_w$ is finite surjective over $\widetilde{BN}^k(v)$. We apply Theorem~\ref{theorem:SystSmoothnessFork=r} to conclude that $\Syst^{k'}(w)^{\circ}$ is smooth with every connected component of dimension 
    $$
    \dim \Syst^{k'}(w)^{\circ} = \dim \mathrm{Spl}(w)-k'(k'-\chi(w)).
    $$
    A computation shows
    \begin{equation*}
    \dim \widetilde{BN}^k(v) = \dim \Syst^{k'}(w)^{\circ} = \dim \mathrm{Spl}(w)-k'(k'-\chi(w))= \dim \mathrm{Spl}(v) -k(k-\chi(v)),
    \end{equation*}
    which is the expected dimension. Now, by smoothness of $\Syst^{k'}(w)^{\circ}$ and the fiber bundle structure of $\phi_w$, all connected components of $\widetilde{BN}^k(v)$ are smooth. The bound on the connected components follows from Proposition \ref{connectednesscompactifiedjaobian} and because $v'=\omega(v)$ for $r=0$.

We now turn to the case $v=(r,c_1,\ch_2)$ with $r>0$. Notice that for $k\geq r$, the forgetful map $\pi\colon \Syst^r(v)^{\circ}\to \Spl(v)$ surjects onto  $\widetilde{BN}^k(v)^{\circ}$ and has constant fiber dimension over this locus. 
Lemma~\ref{lemma:correspondencetranslate} implies
\begin{equation*}
    \dim \widetilde{BN}^{k}(v)^{\circ}
    = \dim \widetilde{BN}^{k-r}(v')+\dim \phi_v-\dim \pi,
\end{equation*}
 where $\dim \phi_v$ and $\dim \pi$ denote the relative dimension of $\phi_v$ and $\pi$, respectively. Here, $v'$ is a rank~$0$ Chern character, so the first case above yields the expected dimension formula. For the relative dimensions we use $\dim \pi = rk'$ and Lemma \ref{lemma:QuotientMapSurjective}. The dimension formula (\ref{eqn:dimensionformula}) then follows from a simple calculation. The smoothness of $\widetilde{BN}^{k}(v)^{\circ}$  follows from the smoothness of $\widetilde{BN}^{k-r}(v')$ in the rank~$0$ case above together with the fiber bundle structure of $\phi_v$ over $\widetilde{BN}^{k-r}(v')$. The bound on the number of connected components follows from the bound on the number of connected components of $BN^{k-r}(v')$ provided by Proposition~\ref{connectednesscompactifiedjaobian}.
\end{proof}

To conclude irreducibility and dimension results for $BN^k(\omega(v))$ from $\widetilde{BN}^k(\omega(v))$, we observe the following.

\begin{proposition}
\label{proposition:ClosureOfBNTilde}
    In the setting of Proposition~\ref{proposition:main}, for all $k\geq r$ we have 
    $$
    \overline{\widetilde{BN}^k(v)^{\circ}} = BN^k(v)^{\circ},
    $$
    where the closure is taken in $BN^k(v)^{\circ}$.
\end{proposition}

\begin{proof}
It suffices to show that $\widetilde{BN}^k(v)^{\circ} \cap Z \neq \emptyset$ for every irreducible component $Z$ of $BN^k(v)^{\circ}$. For the sake of contradiction, assume $\widetilde{BN}^k(v)^{\circ} \cap Z = \emptyset$. Then, for all $F \in Z$, we have $h^0(X,F) \geq k+1$ and thus $Z \subseteq BN^{k+s}(v)^{\circ}$ for some minimal $s>0$ (this uses $k \geq r$, to guarantee the sheaves are ggg). In particular, $\widetilde{BN}^{k+s}(v)^{\circ} \cap Z \eqqcolon U$ is dense in $Z$ and a component of $\widetilde{BN}^{k+s}(v)^{\circ}$. Now, from Proposition \ref{proposition:main}, we know that $\widetilde{BN}^{k+s}(v)^{\circ}$ is equidimensional and of the expected dimension which is strictly smaller than the expected dimension of $BN^k(v)^{\circ}$. It follows that $Z \subseteq BN^k(v)^{\circ}$ is of strictly smaller dimension than the expected dimension. But all components of $BN^k(v)^{\circ}$ have at least the expected dimension by Corollary~\ref{corollary:determinantal}, a contradiction.
\end{proof}

\begin{theorem}[Brill--Noether Theorem]
\label{BNconnected}
    Let $X$ be a smooth projective surface with $H^1(X,\O_X) = 0$. Let $v=(r,c_1,\ch_2)$ be a Chern character with $c_1 > 0$.
    \begin{enumerate}[(i)]
        \item Assume that either $K_X = 0$ or $-K_X.c_1 \geq \max(p_a(C),1)$. Then the open sublocus of $\widetilde{BN}^r(v)^{\circ}$ consisting of sheaves $E$ with $\mathrm{coker}(\mathrm{ev}_E)$ supported on an integral curve $C \in |c_1|$ is smooth of the expected dimension, or empty.
        \item If $0 \leq k < r$, assume $-K_X \geq 0$. Then the open sublocus of $\widetilde{BN}^k(v)^{\circ}$ consisting of sheaves $E$ with $\mathrm{coker}(\mathrm{ev}_E)$ simple and torsion-free in a neighborhood of the base locus of $|-K_X|$ is smooth of the expected dimension, or empty.
        \item If $k \geq \mathrm{max}(\chi(v),r)$, assume $c_1$ is indecomposable and $X$ is a K3 surface. Then the Brill--Noether loci $BN^k(v)^{\circ}$ 
        are normal and all connected components are irreducible of the expected dimension. Moreover, the number of connected components of $BN^k(v)^{\circ}$ is bounded above by the that of $\mathrm{Spl}(k-r-\ch_2,c_1,\ch_2)$.
        \item In case (iii), denote $g\coloneqq \frac{1}{2}c_1^2+1$ and $d\coloneqq \ch_2+\frac{1}{2}c_1^2$. Then, if the Brill--Noether number $\rho(g,k-r-1,d)$ is nonnegative, $BN^k(v)$ is nonempty and contains a ggg sheaf.
    \end{enumerate}
\end{theorem}

\begin{proof}
For $(i)$ and $(ii)$, since the forgetful map $\pi\colon \Syst^k(v)^{\circ}\to BN^k(v)^{\circ}$ restricts to a smooth fiber bundle over $\widetilde{BN}^k(v)^{\circ} \subseteq BN^k(v)^{\circ}$, the statement is an immediate consequence of Theorem~\ref{theorem:SystSmoothnessFork<r}.
For $(iii)$ we know by Propositions \ref{proposition:main} and \ref{proposition:ClosureOfBNTilde} that $BN^k(v)^{\circ}$ has the expected codimension. By the theory of determinantal varieties and the results of Section \ref{coherentsystems}, it follows that $BN^k(v)^{\circ}$ is normal and thus every connected component is irreducible.
For the non-emptiness statement $(iv)$, assume $\rho(g,k-r-1,d) \geq 0$. Then, by classical Brill--Noether theory, for any curve $C$ in $|c_1|$, there exists a torsion-free rank~$1$ sheaf $L$ of degree $d$ on $C$ which satisfies $h^0(C,L) \geq k - r$. By definition $i_*L\in BN^{k-r}(v')$ and from Lemma~\ref{lemma:correspondencetranslate} we get the non-emptiness of $BN^{k}(v)^{\circ}$.
\end{proof}

\begin{remark}
\label{rem:specializingmainthm}
    If $-K_X \geq 0$ and $c_1$ indecomposable, the open locus of ggg sheaves in $\widetilde{BN}^r(v)$ is smooth of the expected dimension. Moreover, if $X$ has Picard rank~$1$ and one considers the stable Brill--Noether loci $BN^k(v) \cap M(v)$, the assumption of being ggg is automatic for $k \geq r$ by Remark~\ref{rem:picardrankonestable}. In case $k \leq r$, also the condition that $\coker(H^0(E) \otimes \O_X \to E)$ be simple can be guaranteed, see, e.g., \cite[Lemma 2.1]{YoshiokaMukaiReflections}. In those cases, the above Theorem recovers the analogous results in \cite{YoshiokaBNTHeory} and \cite{Leyenson}.
\end{remark}

\begin{remark}
    Propositions~\ref{proposition:main} and~\ref{proposition:ClosureOfBNTilde} as well as the proof strategy of part $(iii)$ in Theorem \ref{BNconnected} do not heavily rely on the assumption that $X$ is a K3 surface. Our proofs for these results only use that $c_1$ is indecomposable and $-K_X$ is effective but it turns out that the only such case for which part $(iii)$ in the Theorem is non-trivial is that of K3 surfaces (see the discussion before Proposition~\ref{connectednesscompactifiedjaobian}). Nonetheless, if one had a proof of Theorem~\ref{theorem:SystSmoothnessFork=r} and an extension of the Simplcitiy Lemma~\ref{simplicitylemma1} without assuming $c_1$ indecomposable, then one could upgrade Propositions~\ref{connectednesscompactifiedjaobian}, \ref{proposition:main}, and Theorem~\ref{BNconnected} $(iii)$ and obtain non-trivial results for regular surfaces satisfying $-K_X >0$. The key tool is Lemma~\ref{lemma:correspondencetranslate}.
\end{remark}

\section{Applications to Pencils on Curves}
\label{ApplicationToCurves}
We discuss applications to Brill--Noether loci of complete and non-complete pencils on curves lying on surfaces $X$ with $H^1(X,\O_X) =0$. The main idea is to apply the Simplicity Lemma~\ref{lemma:simplicity2} and the Obstruction Lemma~\ref{obstructionLemma} together with the deformation theory of $\mathrm{Syst}^k(v)$ developed in Section~\ref{coherentsystems}. We obtain criteria characterizing basepoint-free pencils on curves in a fixed linear system that do not define smooth points of the expected local dimension in their Brill--Noether loci. We also apply Proposition \ref{prop:generalizedAproduFarkas} to the study of Green's conjecture for curves on certain regular surfaces and obtain some new cases of the conjecture in this direction.
We will write $\mathcal{L}$ for a line bundle corresponding to an arbitrary nonzero effective divisor class on $X$, always assuming that the locus $|\mathcal{L}|_{\mathrm{int}}$ of integral curves is non-empty. We do \emph{not} assume $\mathcal{L}$ to be indecomposable or $-K_X$ to be effective unless explicitly stated.

We consider the relative space of pencils
$$
\mathcal{G}^1_d(|\mathcal{L}|) \coloneqq \left\{ (i_\ast A,V )\colon i \colon C \hookrightarrow X, \ C \in |\mathcal{L}|, \ V \subseteq h^0(C,A), \ \dim V=2 \right\},
$$
where $A$ is a torsion-free sheaf of rank~$1$ on $C$, or more formally $i_\ast A \in \Spl(0,\mathcal{L},d-\frac{1}{2}\mathcal{L}^2)$. We define the open sublocus of complete pencils as
$$
\widetilde{\mathcal{G}}^1_d(|\mathcal{L}|)\coloneqq \L\{i_*A\in \Spl(0,\mathcal{L},d-\frac{1}{2}\mathcal{L}^2)\colon h^0(C,A) = 2\R\}.
$$
The latter of course agrees with the locally closed Brill--Noether locus of complete pencils $\widetilde{\mathcal{W}}^1_d(|\mathcal{L}|)$. In all of these, the condition of being basepoint-free is an open condition and we obtain the corresponding loci denoted by $\widetilde{\mathcal{W}}^{1,\mathrm{bpf}}_d(|\mathcal{L}|) = \widetilde{\mathcal{G}}^{1,\mathrm{bpf}}_d(|\mathcal{L}|)$ and $\mathcal{G}^{1,\mathrm{bpf}}_d(|\mathcal{L}|)$, respectively. For a fixed curve $C$, we write $\widetilde{W}^{1,\mathrm{bpf}}_d(C) = \widetilde{G}^{1,\mathrm{bpf}}_d(C)$ and $G^{1,\mathrm{bpf}}_d(C)$ for the space of (complete) basepoint-free pencils of degree $d$ on $C$. 

\begin{definition}
\label{def:surface-induced}
    We call a pencil $(A,V)$ on a curve $C\subset X$ \emph{surface-induced} if there exists $H>0$ with $2H \leq [C]$ on $X$ such that $A \leq H|_C$. Further, we call $(A,V)$ \emph{split} if the associated Lazarsfeld--Mukai bundle $E_{C,A,V}$ is a direct sum of line bundles $E_{C,A,V} \simeq \O_X(D_1)\oplus \O_X(D_2)$.
\end{definition}

Here, in case $C \subseteq X$ is integral but possibly singular and $A$ is a torsion-free rank~$1$ sheaf on $C$, we interpret $A \leq H|_C$ as the existence of a sheaf injection $A \hookrightarrow \O_X(H)|_C$. Of course, if $C$ is smooth and $A$ is a line bundle, $A$ is surface-induced if and only if 
$$
    A = H|_C-B,
$$
for some effective divisor class $B$ on $C$. If $C \subseteq X$ is only integral, this must be replaced by $A \simeq \O_X(H)|_C \otimes \mathcal{I}_B$ for a subscheme $B \subseteq C$.
The main tool for this section is
\begin{proposition}
\label{proposition:surfaceinducedpencils}
    Let $X$ be a smooth projective surface. Consider a non-empty complete linear system $|\mathcal{L}|$ on $X$. Let $C\in |\mathcal{L}|$ be an integral curve and assume there exists a basepoint-free pencil $(A,V)$ on $C$. Then at least one of the following holds:
    \begin{enumerate}[(i)]
        \item $(A,V)$ is surface-induced;
        \item $(A,V)$ is split;
        \item the associated Lazarsfeld--Mukai bundle $E_{C,A,V}$ is simple.
    \end{enumerate}
    Moreover, if $(A,V)$ is split but not surface-induced, then $\dim \Hom(E_{C,A,V},E_{C,A,V}) =2$.
\end{proposition}
\begin{proof}
    The Lazarsfeld--Mukai bundle $E=E_{C,A,V}$ is a ggg vector bundle of rank two. Assume $E$ is not simple. By Lemma \ref{lemma:simplicity2}, $E$ is either a direct sum of two effective line bundles, or there exists a decomposition $[C] = 2H+H'$ with $H>0,H'\geq 0$ and $E$ sits in a non-split short exact sequence
    $$
    0\to \O_X(H+H')\to E\xrightarrow{p} \mathcal{I}_Z(H)\to 0,
    $$
    where $\mathcal{I}_Z$ is the ideal sheaf of a $0$-dimensional local complete intersection subscheme $Z\subseteq X$. Moreover, by definition, there exists a sequence
    $$
    0\to V^{\vee}\otimes\O_X\xrightarrow{\iota} E\to i_*L\to 0,
    $$
    where $i\colon C\to X$ is the inclusion of $C$ and $L = A^{\vee}\otimes \omega_C\otimes \omega_X^{\vee}|_C$. Consider $G\coloneqq \mathrm{ker}(p\circ\iota)$. By definition, $G \subseteq V^{\vee} \otimes \O_X$, and by the universal property of $\mathrm{ker}(p)$, we obtain a map $G \to \O_X(H+H')$ such that the following square commutes
    \[\begin{tikzcd}
	   G & {\mathcal{O}_X(H+H')} \\
	   {V^{\vee}\otimes\mathcal{O}_X} & E
	   \arrow[from=1-1, to=1-2]
	   \arrow[hook, from=1-1, to=2-1]
	   \arrow[hook, from=1-2, to=2-2]
	   \arrow["\iota", hook, from=2-1, to=2-2].
    \end{tikzcd}\]
    It follows that also $G\to \O_X(H+H')$ is injective and we can complete the diagram to the following diagram with exact rows and columns:
    \[\begin{tikzcd}
	& 0 & 0 & 0 & \\
	0 & G & {\mathcal{O}_X(H+H')} & Q & 0 \\
	0 & {V^{\vee}\otimes\mathcal{O}_X} & E & {i_*L} & 0 \\
	0 & {G'} & {\mathcal{I}_Z(H)} \\
	& 0 & 0
	\arrow[from=1-2, to=2-2]
	\arrow[from=1-3, to=2-3]
	\arrow[from=1-4, to=2-4]
	\arrow[from=2-1, to=2-2]
	\arrow[from=2-2, to=2-3]
	\arrow[from=2-2, to=3-2]
	\arrow[from=2-3, to=2-4]
	\arrow[from=2-3, to=3-3]
	\arrow[from=2-4, to=2-5]
	\arrow[from=2-4, to=3-4]
	\arrow[from=3-1, to=3-2]
	\arrow["\iota", from=3-2, to=3-3]
	\arrow[from=3-2, to=4-2]
	\arrow[from=3-3, to=3-4]
	\arrow["p", from=3-3, to=4-3]
	\arrow[from=3-4, to=3-5]
	\arrow[from=4-1, to=4-2]
	\arrow[from=4-2, to=4-3]
	\arrow[from=4-2, to=5-2]
	\arrow[from=4-3, to=5-3]
    \end{tikzcd}\]
    In particular, as $C$ is integral and $L$ is a torsion-free rank~$1$ sheaf on $C$, we find that $Q$ is either $0$ or of the form $Q= i_*L'$ for some nonzero $L' \subseteq L$. If $Q=0$, then $G\simeq \O_X(H+H')$, but by the assumption $H>0, H'\geq 0$ we have $\Hom(\O_X(H+H'),\O_X^2) = 0$ and thus $G=0$, a contradiction. It follows that $Q\simeq i_*L'$ and $\det(Q) \simeq \O_X(C)$. Since $G'$ is torsion-free, $G$ is saturated inside $V^\vee \otimes \O_X$, so $G$ is locally free. As $\rk(G) = 1$, $G$ is a line bundle. Then, since $\O_X(C) \simeq \O_X(2H+H')$ and $G$ is a line bundle, it follows from the top horizontal sequence that $G \simeq \O_X(-H)$. Now, the top horizontal sequence reads
    $$
        0\to \O_X(-H)\to \O_X(H+H') \to i_*L'\to 0,
    $$
    but any such sequence is given by a regular section $s\colon \O_X\to \O_X(C)$ twisted by $\O_X(-H)$. It follows $L'\simeq\O_X(H+H')|_C$ and thus $L\geq (H+H')|_C$ as divisor classes on $C$. Using 
    $$
        L = A^{\vee} \otimes \omega_C \otimes \omega_X^{\vee}|_C = A^{\vee} \otimes \O_C(C),
    $$
    we find $A \leq C|_C-(H+H')|_C = H|_C$. This is exactly the condition $A = H|_C - B$ for some effective divisor $B$ on $C$, i.e., $(A,V)$ is surface-induced.

    \smallskip
    \noindent
    For the final statement, if $E \simeq \O_X(D_1) \oplus \O_X(D_2)$, there are two cases: either, without loss of generality, $D_1 \geq D_2$, or $D_1$ and $D_2$ are incomparable. If $D_1 \geq D_2$, then $C= D_1 + D_2 \geq 2D_2$ and in the above proof we can choose $H = D_2$ and $H'= D_1-D_2$, so that again $(A,V)$ is surface-induced. If $D_1$ and $D_2$ are incomparable, we have $\Hom(\O_X(D_1),\O_X(D_2)) = \Hom(\O_X(D_2),\O_X(D_1)) = 0$ and thus 
    \[
        \Hom(E,E)  = \Hom(\O_X(D_1),\O_X(D_1))\oplus \Hom(\O_X(D_2),\O_X(D_2)) \simeq \C^2.\qedhere
    \]
    \end{proof}

\subsection{Complete basepoint-free pencils.}
  For a basepoint-free pencil $A$ on a smooth curve $C$, one can use the basepoint-free pencil trick to obtain a condition on when the Petri map 
  $$
  \mu_A\colon H^0(C,A)\otimes H^0(C,A^{-1}\otimes \omega_C)\to H^0(C,\omega_C)
  $$
  fails to be injective. Indeed, one has $\mathrm{ker}(\mu_A) \simeq H^0(C,A^{-2}\otimes \omega_C)$, so one obtains the condition
 \begin{equation}
 \label{equation:petrinotinjective}
 \text{Petri map not injective } \Leftrightarrow K_C\geq 2A.
 \end{equation}
 Let now $A$ be surface-induced. In particular, let $H>0$, $H'\geq 0$ such that $C = 2H+H'$ as divisors on $X$, and let $B$ be an effective divisor on $C$ such that $A= H|_C-B$. Using adjunction, condition (\ref{equation:petrinotinjective}) becomes
 $$
 (K_X+H')|_C\geq -2B.
 $$
 In particular, if $(K_X+H')|_C$ is effective, the Petri map is not injective.
 
\medskip
\noindent
 The next result shows that simplicity of the Lazarsfeld--Mukai bundle is sufficient to conclude smoothness of the associated complete pencil on a curve, if the curve is general in its linear system. Denote by $|\mathcal{L}|_{\mathrm{int}}$ the locus of integral curves in $|\mathcal{L}|$, and let $\mathcal{S}$ be the closed subscheme
$$
\mathcal{S}=\{(C,A)\colon A \text{ is split or surface-induced}\} \subseteq \widetilde{\mathcal{G}}^{1,\mathrm{bpf}}_d(|\mathcal{L}|_{\mathrm{int}}).
$$

\begin{theorem}
\label{theorem:GlobalCompletePencilIsSmooth}
    Let $X$ be a surface with $H^1(X,\O_X) = 0$. Let $\mathcal{L}$ be a nonzero effective divisor class and let $C \in |\mathcal{L}|$ be an integral curve.
    \begin{enumerate}[(i)]
        \item If $K_X \neq 0$, assume $C$ satisfies $-K_X|_C>0$. Let $A$ be a complete basepoint-free pencil of degree $d$ on $C$ that is neither split nor surface-induced, then
        $$
            (C,A,H^0(C,A))\in \mathcal{G}^1_d(|\mathcal{L}|)
        $$
        is a smooth point of local dimension $\rho(g,1,d)+\dim |\mathcal{L}|$, where $g \coloneqq p_a(C)$.
        \item  Assume that either $-K_X > 0$ with $-K_X.C \neq 0$ and $g \geq 2$, or $K_X = 0$. Then 
        $$
            \widetilde{\mathcal{G}}^{1,\mathrm{bpf}}_d(|\mathcal{L}|_{\mathrm{int}})\setminus\mathcal{S}
        $$
        is smooth of dimension $\rho(g,1,d) + \dim |\mathcal{L}|$, or empty.
    \end{enumerate}
\end{theorem}

\begin{proof}
The complement $\widetilde{\mathcal{G}}^{1,\mathrm{bpf}}_d(|\mathcal{L}|_{\mathrm{int}})\setminus\mathcal{S}$ is isomorphic to an open in $\Syst^2(2,\mathcal{L},\ch_2)^{\circ}$ via the map 
$$
(C,A) \mapsto (E_{C,A},H^0(C,A)^{\vee}).
$$
Moreover, by definition of $\mathcal{S}$, we can apply Proposition~\ref{proposition:surfaceinducedpencils} to see that $E_{C,A}$ is simple. Note that, if $-K_X > 0$ and $C$ integral with $p_a(C) \geq 2$, then $-K_X.C \neq 0$ already implies $-K_X|_C > 0$. With this observation, the Obstruction Lemma~\ref{obstructionLemma} and the simplicity of $E_{C,A}$ imply, similarly as in the proof of Theorem~\ref{theorem:SystSmoothnessFork<r}, that $\mathrm{Syst}^2(2,\mathcal{L},\ch_2)^{\circ}$ is smooth of the expected dimension at $(E_{C,A},H^0(C,A)^{\vee})$, so $\mathcal{G}^{1}_d(|\mathcal{L}|)$ is smooth of the expected dimension at $(C,A)$. A calculation shows
$$
\mathrm{expdim}(\mathrm{Syst}^2(2,\mathcal{L},\ch_2)^{\circ}) = \rho(g,1,d)+\dim |\mathcal{L}|,
$$
using the formulas given in Sections~\ref{coherentsystems} and~\ref{BNTheorem}.
\end{proof}
 
This has consequences for computing the gonality of curves:

\begin{corollary}[Gonality of curves]
    Let $X$ be a surface with $-K_X$ effective and $H^1(X,\O_X) = 0$. Let $|\mathcal{L}|$ be a linear system of genus $g \geq 2$ curves on $X$ containing a smooth curve. Let $C \in |\mathcal{L}|$ be general. If $\mathrm{gon}(C) = d$ and $\rho(g,1,d)<0$, then any pencil $A$ on $C$ computing the gonality is either surface-induced or split.
\end{corollary}

On a K3 surface, we obtain a classification of complete and base-point free pencils that do not define smooth points of the expected dimension in their Brill--Noether loci:

\begin{corollary}[Singular pencils on curves on K3 surfaces]\label{cor:singularPencilsOnK3s}
    Let $X$ be a K3 surface and let $C \in |\mathcal{L}|_{\mathrm{int}}$ be general. Assume the linear system $|\mathcal{L}|$ has no $1$-dimensional base components. Let $A$ be a complete basepoint-free pencil on $C$. Then the following are equivalent:
    \begin{enumerate}[(i)]
        \item $W^1_d(C)$ is singular or has larger than the expected local dimension at $A$;
        \item The Lazarsfeld--Mukai bundle $E_{C,A}$ is not simple;
        \item There exists a divisor class $0 < D < [C]$ on $X$, such that
        $$
        A\leq D|_C \text{ and } A\leq (C-D)|_C. 
        $$
    \end{enumerate}
\end{corollary}

\begin{proof}
    Assume first that $A$ is not split.
    By Theorem~\ref{theorem:GlobalCompletePencilIsSmooth} and generic smoothness, $(i)$ is equivalent to $A$ being surface-induced. Write $A= H|_C-B$ with $C =  2H+H'$ and $H> 0$, $H' \geq 0$ on $X$, $B \geq 0$ on $C$. Then, as $|\mathcal{L}|$ has no $1$-dimensional base components, we have $H|_C\leq (H+H')|_C = (C-H)|_C$. Therefore,
    $$
    A\leq H|_C \text{ and } A\leq (C-H)|_C.
    $$
    Conversely, if $(iii)$ is satisfied, we have $2A \leq D|_C + (C-D)|_C = C|_C = K_C$, and by the basepoint-free pencil trick, $(i)$ holds.
    
    Next, by the Obstruction Lemma~\ref{obstructionLemma}, if $A \in W^1_d(C)$ is not a smooth point of the expected local dimension, then $E_{C,A}$ is not simple. Conversely, assume $E_{C,A}$ is not simple, i.e., $(ii)$. Since we assume that $E_{C,A}$ is not a direct sum of line bundles, $A$ is surface-induced by Proposition~\ref{proposition:surfaceinducedpencils}, so $(i)$ follows as above.
    
    Now, we assume $A$ is split, i.e., $E=E_{C,A} \simeq \O_X(D_1)\oplus \O_X(D_2)$ for two nonzero effective divisor classes $D_1,D_2$ such that $D_1+D_2 = [C]$. We only need to show that in this case, all three statements hold. For this, consider the restriction to $C$ of the sequence defining $E$,
    $$
    0\to K\to E|_C\to L\to 0.
    $$
    Here, $K$ is a line bundle and $L= A^{\vee}\otimes \omega_C$. We find  $\mathrm{det}(E|_C) \simeq \O_C(C) \simeq \omega_C$. It follows that $K \simeq A$. Since $E \simeq \O_X(D_1) \oplus \O_X(D_2)$, it follows from the sequence that $A \leq D_1|_C$ and $A \leq D_2|_C$ with $D_1+D_2 = [C]$. Indeed, the maps $A \to \O_X(D_i)|_C$ are both nonzero as otherwise, without loss of generality, $A \simeq \O_X(D_1)|_C$ and $L \simeq \O_X(D_2)|_C$. In this case, the composition $\O_X(D_1)\hookrightarrow E \to i_\ast L$ is zero, so the universal property of the kernel would give a map $\O_X(D_1)\to \O_X^2$ which is absurd since $D_1 > 0$. Hence, $(iii)$ is satisfied which again, via the basepoint-free pencil trick, implies $(i)$. Finally, as $E$ is a direct sum, $(ii)$ is trivially satisfied.
\end{proof}

\begin{corollary}
\label{cor:degreeconditionspencil}
    Let $X$ be a smooth projective surface with $H^1(X,\O_X) = 0$ and $-K_X\geq 0$. Let $C \in |\mathcal{L}|$ be a general integral curve with $g \coloneqq p_a(C) \geq 4$. Then, if 
    $$
    3d\geq 2g+1-\frac{1}{2}C.K_X,
    $$
    the Brill--Noether locus of basepoint-free complete pencils $\widetilde{W}^{1,\mathrm{bpf}}_d(C)$ is of the expected dimension $\rho(g,1,d)$, or empty.
\end{corollary}

\begin{proof}
    First of all, the assumptions imply $\rho(g,1,d) \geq 0$. By Theorem~\ref{theorem:GlobalCompletePencilIsSmooth}, the statement is true away from the closed sublocus 
    $$
    \mathcal{S}_C=\{A\in \widetilde{W}^{1,\mathrm{bpf}}_d(C)\colon A \text{ is split or surface-induced}\}.
    $$
    Hence, it suffices to show $\dim \mathcal{S}_C \leq \rho(g,1,d)$. First, we bound the dimension of the locus of surface-induced pencils. Let $H>0$ and $B\geq 0$ such that $2H \leq C$ and $A= H|_C-B$. For fixed $H$, the locus of these pencils is parametrized by $B$, which is an effective divisor on $C$ of degree $\mathrm{deg}(B) = C.H-d$. The parameter space for effective divisors of degree $\deg(B)$ on $C$ is $\mathrm{Sym}^{\deg B}(C)$, so the dimension of the locus of $B$ for a fixed $H$ is bounded above by $C.H-d$. As there are only finitely many effective $H$ satisfying $2H\leq C$, we get $\dim \mathcal{S}_C \leq \mathrm{max}_H \{C.H-d\}$. Using $2H \leq C$ again, this implies $\dim \mathcal{S}_C \leq \frac{1}{2}C^2-d$. The condition $\dim \mathcal{S}_C \leq \rho(g,1,d)$ is thus implied by
    \begin{align*}
        \frac{1}{2}C^2-d \leq 2d-2-g \quad{}\Longleftrightarrow{}\quad &g-1-\frac{1}{2}C.K_X\leq 3d-2-g \\
        {}\Longleftrightarrow{}\quad &3d \geq2g+1-\frac{1}{2}C.K_X.
    \end{align*}
    Next, we bound the dimension of the locus of split pencils. Let $E = E_{C,A} \simeq \O_X(D_1) \oplus \O_X(D_2)$ with $D_1 > 0$, $D_2 > 0$. Then $[C] = D_1 + D_2$ and $d = \deg(A) = c_2(E) = D_1.D_2$. Moreover, as in the proof of Corollary~\ref{cor:singularPencilsOnK3s}, we obtain $A \leq D_1|_C$ and $A \leq D_2|_C$, so there are effective divisor classes $B_1$ and $B_2$ on $C$ such that $A \simeq O_X(D_i)|_C(-B_i)$ for $i = 1,2$. Here, $\deg(B_i) = C.D_i - d = D_i^2$. In particular, the locus of complete, basepoint-free pencils $A$ with $E_{C,A} \simeq \O_X(D_1) \oplus \O_X(D_2)$ is dominated by a subscheme of $\mathrm{Sym}^{D_i^2} C$ for $i = 1,2$. In particular, its dimension is bounded above by $\min(D_1^2, D_2^2) \leq \frac{D_1^2 + D_2^2}{2}$. A computation shows that, under our numerical assumptions, $\frac{D_1^2 + D_2^2}{2} \leq \rho(g,1,d)$. 
    Indeed, by adjunction, $2g-2 = C^2 + C.K_X = D_1^2 + D_2^2 + 2d + C.K_X$, so $g = \frac{D_1^2 + D_2^2}{2} + d + \frac{1}{2} C.K_X + 1$, whence
    \begin{equation}\label{formulaRho}
        \rho(g,1,d) = 2d - 2 - g = d - \frac{D_1^2 + D_2^2}{2} - \frac{1}{2} C.K_X - 3.
    \end{equation}
    Our numerical hypothesis yields
    \begin{align*}
        3d &\geq 2g+1-\frac{1}{2}C.K_X = D_1^2 + D_2^2 + 2d + \frac{1}{2} C.K_X + 3,\\
        {}\Longleftrightarrow{}\quad d &\geq D_1^2 + D_2^2 + \frac{1}{2} C.K_X + 3.
    \end{align*}
    Substituting the latter into \eqref{formulaRho}, we obtain $\rho(g,1,d) \geq \frac{D_1^2 + D_2^2}{2}$, as desired.
    Since there are only finitely many decompositions $[C] = D_1 + D_2$ with $D_1 > 0$, $D_2 > 0$, this completes the proof.
\end{proof}

Next, we give numerical conditions under which a pencil $A$ cannot be split. In that case, $\mathcal{S}$ is exactly the locus of surface-induced pencils.
 
\begin{lemma}
    Let $X$ be a surface with $H^1(X,\O_X) = 0$. Assume $H^1(X,\mathcal{L}) = 0$ and that $|\mathcal{L}|$ contains a smooth curve. Let $C \in |\mathcal{L}|$ be general and assume $-K_X|_C>0$. Let $A$ be a complete basepoint-free pencil on $C$ of degree $d$. Assume that at least one of the following conditions is satisfied for each decomposition $[C] = D_1 + D_2$ with $D_1$ and $D_2$ effective and nonzero divisor classes:
\begin{enumerate}[(i)]
        \item $d\neq D_1.D_2$;
        \item $2d> g-C.K_X-2+3\chi(\O_X)$;
        \item $(D_1-D_2)^2< C.K_X-6\chi(\O_X)+2$.
\end{enumerate}
    Then $A$ is not split.
\end{lemma}

\begin{proof}
    Assume $A$ is split, i.e., $E=E_{C,A}\simeq \O_X(D_1)\oplus \O_X(D_2)$ for effective divisor classes $D_1,D_2$ satisfying $D_1+D_2  = [C]$. Then $c_2(E) = \deg A = D_1.D_2$. This is $(i)$.
    For the second condition we notice the following: Given a ggg vector bundle $E$ of rank two on the surface, if $2h^0(E) -4 < \dim |\mathcal{L}|$, then the general curve $C\in |\mathcal{L}|$ does not occur as the vanishing of the determinant of a map $\O_X^2\to E$. Now as there are only finitely many decompositions $[C] = D_1+D_2$ into effective classes, we find the condition
    \begin{equation}
    \label{eqn:dimensioncount}
    2  h^0(E) - 4 < \dim |\mathcal{L}|.
    \end{equation}
    Indeed, if this condition is satisfied, the general curve $C\in |\mathcal{L}|$ does not admit a complete basepoint-free pencil that is split. By assumption we have 
    $$
    \dim |\mathcal{L}| = \chi(\O_X)+\frac{1}{2}C^2-\frac{1}{2}C.K_X-1.
    $$
    Also, $h^1(X,E) = h^0(X,A\otimes\omega_X|_C)$. As $A$ is a complete basepoint-free pencil and $\omega_X^{\vee}|_C$ is nontrivial effective, we find $h^0(X,A\otimes\omega_X|_C) \leq 1$. There are two ways to calculate $h^0(E)$, giving rise to conditions $(ii)$ and $(iii)$, respectively.
    First, we have $h^0(E)  = \chi(E) +h^1(X,E) \leq \chi(E) +1$. Condition (\ref{eqn:dimensioncount}) is thus implied by
    $$
    2\chi(E)-2 < \chi(\O_X) +\frac{1}{2}C^2-\frac{1}{2}C.K_X-1.
    $$
    The statement follows from a computation using $\chi(E) = 2\chi(\O_X) + \chi(L)$, where $L\coloneqq A^{\vee}\otimes\omega_C\otimes\omega_X^{\vee}|_C$ and $\deg(L) = C^2 - \deg(A)$.
    Second, as $A$ is split, we have $h^1(E) = h^1(\O_X(D_1)) + h^1(\O_X(D_2))$. Using $h^1(E) \leq 1$, we can assume without loss of generality $h^1(\O_X(D_1)) = 0$ and $h^1(\O_X(D_2)) = 1$. Riemann--Roch implies
    $$
    2h^0(E) -4 =  4\chi(\O_X) +2+D_1^2+D_2^2-C.K_X-4.
    $$
    A computation reduces (\ref{eqn:dimensioncount}) to $(D_1-D_2)^2 < C.K_X -6\chi(\O_X)+2$.
\end{proof}

\subsection{Non-complete basepoint-free pencils.}
We now discuss the case of possibly non-complete basepoint-free pencils $(A,V)\in G^{1,\mathrm{bpf}}_d(C)$. Consider the exact sequence
$$
0\to E^{\vee}\to V\otimes\O_X\to i_*A\to 0.
$$
We find $h^1(X,E^{\vee}) = h^0(C,A)-\dim V$. After tensoring with $E^{\vee}$, the dual sequence yields
$$
\hom(E,i_*L) \leq \hom(E,E)+\dim V(h^0(A)-\dim V).
$$
From the deformation theory of Section~\ref{coherentsystems}, we deduce that the dimension of the trace-free obstruction space for $\Lambda =(E,V^{\vee})$ is bounded by 
$$
\dim \Ext^2(\Lambda,\Lambda)_0 \leq \hom(E,E) -1 + \dim V (h^0(A) - \dim V)
$$
if either $K_X = 0$ or $-K_X|_C > 0$.
\noindent
We consider the closed subscheme
    $$
    \mathcal{S}=\{(C,A,V)\colon (A,V) \text{ is split or surface-induced}\}\subset \mathcal{G}^{1,\mathrm{bpf}}_d(|\mathcal{L}|_{\mathrm{int}})
    $$
    of the relative Brill--Noether locus of possibly non-complete basepoint-free pencils.
    
\begin{proposition} 
    Let $X$ be a surface with $H^1(X,\O_X)=0$. Let $g \coloneqq p_a(C) \geq 2$ for $C \in |\mathcal{L}|$. If $-K_X$ is not effective, assume also $-K_X.C \geq g$. Consider the forgetful map $\pi$ given by $(C,A,V)\mapsto (C,A)$. Then
    $$
\pi\L(\mathcal{G}^{1,\mathrm{bpf}}_d(|\mathcal{L}|_{\mathrm{int}})\setminus\mathcal{S}\R) = \L\{(C,A)\colon \exists V\subset H^0(C,A) \text{ s.t. } (C,A,V) \in\mathcal{G}^{1,\mathrm{bpf}}_d(|\mathcal{L}|_{\mathrm{int}})\setminus\mathcal{S}\R\}
    $$
    is equi-dimensional of dimension $\rho(g,1,d)+\dim |\mathcal{L}|$, or empty. In particular, if $C \in |\mathcal{L}|_{\mathrm{int}}$ is general, denote by $\mathcal{S}_C$ the fiber of $\mathcal{S}$ over $C$. Then, $\dim \pi\L(G^{1,\mathrm{bpf}}_d(C)\setminus \mathcal{S}_C\R) \leq \rho(g,1,d)$.
\end{proposition}

\begin{proof}
By Proposition~\ref{proposition:surfaceinducedpencils}, $\mathcal{G}^{1,\mathrm{bpf}}_d(|\mathcal{L}|_{\mathrm{int}})\setminus\mathcal{S}$ is isomorphic to an open in $\Syst^2(2,\mathcal{L},\ch_2)^{\circ}$ via $(C,A,V)\mapsto (E_{C,A,V},V^{\vee})$. In particular, we may apply the deformation theory developed in Section~\ref{coherentsystems} to $\mathcal{G}^{1,\mathrm{bpf}}_d(|\mathcal{L}|_{\mathrm{int}})\setminus \mathcal{S}$. By the above discussion, the dimension of the obstruction space is bounded by 
$$
\hom(E_{C,A,V},E_{C,A,V})-1+\dim V^{\vee}(h^0(A)-\dim V^{\vee}) = 2(h^0(A)-2),
$$
so we obtain
\begin{align*}
\dim_{(C,A,V)}\mathcal{G}^{1,\mathrm{bpf}}_d(|\mathcal{L}|_{\mathrm{int}})\setminus\mathcal{S} &= \dim_{(E_{C,A,V},V^{\vee})}\Syst^2(2,\mathcal{L},\ch_2)^{\circ}\\
&\leq \mathrm{expdim}\Syst^2(2,\mathcal{L},\ch_2)^{\circ} +2(h^0(A)-2).
\end{align*}
On the other hand, the fiber of the forgetful map $\pi$ over $(C,A)$ is exactly the Grassmannian $\pi^{-1}(C,A) = \mathrm{Gr}(2,H^0(A))$, so its dimension is $2(h^0(A)-2)$. Let $E_i$ be any irreducible component of $\pi\L(\mathcal{G}^{1,\mathrm{bpf}}_d(|\mathcal{L}|_{\mathrm{int}})\setminus\mathcal{S}\R)$ and consider the open $\mathcal{U}_i\subset E_i$ on which $h^0(C,A)$ is minimal. Then $\pi_{\mathcal{U}_i}$ is an étale-local fiber bundle with fiber $\mathrm{Gr}(2,H^0(A))$. Let $(C,A)$ be a general point in $\mathcal{U}_i$. We obtain
\begin{align*}
    \dim E_i = \dim \mathcal{U}_i &\leq \mathrm{expdim}(\mathrm{Syst}^2(2,\mathcal{L},\ch_2)^{\circ})+2(h^0(A)-2)-\dim \pi^{-1}(C,A)\\
    &= \mathrm{expdim}(\Syst^2(2,\mathcal{L},\ch_2)^{\circ}) = \rho(g,1,d)+\dim |\mathcal{L}|.\qedhere
\end{align*}
\end{proof}

\subsection{A smoothness criterion without simplicity} In the previous sections we employed simplicity results on the Lazarsfeld--Mukai bundle $E_{C,A,V}$ to deduce smoothness results for the associated pencil $(C,A,V)$. In the $-K_X|_C > 0$ case, we obtain a criterion on a complete pencil $(C,A)$ such that the obstructions for the associated coherent system $(E_{C,A},H^0(A)^{\vee})$ vanish without $E_{C,A}$ necessarily being simple.

\begin{proposition}\label{prop:smoothnessWithoutSimplicity}
    Let $X$ be a surface with $H^1(X,\O_X) = 0$. Let $C \in |\mathcal{L}|$ be integral such that $-K_X|_C > 0$ is basepoint-free. Let $A$ be a complete basepoint-free pencil on $C$. Assume that for all decompositions $[C] = 2H+H'$ with $H>0$, $H' \geq 0$ the condition $-K_X|_C \nleq H'|_C$ is satisfied. Then the Lazarsfeld--Mukai bundle $E_{C,A}$ is either a direct sum of two line bundles or $(C,A) \in \widetilde{\mathcal{W}}^1_d(|\mathcal{L}|_{\mathrm{int}})$ is a smooth point of the expected local dimension.
\end{proposition}

\begin{remark}
    If, in the above setup, $H^1(X,\O_X(2H)) = 0$ for a decomposition $[C] = 2H+H'$, the condition $-K_X|_C \nleq H'|_C$ on $C$ can be replaced by $-K_X \nleq H'$ on $X$. Indeed, the divisor sequence associated to $C \subseteq X$ yields
    $$
    H^0(X,\O_X(H'+K_X)) \to H^0(C,\O_C(H'+K_X)) \to H^1(X,\O_X(H'+K_X-C)) \simeq H^1(X,\O_X(2H))^\vee.
    $$ 
\end{remark}

\begin{proof}
    Let $\Lambda \coloneqq (E,V) \coloneqq (E_{C,A},H^0(A)^{\vee})$ be the coherent system associated to $A$ and denote $Q_{\Lambda}\coloneqq \coker(V\otimes \O_X\to E)$. Assume that $E$ is not a direct sum of two line bundles. We need to show that $\Hom(E,Q_{\Lambda}\otimes\omega_X) = 0$, and this is by now clear in case $E$ is simple, so we may assume $E$ is not simple. Consider a section $s\colon \omega_X|_C \hookrightarrow \O_C$. We obtain the exact sequence
    $$
    0\to \Hom(E,Q_{\Lambda}\otimes\omega_X)\to \Hom(E,Q_{\Lambda})\to \Hom(E,Q_{\Lambda}|_{D}),
    $$
    where $D= V(s)$. The sequence 
    $$
    0\to V\otimes\O_X\to E\xrightarrow{f} Q_{\Lambda} \to 0
    $$
    together with $H^0(X,E^{\vee}) = 0 = H^1(X,E^{\vee})$ yields $ \Hom(E,E)\simeq \Hom(E,Q_{\Lambda})$. This isomorphism is given by precomposition $g\mapsto f\circ g$. Therefore, every map $h \colon E \to Q_{\Lambda}$ is of the form $f\circ g$ for some endomorphism $g\colon E\to E$. If $h$ is surjective, then it does not vanish on $D = V(s)$ for any $s$, so $h \notin \mathrm{ker}(\Hom(E,Q_{\Lambda}) \to \Hom(E,Q_{\Lambda}|_{D})) = \Hom(E,Q_{\Lambda}\otimes\omega_X)$. Hence, to show that this kernel is zero, it suffices to consider non-surjective maps $h \colon E \to Q_{\Lambda}$. In that case, necessarily $g \colon E\to E$ is not surjective and thus drops rank everywhere. Consider the Donagi--Morrison extension (\ref{lemma:simplicity2}),
    $$
    0\to \O_X(H+H')\xrightarrow{i} E\to \mathcal{I}_Z(H)\to 0,
    $$
    with $[C] = 2H+H'$ and $H>0$, $H' \geq 0$. Aprodu--Farkas \cite[Lem. 3.4]{FarkasAproduGreenConj} show that any endomorphism of $E$ that drops rank everywhere factors as
    \begin{equation}
    \label{eqn:DMfactorization}
    g\colon E\twoheadrightarrow\mathcal{I}_Z(H)\hookrightarrow\O_X(H)\xrightarrow{t}\O_X(H+H')\xrightarrow{i} E
    \end{equation}
    for some section $t\colon \O_X(H)\hookrightarrow\O_X(H+H')$. Assume now $f\circ g$ maps to zero under $\Hom(E,Q_{\Lambda})\to \Hom(E,Q_{\Lambda}|_{D})$; in particular,
    \begin{equation}
    \label{eqn:condkernel}
    (f\circ g)_x = 0 \text{ for all } x \in D.
    \end{equation}
     We want to show that $f\circ g = 0$, so assume $f\circ g \neq 0$. From (\ref{eqn:DMfactorization}) it follows that $g$ is nonvanishing away from $Z\cup V(t)$. Consider the following commutative diagram with exact rows and columns:
\[\begin{tikzcd}
	& 0 & 0 & 0 & \\
	0 & {\O_X(-H)} & {V\otimes \O_X} & {\mathcal{I}_{Z'}(H)} & 0 \\
	0 & {\O_X(H+H')} & E & {\mathcal{I}_Z(H)} & 0 \\
	0 & {Q_{\Lambda}'} & {Q_{\Lambda}} & {\mathcal{I}_Z/\mathcal{I}_{Z'}} & 0 \\
	& 0 & 0 & 0
	\arrow[from=1-2, to=2-2]
	\arrow[from=1-3, to=2-3]
	\arrow[from=1-4, to=2-4]
	\arrow[from=2-1, to=2-2]
	\arrow[from=2-2, to=2-3]
	\arrow[from=2-2, to=3-2]
	\arrow[from=2-3, to=2-4]
	\arrow[from=2-3, to=3-3]
	\arrow[from=2-4, to=2-5]
	\arrow[from=2-4, to=3-4]
	\arrow[from=3-1, to=3-2]
	\arrow["i", from=3-2, to=3-3]
	\arrow[from=3-2, to=4-2]
	\arrow[from=3-3, to=3-4]
	\arrow["f", from=3-3, to=4-3]
	\arrow[from=3-4, to=3-5]
	\arrow[from=3-4, to=4-4]
	\arrow[from=4-1, to=4-2]
	\arrow[from=4-2, to=4-3]
	\arrow[from=4-2, to=5-2]
	\arrow[from=4-3, to=4-4]
	\arrow[from=4-3, to=5-3]
	\arrow[from=4-4, to=4-5]
	\arrow[from=4-4, to=5-4].
\end{tikzcd}\]
In the diagram, one indeed has $\ker(f\circ i)\simeq \O_X(-H)$ as this is a saturated rank-one subsheaf of a vector bundle and thus a line bundle.
Now, by (\ref{eqn:DMfactorization}), $f\circ g$ factors through $f\circ i$ and we find from the above diagram that $f\circ g$ is nonvanishing away from $Z\cup V(t)\cup \mathrm{supp}(\mathcal{I}_Z/\mathcal{I}_{Z'}) = Z'\cup V(t)$. From condition (\ref{eqn:condkernel}) we find in total
$$
D = V(s) \subseteq (Z'\cup V(t))\cap C.
$$
As $\mathrm{Bs}(|-K_X|_C|) = \emptyset$, we may choose an $s$ such that $V(s)$ avoids $Z'$, i.e., $V(s) \cap Z' =\emptyset$. For every such $s$, we have 
$$
(f\circ g)\in \ker(\Hom(E,Q_{\Lambda})\to \Hom(E,Q_{\Lambda}|_{V(s)})) \Rightarrow V(s) \subseteq C\cap V(t).
$$
Choosing $V(s)$ reduced implies $-K_X|_C\leq H'|_C$ and we arrive at our claim by contradiction.
\end{proof}

There are interesting examples of linear systems $|\mathcal{L}|$ such that for all integral curves $C \in |\mathcal{L}|$, the condition $-K_X|_C>0$, or rather the stronger condition $-K_X.C\geq g(C)$, is satisfied without $-K_X$ necessarily being effective.

\begin{example}\label{example}
    Let $X'$ be a surface with $H^1(X',\O_{X'}) = 0$ and $\mathcal{L} \in \Pic(X')$ a nonzero effective class. Denote the arithmetic genus of every member of the associated linear system $|\mathcal{L}|$ by $p_a$. Let $q_1,\dots, q_r$ be points on $X'$ and $m_1,\dots,m_r$ integers with $m_i\geq 2$. Assume that $C' \in |\mathcal{L}|$ is an integral curve with an ordinary $m_i$-fold point at each $q_i$ and no other singularities. Consider the blow-up $\pi\colon X = \mathrm{Bl}_{q_1,\dots,q_r}(X')\to X'$ and the strict transform $C\subset X$ of $C'$. Since $C\to C'$ is the normalization, we find
    $$
    g(C) = p_a(C') - \sum_{i=1}^r \binom{m_i}{2}, \text{ and } -K_X.C = -K_{X'}.C' - \sum_{i=1}^r m_i.
    $$
    Now, if 
    \begin{equation}
    \label{eqn:example}
    -K_{X'}.C' - p_a + \sum_{i=1}^r \frac{m_i(m_i-3)}{2} \geq 0,
    \end{equation}
    then $-K_X.C \geq g(C)$, and from Lemma \ref{lem:numericsimplyupgrade} we find that all integral curves in the linear system $|\O_X(C)|$ satisfy $-K_X|_C > 0$. Condition~\eqref{eqn:example} actually becomes \emph{easier} to satisfy if the singularities of $C$ are more severe. Indeed, the term $m_i(m_i-3)$ contributes non-negatively whenever $m_i \geq 3$.
    
    On toric surfaces, condition (\ref{eqn:example}) can be reformulated combinatorially. Let $T$ be the torus of the toric surface $X'$ and denote by $N = \Hom(\mathbb{G}_m,T)$  the associated lattice and by $M$ its dual. Assume $\mathcal{L}$ is an ample class on $X'$ corresponding to the two-dimensional lattice polygon $P$ in $M_{\mathbb{R}}$. Consider $\#(\partial P\cap M)$ and $\#(\mathrm{int}P\cap M)$, the numbers of boundary and interior lattice points of $P$, respectively. For any curve $C' \in |\mathcal{L}|$, \cite{CoxLittleSchenck} provides the formulas
    $$
    p_a(C')= \#(\mathrm{int}P\cap M), \quad -K_{X'}.C' = \#(\partial P\cap M).
    $$
    Thus, condition (\ref{eqn:example}) becomes precisely
    $$
    \#(\partial P\cap M)-\#(\mathrm{int}P\cap M)+\sum_{i=1}^{r}\frac{m_i(m_i-3)}{2}\geq 0.
    $$
    Therefore, Proposition~\ref{prop:generalizedAproduFarkas} and the results of Section~\ref{ApplicationToCurves} can be applied to $X$ and the linear system $|\O_X(C)|$ if they satisfy the above condition.
    
    \noindent An interesting special case is where $C'$ is a nodal curve (i.e., $m_i=2$) on a Hirzebruch surface $X' = \mathbb{F}_n$. In this case, \cite[Cor. 4.2b]{KeilenTyomkin} guarantees that for $a\gg 0$ one finds an integral curve of arithmetic genus $a-n-1$ in the linear system $|\O_{X'}(2C_0+aF)|$ with prescribed nodes such that the condition is satisfied.
\end{example}

\subsection{On Green's conjecture for curves on regular surfaces} \label{Green} We use Proposition~\ref{prop:generalizedAproduFarkas} and closely follow the strategy in \cite{FarkasAproduGreenConj} to obtain new instances of Green's conjecture. We treat the case of certain non-K3 regular surfaces, extending the results in \cite{Lelli-ChiesaGreenConj}. In the following, $X$ will be a smooth projective regular surface with $H^0(X,\omega_X) = 0$. Throughout we fix a nonzero effective line bundle $\mathcal{L}$ on $X$. For a smooth curve $C \in |\mathcal{L}|$, write $c=\mathrm{Cliff}(C)$ and $k= \mathrm{gon}(C)$. In \cite{AproduSyzygies}, Aprodu proved that Green's conjecture for a curve $C$ of genus $g$ and gonality $k$ is implied by the linear growth condition
\begin{equation*}
    \dim W^1_{d}(C) \leq d-k, \quad \text{for all} \ \ k\leq d \leq g-k+2,
\end{equation*}
and similarly to \cite{FarkasAproduGreenConj}, it is enough to establish the linear growth condition for the open subloci of $W^1_d(C)$ consisting of basepoint-free and complete pencils.
Moreover, if the general curve has Clifford dimension $>1$ a weaker condition suffices. Indeed, in \cite[Thm. 4.5]{AproduFarkasKoszulCohomology} Aprodu--Farkas show that, for each $2 \leq a \leq \lfloor \frac{g}{2} \rfloor + 1$,
$$
    \dim W^1_{g-a+2}(C)\leq g-2a+2 \quad \Rightarrow \quad K_{g-a+1,1}(C,\omega_C) = 0,
$$
and Green's conjecture is equivalent to $K_{g-c-1,1}(C,\omega_C) = 0$.
Moreover, for each $a$, the condition $\dim W^1_{g-a+2}(C)\leq g-2a+2$ is equivalent to
\begin{equation*}
    \dim W^1_{d}(C)\leq d - a, \quad \text{for all} \ \ a \leq d \leq g - a + 2.
\end{equation*}
A standard argument shows that this latter chain of inequalities is implied by the same chain of inequalities for the open subloci $\widetilde{W}^{1,\mathrm{bpf}}_{d}(C)$ of basepoint-free and complete pencils.

Setting $a = c+2$ is valid here unless $c \geq \lfloor \frac{g}{2} \rfloor$ which implies that $c = k-2$ and $k = \frac{g+3}{2}$ is maximal (and $g$ is odd). In this case, Green's conjecture has been settled completely by Hirschowitz--Ramanan \cite{HirschowitzRamanan} and Voisin \cite{VoisinGreenOddGenus}. Summarizing, Green's conjecture is implied by
\begin{equation}\label{eqn:lineargrowthcondition}
    \dim \widetilde{W}^{1,\mathrm{bpf}}_{d}(C) \leq d - c - 2, \quad \text{for all} \ \ c+2 \leq d \leq g-c.
\end{equation}
We will refer to \eqref{eqn:lineargrowthcondition} as the \emph{linear growth condition} even though it deviates slightly from Aprodu's classical formulation. The advantage is that \eqref{eqn:lineargrowthcondition} allows us to treat all Clifford dimensions equally.

\medskip
\noindent
By Proposition~\ref{prop:generalizedAproduFarkas}, the linear growth condition is satisfied for all components of $W^1_d(C)$ whose general element has a Lazarsfeld--Mukai bundle that is simple because $\rho(g,1,e) \leq e - c -2$ holds for all $e$ in the required range. We will therefore restrict to the case of non-simple Lazarsfeld--Mukai bundles from now on. By Lemma~\ref{lemma:simplicity2}, these are either decomposable or sit in a non-split extension
$$
0\to \O_X(H+H')\to E\to \O_X(H)\otimes \mathcal{I}_Z\to 0,
$$
where $H>0$ and $H'\geq 0$ satisfy $2H+H' = [C]$ and $Z \subseteq X$ is a zero-dimensional lci subscheme.

\medskip
\noindent
We consider the parameter space $\widetilde{\mathcal{P}}_{H,l}$ parametrizing families of vector bundles given by non-split extensions
of the above form for some zero-dimensional subscheme $Z \subseteq X$ of length $l$. For a complete and basepoint-free pencil $A$ on an integral curve $C$ on $X$, the associated Lazarsfeld--Mukai bundle always satisfies $h^2(X,E_{C,A}) =0$, and if $X$ is K3 or contains an anticanonical pencil, also $h^1(X,E_{C,A}) = 0$ holds except in the single case where $h^0(X,\omega_X^{\vee})=2$ and $A\simeq \omega_X^{\vee}|_C$, see \cite[Lemma 3.1]{Lelli-ChiesaGreenConj}. However, in our setting $h^1(X,E_{C,A}) = 1$ can occur as well. Those are the only possibilities since $h^1(X,E_{C,A}) = h^0(C, A \otimes \omega_X|_C)$ and $A$ is a complete and basepoint-free pencil.
We consider the two strata 
 \begin{align*}
 \mathcal{P}^0_{H,l}&\coloneqq \{E\in \widetilde{\mathcal{P}}_{H,l}\colon h^1(X,E) = 0 \}\subseteq \widetilde{\mathcal{P}}_{H,l},\\
 \mathcal{P}^1_{H,l}&\coloneqq \{E\in \widetilde{\mathcal{P}}_{H,l}\colon h^1(X,E) = 1 \}\subseteq \widetilde{\mathcal{P}}_{H,l},
 \end{align*}
and denote by $\gamma_i$ the codimension of $\mathcal{P}^i_{H,l}$ in $\widetilde{\mathcal{P}}_{H,l}$. As both are locally closed subspaces of a projective bundle over $X^{[l]}$, we find 
 $$
 \dim \mathcal{P}^i_{H,l} = 2l+ \ext^1(\O_X(H)\otimes \mathcal{I}_Z,\O_X(H+H'))-1-\gamma_i,
 $$
 if nonempty. By upper-semicontinuity of cohomology we have $\gamma_0 = 0$ when nonempty.
 Compare the following to \cite[Lemma 3.7]{FarkasAproduGreenConj}:
 
 \begin{lemma}
 \label{lemma:dimensionExt}
     Let $X$ be a regular surface with $H^0(X,\omega_X) = 0$. Let $Z\in X^{[l]}$ and $\O_X(H)$ an effective line bundle on $X$ with $\O_X(H')\coloneqq \mathcal{L}(-2H)$ effective. Then,
     $$
     \dim \Ext^1(\O_X(H)\otimes \mathcal{I}_Z,\O_X(H+H')) = l+h^1(X,\O_X(H')).
     $$
 \end{lemma}
 
 \begin{proof}
By Serre duality, $\dim \Ext^1(\O_X(H)\otimes \mathcal{I}_Z,\O_X(H+H')) = h^1(\omega_X(-H') \otimes \mathcal{I}_Z)$. The claim follows from the cohomology sequence associated to
\begin{equation*}
    0 \rightarrow \omega_X(-H') \otimes \mathcal{I}_Z \rightarrow \omega_X(-H') \rightarrow \O_Z \rightarrow 0. \qedhere
\end{equation*}
 \end{proof}

\noindent We consider the two Grassmannian-bundles $\mathcal{G}^i_{H,l}\to\mathcal{P}^i_{H,l}$ for $i\in \{0,1\}$ with fiber $ \mathrm{Gr}(2,\chi(E)+i)$ over $E$.

\begin{lemma}
\label{lemma:dimensionGNL}
    If nonempty, we have
    $$
    \dim \mathcal{G}^i_{H,l} = g+l+h^0(X,\O_X(H'))-K_X.(3H+H')-3+2i-\gamma_i,
    $$
    where $g$ is the arithmetic genus of a curve in the linear system $|\mathcal{L}| =|2H+H'|$.
\end{lemma}

\begin{proof}
    We have
\begin{align*}
    \dim \mathcal{G}^i_{H,l} &= \dim \mathcal{P}^i_{H,l} +2(\chi(E)+i-2) \\
    &= 2l+\dim \P\mathrm{Ext}^1(\O_X(H)\otimes \mathcal{I}_Z,\O_X(H+H'))-\gamma_i+2(\chi(E)+i-2).
\end{align*}
A Hirzebruch--Riemann--Roch calculation using $\chi(\O_X) = 1$ yields $\chi(E) = g-\mathcal{L}.K_X-c_2(E)+1$ and
$$
    \chi(\O_X(H')) = h^0(X,\O_X(H'))-h^1(X,\O_X(H')) = g-2H.(H+H')-K_X.(H+H').
$$
Together with $l = c_2(E) - H.(H+H')$ this proves the claim.
\end{proof}

\noindent Set $d\coloneqq c_2(E)$ for $E\in \mathcal{P}^i_{H,l}$. Define the rational map $h^i_{H,l}\colon \mathcal{G}^i_{H,l}\dashrightarrow \mathcal{W}^1_d(|\mathcal{L}|)$ by $(E,V)\mapsto (C,A)$ on the open where we have an exact sequence
$$
0\to V\otimes \O_X \to E\to A^{\vee}\otimes \omega_C\otimes \omega_X^{\vee}|_C\to 0.
$$
We obtain the folowing variation of \cite[Lemma 3.9]{FarkasAproduGreenConj}:
\begin{lemma}
    \label{lem:dimensionofW}
    Assume $\mathcal{P}^i_{N,l}$ contains a Lazarsfeld--Mukai bundle $E$ with $c_2(E) = d$, and let $\mathcal{W}^i\subset \widetilde{\mathcal{W}}^{1,\mathrm{bpf}}_d(|\mathcal{L}|)$ be the closure of the image of $h^i_{N,l}$. Then 
    $$
    \dim \mathcal{W}^i = g+l-K_X.(H+\mathcal{L})-3+2i-\gamma_i,
    $$
    where $g$ is the arithmetic genus of a curve in $|\mathcal{L}|=|2H+H'|$.
\end{lemma}

\begin{proof}
    As $\mathcal{P}_{H,l}$ parametrizes extension classes, it is clear that the fiber dimension of $h^i_{H,l}$ is given by $\dim (h^i_{H,l})^{-1}(C,A) = \hom(E_{C,A},A^{\vee}\otimes\omega_C\otimes\omega_X^{\vee}|_C)-1$. As $h^0(X,E_{C,A}^{\vee}) = 0 =h^1(X,E_{C,A}^{\vee})$, the sequence defining $E_{C,A}$ gives
    $$
    \dim (h^i_{H,l})^{-1}(C,A) = \hom (E_{C,A},E_{C,A})-1 = h^0(X,\O_X(H')).
    $$
    The formula now follows from Lemma \ref{lemma:dimensionGNL}.
\end{proof}
From now on we will impose the condition $-K_X|_C>0$ for the general curve $C\in|\mathcal{L}|$. Note that if this condition is satisfied for some smooth curve $C$, then already $H^1(X,\mathcal{L})^\vee \simeq H^0(C,\omega_X|_C) =0$. Thus one can safely assume $H^1(X,\mathcal{L}) = 0$ from here on. Moreover, in this case $X$ is automatically rational whenever $p_a(C) \geq 1$. Indeed, $p_a(C)\geq 1$ and $-K_X|_C > 0$ imply $C^2>0$, so that $C$ is nef. As $2K_X.C <0$, the line bundle $\omega_X^2$ cannot be effective. Hence, $H^0(X,\omega_X^2) = 0 = H^1(X,\O_X)$. Hence, $X$ is rational by Castelnuovo's theorem.

\smallskip
\noindent The last input needed to establish the linear growth condition (\ref{eqn:lineargrowthcondition}) for a general curve in its linear system is the following inequality for the Clifford index.

\begin{lemma}\label{lemma:goanlityinequalityNEW}
    Let $X$ be a regular surface with $H^0(X,\omega_X) = 0$ and $\mathcal{L}$ a line bundle on $X$ such that the general curve $C \in |\mathcal{L}|$ is smooth, has Clifford index $c$ and satisfies $-K_X|_C>0$. Let $A$ be a complete basepoint-free pencil on $C$ whose Lazarsfeld--Mukai bundle $E = E_{C,A}$ satisfies $\Hom(E, i_\ast(\omega_C \otimes A^\vee)) \neq 0$ and is a (possibly trivial) Donagi--Morrison extension
    \begin{equation*}
        0 \rightarrow \O_X(H+H') \rightarrow E_{C,A} \rightarrow \mathcal{I}_Z(H) \rightarrow 0,
    \end{equation*}
    where $H > 0$, $H' \geq 0$, $[C] = 2H+H'$, and $Z \subseteq X$ is a zero-dimensional lci subscheme or empty. Then,
    $$
    H.(H+H')+H.K_X \geq c+2h^1(X,\O_X(H)).
    $$
\end{lemma}

\begin{proof}
    As $A$ is a complete pencil, we have $h^1(X,E^{\vee}) = 0$. Applying $\Hom(-,\O_X)$ to the Donagi--Morrison sequence gives $H^1(X,\mathcal{I}_Z(H+K_X))^\vee \simeq \Ext^1(\mathcal{I}_Z(H),\O_X) = 0$. Consequently, evaluation $\O_X(H+K_X) \rightarrow \O_Z$ is surjective on global sections and
    $$
    l=h^0(X,\O_Z)\leq h^0(\O_X(K_X+H))\leq h^1(C,\O_C(H)).
    $$
    Thus, if $h^1(C,\O_C(H)) \geq 2$, for example when $l \geq 2$, then the line bundle $\O_C(H)$ contributes to the Clifford index of $C$. We obtain the inequality
    $$
    c\leq \mathrm{Cliff}(H|_C) = H.C -2h^0(C,H|_C)+2.
    $$
    Using $h^0(X,\O_X(H))\leq h^0(C,\O_X(H)|_C)$, we can apply Riemann--Roch to $\O_X(H)$ and find
    $$
    H.(H+H')+H.K_X \geq c +2h^1(X,\O_X(H)),
    $$
    as desired.
    If $h^1(C,\O_C(H)) \leq 1$, we choose a non-zero map $$h \in \Hom(E, i_\ast(\omega_C \otimes A^\vee)) = \Hom(E, i_\ast(A^{\vee}\otimes\O_C(C)\otimes\omega_X|_C))$$ and a nonzero section $s \in H^0(C,\omega_X^{\vee}|_C)$. Denote by $V(s)$ its zero divisor. We obtain the composition $s \circ h \colon E \to A^{\vee} \otimes \O_C(C)$. The identification $\Hom(E,A^{\vee} \otimes \O_C(C)) \simeq \Hom(E,E)$ from Lemma~\ref{obstructionLemma} yields an endomorphism $E \to E$ that is not surjective because it vanishes where $s$ does. If the Donagi--Morrsion sequence does not split, this endomorphism necessarily factors as
    $$
    E\to \mathcal{I}_Z(H)\hookrightarrow\O_X(H)\xrightarrow{\cdot t}\O_X(H+H')\hookrightarrow E,
    $$
    for some nonzero section $t\in H^0(X,\O_X(H'))$. Denote its zero divisor by $V(t)$. By Proposition~\ref{proposition:surfaceinducedpencils}, the pencil $A$ on $C$ is surface-induced, i.e., $A = H|_C-B$ for some effective divisor $B$ on $C$.
    As $s\circ h$ factors through $\mathcal{I}_Z(H)$ and multiplication with $t$, as in the proof of Proposition \ref{prop:smoothnessWithoutSimplicity}, we obtain 
    \begin{equation}
    \label{eqn:intermediateeqn1}
    V(s)\subset (Z\cap C)\cup V(t).
    \end{equation}
    Let $D \coloneqq Z \cap C$ as a divisor on $C$. If $D = \emptyset$, for example when $l = 0$, we obtain 
    $$
    A\leq (H+H'+K_X)|_C
    $$
    and $H|_C$ again contributes to $\mathrm{Cliff}(C)$. In the remaining case we have $l = 1$, $h^1(C,\O_C(H)) = 1$ and $D \neq \emptyset$. Here we consider $\O_C(H-D)$. We have the commutative diagram
\[\begin{tikzcd}
	0 & {\mathcal{I}_Z(-H-H')} & {\mathcal{I}_Z(H)} & {\mathcal{I}_Z(H)|_C} & 0 \\
	0 & {\mathcal{O}_X(-H-H')} & {\mathcal{O}_X(H)} & {\mathcal{O}_C(H)} & 0
	\arrow[from=1-1, to=1-2]
	\arrow[from=1-2, to=1-3]
	\arrow[from=1-2, to=2-2]
	\arrow[from=1-3, to=1-4]
	\arrow[hook, from=1-3, to=2-3]
	\arrow[from=1-4, to=1-5]
	\arrow["\alpha", from=1-4, to=2-4]
	\arrow[from=2-1, to=2-2]
	\arrow[from=2-2, to=2-3]
	\arrow[from=2-3, to=2-4]
	\arrow[from=2-4, to=2-5]
\end{tikzcd}\]
    with $\mathrm{im}(\alpha)\simeq \O_C(H-D)$. As $h^0(X,\mathcal{I}_Z(H))\geq 2$, also $h^0(C,\O_C(H-D))\geq 2$. Now, by (\ref{eqn:intermediateeqn1}) we have 
    $$
    A\leq (H+H'+K_X+D)|_C
    $$
    and thus $h^1(C,\O_C(H-D))\geq 2$. It follows that $\O_C(H-D)$ contributes to $\mathrm{Cliff}(C)$. The sequence
    $$
    0\to \O_C(H-D)\to \O_C(H)\to \C_D\to 0,
    $$
    together with $h^1(C,\O_C(H)) = 1$, shows $h^0(C,\O_C(H-D)) = h^0(C,\O_C(H))$. The Clifford inequality for $\O_C(H-D)$ now again gives the desired bound.

    Finally, if the Donagi--Morrison sequence splits, then $Z = \emptyset$ since $E$ is a vector bundle, and we claim that $\O_C(H)$ again contributes to the Clifford index. Indeed, a nonzero map $E \simeq \O_X(H) \oplus \O_X(H+H') \rightarrow i_\ast(\omega_C \otimes A^\vee)$ either restricts to a nonzero map $\O_C(H) \rightarrow \omega_C \otimes A^\vee$ or else to a nonzero map $\O_C(H+H') \rightarrow \omega_C \otimes A^\vee$ which may be precomposed with an inclusion $\O_C(H) \subseteq \O_C(H+H')$. In both cases, we obtain an injection $\O_C(H) \rightarrow \omega_C \otimes A^\vee$ implying $A \leq K_C - H|_C$, hence $h^1(C,\O_C(H)) \geq 2$.
\end{proof}

\begin{theorem}
\label{thm:genericgreen}
    Let $X$ be a regular surface with $H^0(X,\omega_X) = 0$. Let $\mathcal{L}$ be an effective line bundle on $X$ such that $|\mathcal{L}|$ contains a smooth integral curve and the general $C\in |\mathcal{L}|$ satisfies $-K_X|_C>0$. Assume that at least one of the following holds:
    \begin{enumerate}[(i)]
    \item For the general $C \in |\mathcal{L}|_{\mathrm{sm}}$, we have $h^0(\omega_X^\vee|_C) \geq 2$.
    \item For all decompositions $[C]=2H+H'$ in $\Pic(X)$ with $H > 0$, $H' \geq 0$, we have $K_X.H'\leq 0$ or $H^1(X,\O_X(H+H')) =0$ (or both).
    \end{enumerate}
     Then Green's conjecture holds for the general curve in $|\mathcal{L}|$.
\end{theorem}

\begin{proof}
    Let $\mathcal{W}\subset \mathcal{W}^1_d(|\mathcal{L}|)$ be an irreducible component for which a general $(C,A) \in \mathcal{W}$ has a non-simple associated Lazarsfeld--Mukai bundle $E=E_{C,A}$. Assume first that $E$ is a (possibly split) Donagi--Morrison extension as in Lemma~\ref{lemma:goanlityinequalityNEW}. It is clear that $\mathcal{W}$ is either dominated by $\mathcal{G}^0_{H,l}$ or $\mathcal{G}^1_{H,l}$ for some $H$ and $l$. Assume $\mathcal{W}$ is dominated by $h^i_{H,l}$, then we can apply Lemmas \ref{lem:dimensionofW} and \ref{lemma:goanlityinequalityNEW} and use the formulas $\dim |\mathcal{L}|= g-1-\mathcal{L}.K_X$ and $l=d-H .(H+H')$  to find
    \begin{equation}
    \label{eqn:keydimensionestimate}
    \dim \mathcal{W} \leq \dim |\mathcal{L}|+d-c-2-2h^1(X,\O_X(H))+2i-\gamma_i.
    \end{equation}
     If $i=0$, this establishes the linear growth bound \eqref{eqn:lineargrowthcondition}.

    \smallskip
    \noindent
    Let now $i=1$, i.e., $\mathcal{W}$ is dominated by $h^1_{H,l}$. If $h^1(X,\O_X(H))\neq 0$, then \eqref{eqn:keydimensionestimate} again establishes linear growth, so we may assume $h^1(X,\O_X(H)) = 0$. If $H^1(X,\O_X(H+H')) \neq 0$, our assumptions imply $K_X.H' \leq 0$. By definition $(E,V)\in \mathcal{G}^1_{H,l}$ satisfies $h^1(X,E) = h^0(C,A\otimes\omega_X|_C) = 1$. Multiplication by a nonzero section $s\in H^0(C,A\otimes\omega_X|_C)$ yields an injection $H^0(C,\omega_X^{\vee}|_C)\hookrightarrow H^0(C,A)$. As $A$ is a complete and basepoint-free pencil, $h^0(C,\omega_X^{\vee}|_C)\geq 2$ implies $A\simeq \omega_X^{\vee}|_C$ but as this constitutes a single point in $W^1_d(C)$, it cannot violate the linear growth condition. Thus we can assume $h^0(C,\omega_X^{\vee}|_C) = 1$. Now both the linear systems $|\O_C(A+K_X|_C)|$ and $|\O_C(-K_X|_C)|$ are rigid. Denote by $\Delta\in |\O_C(A+K_X|_C)|$ the unique effective divisor. By Proposition \ref{proposition:surfaceinducedpencils}, $A$ is surface-induced, i.e., $A= H|_C-B$ for some effective divisor $B$. We find $\Delta \in |\O_C((H+K_X)|_C-B)|$ and $\Delta+B\in |\O_C((H+K_X)|_C)|$. Now if $h^0(C,\O_C((H+K_X)|_C)) =1$, there are only finitely many decompositions of its unique divisor into $\Delta+B$ and hence only finitely many $A=H|_C-B$. So we may assume $h^0(C,\O_C((H+K_X)|_C)) \geq 2$. Adjunction and Serre duality imply $h^1(C,\O_X((H+H')|_C)) = h^0(C,\O_X((H+K_X)|_C))\geq 2$. As $A$ is a pencil and $A\leq H|_C \leq (H+H')|_C$, we have $h^0(C,\O_C((H+H')|_C))\geq 2$ and so $\O_C((H+H')|_C)$ contributes to $\mathrm{Cliff}(C)$. As in the proof of Lemma~\ref{lemma:goanlityinequalityNEW}, we obtain the inequality
    \begin{equation}
    \label{eqn:CliffordineqH+H'}
    c+2h^1(X,\O_X(H+H'))\leq H.(H+H')+K_X.H+K_X.H'\leq H.(H+H')+K_X.H.
    \end{equation}
    The second inequality follows from $K_X.H'\leq 0$. Now, if $h^1(\O_X(H+H')) = 1$, this gives $c+2\leq H.(H+H')+K_X.H$. The dimension calculation in Lemma~\ref{lem:dimensionofW} then shows linear growth regardless of the value of $\gamma_1$. This in particular finishes the case where the Donagi--Morrison sequence splits because here $1 = h^1(E) = h^1(\O_X(H)) + h^1(\O_X(H+H')) = h^1(\O_X(H+H'))$. In the remaining cases, the Donagi--Morrison extension is not split and $h^1(X,\O_X(H+H')) = 0$. Here, the sequence
    \begin{equation}
    \label{eqn:DMExtension}
    0\to \O_X(H+H')\to E\to \mathcal{I}_Z(H)\to 0 
    \end{equation}
    shows $h^1(X,\mathcal{I}_Z(H)) = 1$, so $Z$ imposes $l-1$ independent conditions on $H^0(X,\O_X(H))$. If we show that 
\[
\mathcal V :=
\left\{
  Z \in X^{[l]} \;\middle|\;
  \begin{array}{@{}l@{}}
    Z \text{ imposes } (l-1)
      \text{ independent conditions on } |\O_X(H)| \text{ and}\\
    (\O_X(H), H^0(\mathcal{I}_Z(H)))
      \text{ has at most finitely many basepoints}
  \end{array}
\right\} \subseteq X^{[l]}
\]
    has codimension at least $2$, then also $\gamma_1 \geq 2$, so that (\ref{eqn:keydimensionestimate}) establishes linear growth. To prove this, we first observe that, since $h^0(X,\O_X(H))\geq 2$, there exists a non-reduced, length~$2$ subscheme $Z'\subset X$ imposing $2$ independent conditions on $\O_X(H)$ (this uses characteristic zero). Of course, the same is true for the general non-reduced length~$2$ subscheme. Now, the general non-reduced member $Z\in \mathcal{V}$ is of the form $Z'\cup \{p_1,\dots,p_{l-2}\}$ for some $Z'$ with the above property and $p_i$ general points in $X$ (this uses that $X$ is a smooth \emph{surface}). The closure of the subschemes of this form is an irreducible divisor in $X^{[l]}$, and it is clear the the general subscheme of this form imposes $l$ independent conditions, so the set of non-reduced subschemes in $\mathcal{V}$ that impose less conditions has dimension at most $2l-2$. On the other hand, the general \emph{reduced} member of $\mathcal{V}$ consists of the choice of $(l-1)$ distinct points $Z'=\{q_1,\ldots, q_{l-1}\}$ imposing independent conditions and an additional disjoint point $q_l$ not imposing an additional condition on $(\O_X(H), H^0(\mathcal{I}_{Z'}(H)))$. Since $\mathcal{I}_Z(H) \subseteq \mathcal{I}_{Z'}(H)$, the linear system $(\O_X(H), H^0(\mathcal{I}_{Z'}(H)))$ has at most finitely many basepoints, so the choice of $q_l$ contributes zero moduli. Hence, $\dim \mathcal{V}$ is bounded above by $2l-2$ and hence $\gamma_1 \geq 2$.

    \smallskip
    \noindent
    The last remaining case is that of $E=\O_X(D_1)\oplus \O_X(D_2)$ for two effective divisors $D_1$, $D_2$ with $D_1 + D_2 = [C]$ whose difference $D_1 - D_2$ is neither effective nor anti-effective. Assume $i=h^1(X,E)= 0$, then the obstruction space to $(E,H^0(C,A)^{\vee})$ is identified with $\Hom(E,E\otimes \omega_X)$. As $\Hom(\O_X(D_j),\O_X(D_j+K_X)) = 0$, we can relabel $D_1$ and $D_2$ and assume without loss of generality that an obstruction implies $H^0(X,\O_X(D_1-D_2+K_X))\neq 0$. Similar to before, by integrality of $C$, the projection of $H^0(C,A)^{\vee}\subset H^0(X,E)$ to $H^0(X,\O_X(D_2))$ has dimension two. It follows $h^0(C,D_2|_C)\geq 2$. Adjunction and $D_1-D_2+K_X>0$ give $K_C-D_2|_C\geq D_2|_C$ and we find that $\O_X(D_2)$ contributes to $\mathrm{Cliff}(C)$. The dimension estimate becomes 
    $$
    \dim \mathcal{W}\leq \dim |\mathcal{L}|+d-2-D_1.D_2-D_2.K_X
    $$
    and the Clifford index inequality reads $c\leq D_1.D_2+D_2.K_X$ and we find again $\dim \mathcal{W}\leq \dim |\mathcal{L}|+d-c-2$. Let now $i=1$. Since $D_1-D_2$ is neither effective nor anti-effective, $\mathrm{end}(E) = 2$. If $\mathcal{W}$ dominates $|\mathcal{L}|$, then Proposition~\ref{prop:AproduFarkas} bounds the dimension of a general fiber by $\rho(g,1,d) + 1 = 2d-1-g$ which is $\leq d-c-2$ whenever $d \leq g-c-1$. So the linear growth bound \eqref{eqn:lineargrowthcondition} is satisfied except possibly for $d = g-c$. In that case, by integrality of $C$, the projections of $H^0(C,A)^{\vee}\subset H^0(X,E)$ to $H^0(X,\O_X(D_j))$ are two-dimensional with zero-dimensional base locus. Also note that the image of $H^0(C,A)^{\vee}\otimes \O_X\to E$ is exactly $A$ when restricted to $C$, i.e., we have $A\hookrightarrow E|_C$. It follows $h^0(C,D_j|_C)\geq 2$ and $A\leq D_j|_C$ for $j\in \{0,1\}$.
    In general, the obstruction space to $(E,H^0(C,A)^{\vee})$ is given by $\Hom(E,i_*(A^{\vee}\otimes \omega_C))$. Obstructedness thus implies, after relabelling, $K_C-A-D_1|_C\geq 0$. In total $A\leq D_1|_C$ and $A\leq K_C-D_1|_C$, thus $\O_C(D_1|_C)$ contributes to $\mathrm{Cliff}(C)$. We argue that also $\O_C(D_2|_C)$ contributes to $\mathrm{Cliff}(C)$. We can again assume $h^0(C,\omega_X^{\vee}|_C) =1$ as otherwise already $A\simeq \omega_X^{\vee}|_C$. Analogously to the non-split case, we have $A \simeq D_i|_C(-B_i)$, where $B_i$ is the base locus of the projection of $H^0(C,A)^{\vee}\subset H^0(X,E)$ to $H^0(X,\O_X(D_i))$. If the linear system $|\O_C(K_C-D_2|_C)|$ is a point, there are only finitely many decompositions of its unique divisor and in total only finitely many pencils $A$. Thus, for the general $(C,A)$ we can assume $h^0(C,\O_C(K_C-D_2|_C))\geq 2$ and also $\O_C(D_2|_C)$ contributes to $\mathrm{Cliff}(C)$. Using $d=c_2(E) =D_1.D_2$, we obtain the two inequalities
    \begin{align*}
        c+2h^1(X,\O_X(D_1))&\leq d+K_X.D_1\\
        c+2h^1(X,\O_X(D_2))&\leq d+K_X.D_2.
    \end{align*}
    Adding the inequalities and using $h^1(X,\O_X(D_1))+h^1(X,\O_X(D_2)) = h^1(X,E) = 1$, we find
    \begin{equation}
        \label{eqn:IncomparableIneq}
        2c+2\leq 2d+K_X.C.
    \end{equation}
    For split vector bundles $E$, the parameter count simplifies a lot. Indeed the parameter space of coherent systems has exactly $2h^0(X,E)-4-\mathrm{end}(E)+1$ moduli. In our case, $\mathrm{end}(E) = 2$, so
    $$
    \dim \mathcal{W}\leq 2h^0(X,E)-5.
    $$
    Riemann--Roch yields $\chi(E) = h^0(X,E)-1=\dim |\mathcal{L}|-d+2$ and $\dim |\mathcal{L}| = g-1 - C.K_X$, so
    $$
    \dim \mathcal{W} - \dim |\mathcal{L}| \leq g - (2d + K_X.C) \leq g - 2c - 2,
    $$
    where the last inequality uses \eqref{eqn:IncomparableIneq}. This is the linear growth condition \eqref{eqn:lineargrowthcondition} for $d = g-c$, as desired.
\end{proof}

\begin{corollary}\label{cor:delPezzoBlowup}
    Let $Y$ be a del Pezzo surface and $X \to Y$ the blowup of $Y$ in $K_Y^2$ general points. Let $\mathcal{L}$ be any line bundle on $X$ containing a smooth curve $C \in |\mathcal{L}|$ which is not a fixed component of $|-K_X|$ and satisfies $C.K_X \neq 0$. Then Green's conjecture holds for the general smooth curve in $|\mathcal{L}|$.
\end{corollary}

\begin{proof}
    The proof follows from the observation that $-K_X$ is nef and $H^0(X,\omega_X^{\vee}) = 1$, i.e., $K_X.H'\leq 0$ for all effective divisors $H'$. Moreover, the short exact sequence
    \begin{equation*}
        0 \rightarrow \mathcal{L}^\vee \otimes \omega_X^\vee \rightarrow \omega_X^\vee \rightarrow \omega_X^\vee|_C \rightarrow 0
    \end{equation*}
    shows that $-K_X|_C > 0$ if the first map is not an isomorphism on global sections. But this is the case if and only if $C$ is a fixed component of $|-K_X|$. In particular, the assumptions of Theorem~\ref{thm:genericgreen} are satisfied.
\end{proof}

\begin{remark}
    Under suitable hypotheses, standard arguments in Koszul cohomology allow one to extend Green’s conjecture from the general curve in a linear system to every smooth member. For instance, such an argument is given in \cite{Lelli-ChiesaGreenConj} under the assumption that $-K_X$ contains a pencil. It is less clear how these arguments extend to the case $h^0(C,\omega_X^\vee|_C)=1$.
\end{remark}

\stoptoc 
\resumetoc 

\bibliographystyle{amsalpha} 
\bibliography{TEMPLATEBIB}
\end{document}